\documentclass[a4paper,fleqn]{cas-dc}
\usepackage{moreverb}
\usepackage{amsthm}
\usepackage{amssymb}

 {\usepackage{color}}{}
\newtheorem{proposition}{Proposition}[section]
\newtheorem{theorem}{Theorem}[section]

\newtheorem{corollary}{Corollary}[section]
\newtheorem{definition}{Definition}[section]
\usepackage{subcaption}
\newcommand\BibTeX{{\rmfamily B\kern-.05em \textsc{i\kern-.025em b}\kern-.08em
T\kern-.1667em\lower.7ex\hbox{E}\kern-.125emX}}
\usepackage{algorithm}
\usepackage[algo2e]{algorithm2e}
\usepackage{adjustbox}
\received{<day> <Month>, <year>}
\revised{<day> <Month>, <year>}
\accepted{<day> <Month>, <year>}

\DeclareMathOperator*{\argmin}{arg\,min}

\newcommand{\revise}[1]{\textcolor{black}{#1}}

\usepackage[authoryear,longnamesfirst]{natbib}
\usepackage{rotating}
\usepackage{natbib}
\def\tsc#1{\csdef{#1}{\textsc{\lowercase{#1}}\xspace}}
\tsc{WGM}
\tsc{QE}

\usepackage{hyperref}

\begin{document}
\let\WriteBookmarks\relax
\def\floatpagepagefraction{1}
\def\textpagefraction{.001}
        
\shorttitle{A Framework for Controllable PFL with Completed Scalarization Functions and its Applications}    

\shortauthors{Tran Anh Tuan et al.}  

\title [mode = title]{A Framework for Controllable Pareto Front Learning with Completed Scalarization Functions and its Applications}  
\author[1]{Tran Anh Tuan}[type=editor,
orcid=0000-0001-6287-0173]
\cormark[1]
\ead{tuan.ta181295@sis.hust.edu.vn}
\affiliation[1]{organization={School of Applied Mathematics and Informatics, Hanoi University of Science and Technology},
            city={Ha Noi},
            country={Viet Nam}}

\affiliation[2]{organization={College of Engineering and Computer Science, VinUniversity},
            city={Ha Noi},
            country={Viet Nam}}
\author[2]{Long P. Hoang}[]
\cormark[1]
\ead{long.hp@vinuni.edu.vn}
\author[2]{Dung D. Le}[]
\ead{dung.ld@vinuni.edu.vn}
\author[1]{Tran Ngoc Thang}[]
\cormark[2]
\ead{thang.tranngoc@hust.edu.vn}
\cortext[1]{Equal Contribution and Co-first Author}
\cortext[2]{Corresponding Author}
\begin{abstract}
Pareto Front Learning (PFL) was recently introduced as an efficient method for approximating the entire Pareto front, the set of all optimal solutions to a Multi-Objective Optimization (MOO) problem. In the previous work, the mapping between a preference vector and a Pareto optimal solution is still ambiguous, rendering its results. This study demonstrates the convergence and completion aspects of solving MOO with pseudoconvex scalarization functions and combines them into Hypernetwork in order to offer a comprehensive framework for PFL, called Controllable Pareto Front Learning. Extensive experiments demonstrate that our approach is highly accurate and significantly less computationally expensive than \revise{prior methods in term of inference time.} 
\end{abstract}

\begin{graphicalabstract}
\centering
\includegraphics[width=17.5cm]{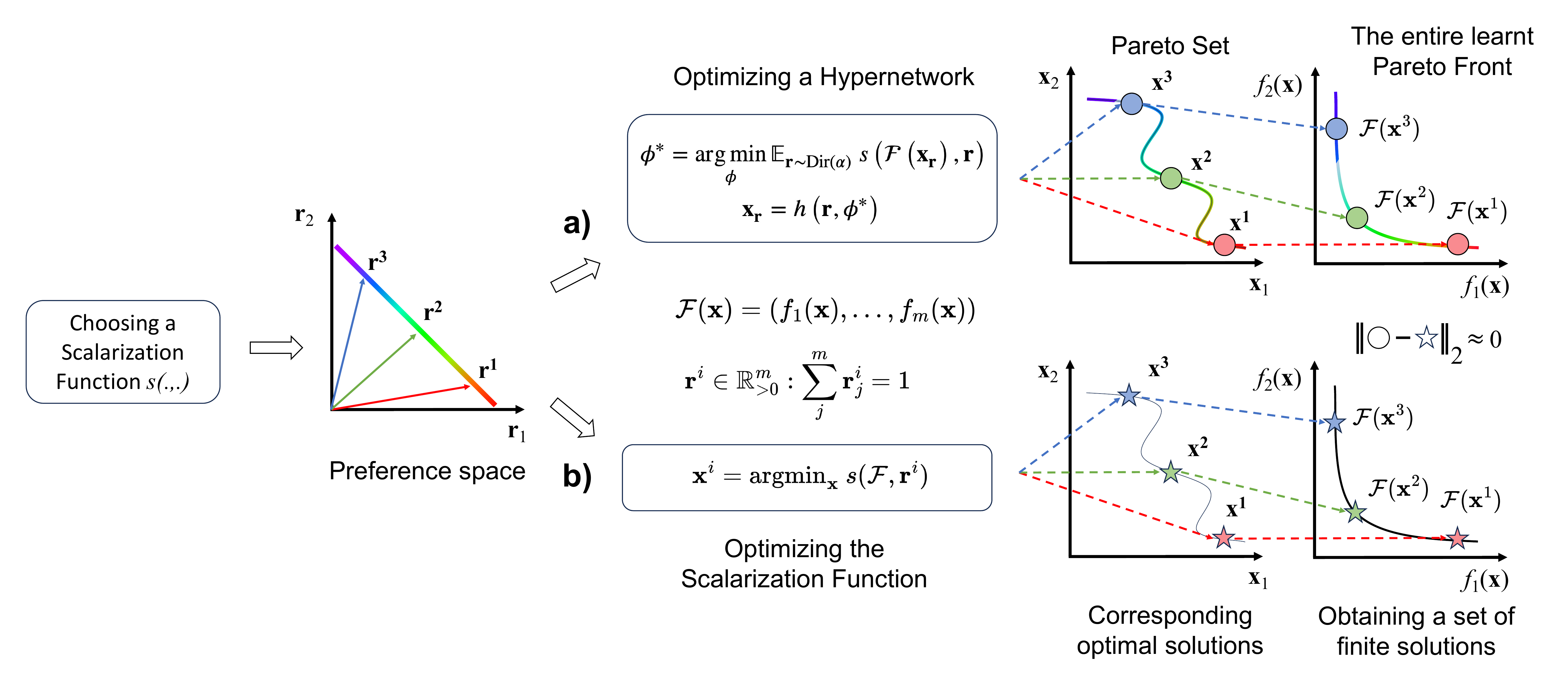}
\textbf{Graphical Abstract (Demo for Two-Objective Optimization):} Controllable Pareto Front Learning (PFL) achieves the approximation of the entire Pareto Front by employing Completed Scalarization functions (Sub-figure (a)). Instead of directly optimizing Completed Scalarization functions to obtain a finite set of solutions (Sub-figure (b)), Controllable PFL aims at learning a mapping between a preference vector $r_i$ and its corresponding optimal solution $x_i^{*}$.
\end{graphicalabstract}

\begin{highlights}
\item Controllable Pareto Front Learning is based on Completed Scalarization Functions optimization problem
\item A theoretical basis for how the Hypernetwork framework may give a precise mapping between a preference vector and the corresponding Pareto optimal solution
\item The framework makes a computational cost which is significantly less than prior methods in large scale Multi Task Learning
\end{highlights}


\maketitle

\section{Introduction}
\begin{figure*}[!htb]
    \centering
    \includegraphics[width=1.\textwidth]{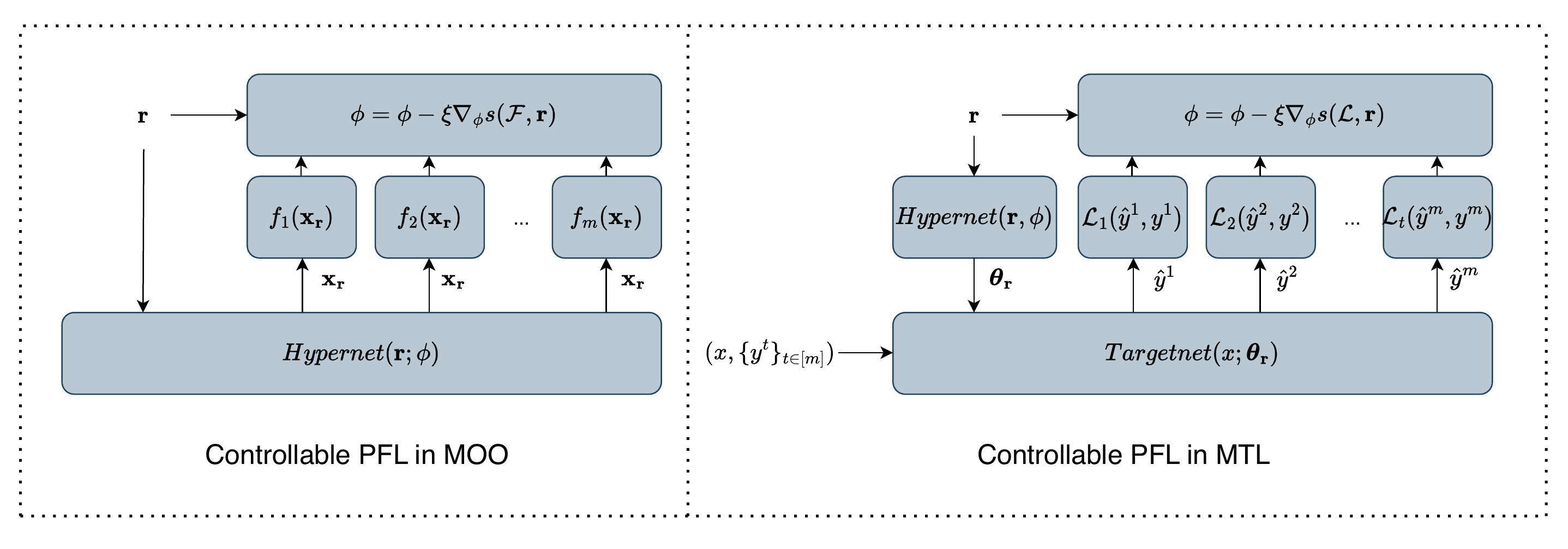}
    \caption{\revise{Controllable Pareto Front Learning framework with Completed Scalarization Functions by reference vector $r$ in MOO and MTL settings. In MOO setting, Controllable PFL uses a Hypernetwork to generate corresponding value for $x_{\mathbf{r}}$ according to a preference vector $\mathbf{r}$. On the other hand, in MTL setting, Controllable PFL use a Hypernetwork to generate the weight $\theta_{\mathbf{r}}$ for a Target Network, which are Deep Learning models. The only trainable parameters of our framework are the parameters $\phi$ of the Hypernetwork.}}
\end{figure*}
Multi-objective optimization, which considers simultaneously related objectives by sharing optimized variables, plays a crucial role in many domains such as chemistry \citep{cao2019multi}, biology \citep{lambrinidis2021multi}, \revise{optimal power flow \citep{dilip2018optimal, premkumar2021many}}, and especially Multi-Task Learning (MTL) \citep{sener2018multi}, an important area of Deep Learning. Due to the trade-off between conflicting objectives, it is impossible to discover the optimal solution for each objective separately; hence, addressing a MOO problem often involves finding a portion or all of the Pareto or weakly Pareto set.  Genetic algorithms \citep{murugan2009nsga} and \revise{evolutionary algorithms \citep{jangir2021elitist}} are plausible options for providing a good approximation of the Pareto front, however, this approach often suffers when dealing with large-scale problems. Another approach, Multi-gradient descent \citep{desideri2012multiple}, one of the steepest methods to obtain a Pareto stationary solution, identifies a common descent direction for objectives at each iteration. However, while using these two approaches, the users are unable to control the optimal solutions by specifying their preferred objectives.

Therefore, optimization in the Pareto set with an extra criterion function is a useful strategy for maximizing user satisfaction. \revise{\citep{thang2015outcome, thang2016solving, thang2016outcome, vuong2023optimizing} proposed approaches for optimization over the efficient solution set, typically focusing on algorithms to find an optimal solution to a single objective and apply those hypotheses in real applications such as portfolio selection.} Most modern MOO algorithms \citep{lin2019pareto, mahapatra2021exact, pmlr-v162-momma22a}  must choose the preference vector as the extra criterion function in advance. Unfortunately, these approaches restrict flexibility because, in real-time, the decision-maker cannot change their priorities freely because the corresponding solutions are not always readily available and need to be optimized from scratch. As a result, \citep{lin2020controllable, navon2020learning, hoang2022improving} have recently investigated and created a research methodology known as Pareto Front Learning (PFL), which tries to approximate the whole Pareto front using Hypernetwork \citep{hypernetwork, chauhan2023brief}. 

Hypernetwork generate parameters for other networks (target networks). One of the first works of PFL, Controllable Pareto MTL \citep{lin2020controllable} is ineffective since it does not provide a mapping between a preference vector and a corresponding efficient solution, i.e. surjection. Other pioneers of PFL, \citep{navon2020learning}, proposed two Pareto Hypernetwork, PHN-LS and PHN-EPO, based on Linear Scalarization and EPO solver \citep{mahapatra2021exact}. As a development of the previous work, Multi-Sample Hypernetwork \citep{hoang2022improving}, the current state-of-the-art PFL, utilizes the dynamic relationship between solutions by maximizing hypervolume in combination with a penalty function. However, no convergence proofs were provided. Building upon the previous work, our study develops a novel framework named Controllable Pareto Front Learning, a complete version of PFL with mathematical explanations. Our key contributions include the following:
\begin{itemize}
   \item Firstly, we establish a theoretical basis for how the proposed framework may give a precise mapping between a preference vector and the corresponding Pareto optimal solution.
   \item Secondly, we presented a Framework for Pareto Front Learning with Completed Scalarization Functions which is built on the previous theoretical basis. 
   \item Thirdly,  we applied our framework to a wide range of problems from MOO to large-scale Multi-Task Learning. The computational results show that the mapping approximated the entire Pareto Front is sufficiently precise, and the computational cost is significantly less than \revise{training prior methods on multiple preference vectors}.

\end{itemize}

Our paper is structured as follows: Section 2 provides background knowledge for MOO. Section 3 presents a solving method and convergence proof for the vector optimization problem. Section 4 describes Controllable Pareto Front Learning, which approximates the entire Pareto front and maps a preference vector to a solution on the Pareto front. Section 5 reviews the completed scalarization functions that generate Pareto optimal solutions. In section 6, we apply the proposed techniques to Multi-task learning, and in Section 7, we conduct experiments on MOO and Multi-task learning problems. The final section discusses some conclusions and future works.
\section{Preliminaries}
MOO aims to find $\mathbf{x}^* \in X$ to optimize $m$ objective functions:
\begin{align*}\tag{MOP}\label{MOP}
    \mathbf{x}^* = \underset{\mathbf{x}\in X}{\text{arg\,min}} \mathcal{F}(\mathbf{x}),
\end{align*}
where $\mathcal{F}(\cdot): X\to Y\subset\mathbb{R}^m, \mathcal{F}(\mathbf{x}) = \{f_1\left(\mathbf{x}\right),\dots, f_m\left(\mathbf{x}\right) \}$, $X\subset\mathbb{R}^{n}$ is nonempty convex set, and objective functions $f_i(\cdot): \mathbb{R}^n \rightarrow \mathbb{R}$, $i= 1,\dots,m$ are convex functions on $X$.

\theoremstyle{definition}
\begin{definition}[Dominance] A solution $\mathbf{x}^a$ dominates $\mathbf{x}^b$ if $f_j\left(\mathbf{x}^a\right) \leq f_j\left(\mathbf{x}^b\right), \forall j$ and $f_j\left(\mathbf{x}^a\right) \neq f_j\left(\mathbf{x}^b\right)$. Denote $\mathcal{F}\left(\mathbf{x}^a\right) \prec \mathcal{F}\left(\mathbf{x}^b\right)$.
\end{definition}

\begin{definition}[Pareto optimal solution] A solution $\mathbf{x}^a$ is called Pareto optimal solution (efficient solution) if $\nexists \mathbf{x}^b: \mathcal{F}\left(\mathbf{x}^b\right) \preceq \mathcal{F}\left(\mathbf{x}^a\right)$.
\end{definition}

\begin{definition}[Weakly Pareto optimal solution] A solution $\mathbf{x}^a$ is called weakly Pareto optimal solution (weakly efficient solution) if $\nexists \mathbf{x}^b: \mathcal{F}\left(\mathbf{x}^b\right) \prec \mathcal{F}\left(\mathbf{x}^a\right)$.
\end{definition}

\begin{definition}[Pareto stationary] A point $\mathbf{x}^{*}$ is called Pareto stationary (Pareto critical point) if $\nexists d \in X:\left\langle J \mathcal{F}\left(\mathbf{x}^{*}\right), d\right\rangle<0$ or $\forall d \in X:\left\langle J \mathcal{F}\left(\mathbf{x}^{*}\right), d\right\rangle \nless 0$, corresponding:
        \begin{align*}
            \max _{i=1, \ldots, m} \nabla f_{i}\left(\mathbf{x}^{*}\right)^{\top} d \geq 0, \quad \forall d \in X,
        \end{align*}
where $J \mathcal{F}\left(\mathbf{x}^{*}\right) = \left[\nabla f_{1}(\mathbf{x}^{*})^{T},\dots,\nabla f_{m}(\mathbf{x}^{*})^{T}\right]^T$ is Jacobian matrix of $\mathcal{F}$ at $\mathbf{x}^{*}$.
\label{def:paretostationary}
\end{definition}

\theoremstyle{definition}
\begin{definition}[Pareto front] The set of Pareto optimal solutions is called Pareto set, denoted by $X_E$, and the corresponding images in objectives space are Pareto front $PF_E=PF(X_E)$.
\end{definition}
\begin{proposition}
    $\mathbf{x}^*$ is Pareto optimal solution to Problem \eqref{MOP} $\Leftrightarrow \mathbf{x}^*$ is pareto stationary point.
\end{proposition}
\begin{definition}\citep{dinh2005generalized}
    A function $\varphi$ is specified on convex set $X\subset\mathbb{R}^n$, which is called:
\begin{enumerate}
    \item nondecreasing on $X$ if $\mathbf{x}\succeq \mathbf{y}$ then $\varphi(\mathbf{x})\ge \varphi(\mathbf{y}),$ $\forall \mathbf{x},\mathbf{y}\in X$.
    \item weakly increasing on $X$ if $\mathbf{x}\succ\mathbf{y}$ then $\varphi(\mathbf{x})\ge\varphi(\mathbf{y}),$ $\forall \mathbf{x},\mathbf{y}\in X$.
    \item monotonically increasing on $X$ if $\mathbf{x}\succ \mathbf{y}$ then $\varphi(\mathbf{x})>\varphi(\mathbf{y}),$ $\forall \mathbf{x},\mathbf{y}\in X$.
\end{enumerate}
\end{definition}

\begin{definition}\citep{mangasarian1994nonlinear}
	The function $f:\mathbb{R}^m \to \mathbb R$  is said to be 
	\begin{itemize}
		\item[$\bullet$] 		 convex on $ X $ if for all $ \mathbf{x},\mathbf{y}\in X $, $ \lambda \in [0,1] $, it holds that
		\begin{equation*}
			f(\lambda \mathbf{x} +(1-\lambda)\mathbf{y})\leq \lambda f(\mathbf{x})+(1-\lambda)f(\mathbf{y}).
		\end{equation*}
		\item[$\bullet$] 		pseudoconvex on $ X $ if for all $ \mathbf{x},\mathbf{y}\in X $, it holds that
		\begin{equation*}
			\langle \nabla f(\mathbf{x}), \mathbf{y}-\mathbf{x} \rangle \geq 0 \Rightarrow f(\mathbf{y})\geq f(\mathbf{x}).
		\end{equation*}
		\item[$\bullet$]  quasiconvex on $ X $ if for all $ \mathbf{x},\mathbf{y}\in X $, $ \lambda \in [0;1] $, it holds that
		\begin{equation*}
			f(\lambda \mathbf{x} +(1-\lambda)\mathbf{y})\leq \max \left\{f(\mathbf{x});f(\mathbf{y}) \right\}.
		\end{equation*}
	\end{itemize}
\end{definition}
    \begin{proposition}\citep{dennis1996numerical}
	The differentiable function $ f $ is quasiconvex on $ X $ if and only if 
	\[ 
	f(\mathbf{y})\leq f(\mathbf{x}) \Rightarrow 		\langle \nabla f(\mathbf{x}), \mathbf{y}-\mathbf{x} \rangle \leq 0.
	\] 
\end{proposition}
We see that "$f$ is convex" $\Rightarrow$ "$f$ is  pseudoconvex" $\Rightarrow$ "$f$ is  quasiconvex" \citep{mangasarian1994nonlinear}.

\section{Solving the Scalarization problem}
Consider the following Scalarization problem:
\begin{align*}\tag{CP}\label{CP}
&\min s\left(\mathcal{F}(\mathbf{x}),\mathbf{r}\right)\\
&\text{s.t. }\mathbf{x}\in X,
\end{align*}
where $\mathbf{r}\in\mathcal{P}=\left\{\mathbf{r} \in \mathbb{R}^m_{>0}: \sum_i r_i = 1\right\}$, $s(\cdot,\cdot):Y\times\mathcal{P}\to\mathbb{R}$ is some monotonically increasing and pseudoconvex function on $Y$, be a differentiable function on $X$. We assume that with each $\mathbf{r}$, then the solution of Problem \eqref{CP} is unique.

\begin{definition}\citep{Luc1989}\label{definition2.8}
    Given a family $\mathcal{S}$ of functions $s(\cdot,\cdot):Y\times\mathcal{P}\to\mathbb{R}$, we say that $s$ is a completed scalarization function if for every $\mathbf{x}\in X_E$, there exists $s\in\mathcal{S}$ such that $\mathbf{x}\in E$ and $E\subseteq X_E$, where $E$ denotes the optimal solution set to Problem \eqref{CP}.
\end{definition}

Problem \eqref{CP} covers a broad class of complex nonconvex
optimization problems including, such as convex multiplicative programming,
bi-linear programming and quadratic programming.

\begin{theorem}\label{theorem3.1}
     $s$ is a completed scalarization function, and all efficient solutions for Problem \eqref{MOP} can be found by Problem \eqref{CP}.
\end{theorem}
\begin{proof}
    By Proposition 5.21 \citep{dinh2005generalized}, then $s$ is an increasing function then, and every optimal solution of Problem \eqref{CP} is efficient for Problem \eqref{MOP}. It is obvious that $s(\mathcal{F}(\mathbf{x}),\mathbf{r})$ is a completed scalarization function according to Definition \ref{definition2.8}. Moreover, Problem \eqref{MOP} is a convex MOO problem, by Proposition 5.22 \citep{dinh2005generalized}, then all efficient solutions for Problem \eqref{MOP} can be found by Problem \eqref{CP}. It means that with each $\mathbf{r} (r_i>0, \sum_i r_i=1)$, then we will find an efficient solution of Problem \eqref{MOP} by solving Problem \eqref{CP}.
\end{proof}
\begin{proposition}
    $s$ is pseudoconvex function on $X$, and Problem \eqref{CP} is a pseudoconvex programming problem.
\end{proposition}
\begin{proof}
Since $-s(\mathcal{F}(\mathbf{x}),\mathbf{r})$ is pseudoconcave then semistrictly quasiconcave function over $Y$ \citep{mangasarian1994nonlinear}, we have $-s(\mathcal{F}(\mathbf{x}),\mathbf{r})$ is decreasing and semistrictly quasiconcave over $Y$. Since $\mathcal{F}(\mathbf{x})$ is convex, by Proposition 5.3 \citep{avriel2010generalized}, $-s(\mathcal{F}(\mathbf{x}),\mathbf{r})$ is semistrictly quasiconcave with respect to $\mathbf{x}$. Therefore, $s$ is semi-strictly quasiconvex. In the case when $s$ is smooth over $Y$, then $s$ is pseudoconvex on $X$.
\end{proof}
Problem \eqref{CP} is a pseudoconvex programming problem. Therefore, we utilize an algorithm of \citep{adaptiveThang} to solve Problem \eqref{CP} by the gradient descent method. Despite other methods, such as the neurodynamic approach, which uses recurrent neural network models to solve pseudoconvex programming problems with unbounded constraint sets \citep{bian2018neural, xu2020neurodynamic,  liu2022one}. It was restricted to the ability to scalable applications in machine learning.

First, we  recall some definitions and basic results that will be used in the next section. The interested reader is referred to \citep{bauschke2011convex} and \citep{rockafellar1970convex} for bibliographical references.  

For $x\in \mathbb{R}^m$, denote by $P_X (\mathbf{x})$ the projection of $\mathbf{x}$ onto $X$, i.e., 
$$ 
P_X(\mathbf{x}):=\argmin\left\{\|\mathbf{z}-\mathbf{x}\|:\mathbf{z}\in X\right\}.
 $$
\begin{proposition}\citep{bauschke2011convex}) It holds that
\begin{itemize}
\item[\textup{(i)}]  $\left \| P_X(\mathbf{x})-P_X(\mathbf{y}) \right \| \leq \left \|\mathbf{x}-\mathbf{y} \right \|$ for all $\mathbf{x},\mathbf{y}\in \mathbb{R}^m$,
\item[\textup{(ii)}] $\left\langle  \mathbf{y}-P_X(\mathbf{x}),\mathbf{x}-P_X(\mathbf{x})  \right\rangle \leq 0$ for all $\mathbf{x}\in \mathbb{R}^m$, $\mathbf{y}\in X$.
\end{itemize}
\end{proposition}

\begin{algorithm}[ht]
\KwIn{$\mathbf{x}^0 \in X$,  $\lambda \subset (0,2/L)$. Set $k=0$.}
\KwOut{$x^*$.}

\While{not converged}{

    $\mathbf{x}^{k+1}=P_X \left(\mathbf{x}^k-\lambda \nabla f \left(\mathbf{x}^k\right)\right)$
    
    \uIf{$\mathbf{x}^{k+1}=\mathbf{x}^{k}$}{
    STOP \;
  }
  \Else{
  $k :=k+1$ \;
  }    
}
$x^* = x^{k+1}$.
\caption{\label{algGD}: Algorithm GD for Problem \eqref{CP}.}
\end{algorithm}

Now, suppose that the algorithm generates an infinite sequence. Below theorem will prove that this sequence converges to a solution of the Problem \eqref{CP}.
	\begin{theorem}
		Assume that the sequence $ \left\{ \mathbf{x}^k\right\} $ is generated by Algorithm \ref{algGD}. Then,
		the sequence $ \left\{f\left(\mathbf{x}^k\right) \right\} $ is convergent and each limit point (if any) of the sequence $ \left\{ \mathbf{x}^k\right\} $ is a stationary point of the problem. Moreover,
		\begin{itemize}
			\item[$ \bullet $] if $ f $ is quasiconvex on $ X $, then the sequence $ \left\{\mathbf{x}^k \right\} $ converges to a stationary point of the problem.
			\item[$ \bullet $]  if $ f $ is pseudoconvex on $ X $, then the sequence $ \left\{\mathbf{x}^k \right\} $ converges to a solution of the problem.
		\end{itemize}
	\end{theorem}

Next, the convergence rate of Algorithm \ref{algGD} is estimated by solving  unconstrained optimization problems.
\begin{corollary}
	Assume that $f$ is convex, $ X=\mathbb R^m $ and $\left\{\mathbf{x}^k\right\}$ is the sequence generated by Algorithm \ref{algGD}. Then,
	$$ 
f(\mathbf{x}^k)-f(\mathbf{x}^*) =O \left( \frac 1 k\right) , 
 $$
where $\mathbf{x}^*$ is a solution of the problem.
	\end{corollary}

\section{Controllable Pareto Front Learning}\label{sec4}
In this section, we present a method to approximate the entire Pareto front and help map from a preference vector to an optimal Pareto solution respectively. Unlike \citep{lin2020controllable} didn't specify a mapping between a preference vector and a corresponding Pareto optimal solution, or \citep{hoang2022improving} navigate the generated solutions by cosine similarity function, we introduce Controllable Pareto Front Learning framework by transforming Problem \eqref{CP} into:
\begin{align*}\label{CS-PHN}\tag{PHN-CSF}
 &\phi^* = \argmin_{\phi} \mathbb{E}_{\mathbf{r} \sim \text{Dir}(\alpha)} \text{ }s\left(\mathcal{F}\left(\mathbf{x}_{\mathbf{r}}\right),\mathbf{r}\right)\\
    &\text{ s.t. } \mathbf{x}_{\mathbf{r}} = h\left(\mathbf{r}, \phi^*\right) \in X_E, h\left(\mathcal{P}, \phi^*\right) \equiv X_E
\end{align*}
where $h: \mathcal{P} \times \mathbb{R}^q \rightarrow \mathbb{R}^n$, $q$ is the number of parameters of hypernetwork, and Dir$(\alpha)$ is Dirichlet distribution. After the optimization process, we can achieve the corresponding efficient solution $\mathbf{x}_{\mathbf{r}} = h\left(\mathbf{r}, \phi^*\right)$ which is the solution of the Problem \eqref{CP}.

We use the below updating rule to solve problem \ref{CS-PHN}:
\begin{align*}
    \phi_{t+1} =\phi_t - \xi\nabla_{\phi}s\left(\phi_t,\mathbf{r}\right),
\end{align*}
where $\xi$ is the step size.

Let $\mathbf{x}\left(\hat{\mathbf{r}}\right)$ be a local optimal minimum to Problem \eqref{CP} corresponding to preference vector $\hat{\mathbf{r}}$, i.e. $\nabla s\left(\mathbf{x},\hat{\mathbf{r}}\right) = 0$. Then $\mathbf{x}\left(\hat{\mathbf{r}}\right) = h\left(\hat{\mathbf{r}},\phi\right)$ is a Pareto optimal solution to Problem \eqref{MOP}, and there exists a neighborhood $U$ of $\hat{\mathbf{r}}$ and a smooth mapping $\mathbf{x}(\mathbf{r})$ such that $\mathbf{x}\left(\mathbf{r}^*\right)_{\mathbf{r}^*\in U}$ is also a Pareto optimal solution to Problem \eqref{MOP}.

Indeed, by using universal approximation theorem \citep{cybenko1989approximation}, we can approximate smooth function $\mathbf{x}(\mathbf{r})$ by a network $h\left(\mathbf{r},\phi\right)$ with $\forall \mathbf{r}\in\mathcal{P} = \left\{\mathbf{r} \in \mathbb{R}^m_{>0}: \sum_i r_i = 1\right\}$, and with an $error\le\epsilon$ in approximating $\mathbf{x}\left(\hat{\mathbf{r}}\right)$ (Theorem 4 of \citep{galanti2020modularity}. Besides, Problem \eqref{CP} is a pseudoconvex programming problem, which means, $\mathbf{x}\left(\hat{\mathbf{r}}\right)$ is a global optimal minimum. Moreover, by Proposition 2.1, $\mathbf{x}$ is also a Pareto optimal solution to Problem \eqref{MOP}. We choose any $\mathbf{r}^{*}\in U - $ neighborhood of $\hat{\mathbf{r}}$, i.e. $\mathbf{r}^*\in\mathcal{P}$, obviously a local optimal minimum $\mathbf{x}(\mathbf{r}^*)$ to Problem \eqref{CP} be approximated by a smooth function $h\left(\mathbf{r}^*,\phi\right)$ \citep{cybenko1989approximation}. Hence, $\mathbf{x}\left(\mathbf{r}^*\right)$ is also a Pareto optimal solution to Problem \eqref{MOP}.

\begin{algorithm}[h]
\KwIn{Init $\phi_0,t=0$.}
\KwOut{$\phi^*$.}

\While{not converged}{
$\mathbf{r}\sim \text{Dirichlet}(\alpha)$

$\mathbf{x}_{\mathbf{r}} = h \left(\mathbf{r},\phi_t\right)$

$\phi_{t+1} = \phi_t - \xi\nabla_{\phi}s\left(\mathcal{F}\left(\mathbf{x}_{\mathbf{r}}\right),\mathbf{r}\right)$

$t=t+1$
}
$\phi^* = \phi_t$
\caption{: Hypernetwork training for Problem \eqref{CS-PHN}.}
\label{alg:hypermoo}
\end{algorithm}
\section{Preference-Based Completed Scalarization Functions}
Based on Scalarization theories, we discuss three approaches to define the different preference-base functions: scalarization functions with preference vector parameter $\mathbf{r}$.
\subsection{Preference-Based Linear Scalarization function}
A simple and straightforward approach is to define the preference vector $\mathbf{r}$ and the corresponding solution $\mathbf{x}_{\mathbf{r}}$ via the weighted linear scalarization:
\begin{align*}\label{LS}\tag{LS}
    s\left(\mathcal{F}\left(\mathbf{x}_{\mathbf{r}}\right),\mathbf{r}\right) = \sum_{i=1}^m r_{i}f_i\left(\mathbf{x}_{\mathbf{r}}\right),
\end{align*}
where $\mathbf{x}_{\mathbf{r}} = h(\mathbf{r},\phi)$, $\mathbf{r}\in \left\{\mathbf{r} \in \mathbb{R}^m_{>0}: \sum_i^{m} r_i = 1\right\}$.

Although this approach is straightforward and proposed by \citep{navon2020learning}, Linear scalarization can not find Pareto optimal solutions on the non-convex part of the Pareto front \citep{das1997closer, boyd2004convex}. In other words, unless the problem has a convex Pareto front, the generator defined by linear scalarization cannot cover the whole Pareto set manifold.
\begin{theorem}\citep{miettinen2012nonlinear}
\label{theorem5.1}
    If $\mathbf{x}^*\in X$ is Pareto optimal of Problem \eqref{MOP}, then there exists a preference vector $\mathbf{r} \left(r_i>0\right)$ such that $\mathbf{x}^*$ is a solution of Problem \eqref{LS}
\end{theorem}
According to Theorem \ref{theorem5.1}, all the Pareto optimal solutions of Problem \eqref{MOP} can be found by Problem \eqref{LS}. Indeed, by $f_i$ is a convex function, then $s(\mathcal{F}(\mathbf{x}),\mathbf{r})$ is a convex function, i.e., $s$ is a pseudoconvex function on $X$. From Proposition 5.52 \citep{dinh2005generalized}, and $s$ is a monotonically increasing function, then we can obtain \eqref{LS} is a completed scalarization function, combined with Theorem \ref{theorem3.1}, we confirmed that all the Pareto optimal solutions of Problem \eqref{MOP} could be found by Problem \eqref{LS}.

\subsection{Preference-Based Chebyshev function}
Chebyshev function has been used in many EMO algorithms such as MOEA/D \citep{Zhang2008}, is defined as follow:
\begin{align*}\label{Cheby}\tag{Cheby}
    s\left(\mathcal{F}\left(\mathbf{x}_{\mathbf{r}}\right),\mathbf{r}\right) = \underset{i=1,\dots,m}{\max}\left\{r_i\left|f_i\left(\mathbf{x}_{\mathbf{r}}\right)-z^{*}_i\right|\right\},
\end{align*}
where $\mathbf{x}_{\mathbf{r}} = h(\mathbf{r},\phi)$, $\mathbf{r}\in\left\{\mathbf{r} \in \mathbb{R}^m_{>0}: \sum_i^{m} r_i = 1\right\}$, $z^{*}_i$ is the reference point. In this study, we normalized the objective function values in the range $[0,1]$ before building the model. Therefore, $z^{*}_i$ is a value of zeros. 
\begin{theorem}\citep{miettinen2012nonlinear}
    If $\mathbf{x}^*\in X$ is Pareto optimal of Problem \eqref{MOP}, then there exists a preference vector $\mathbf{r} \left(r_i>0\right)$ such that $\mathbf{x}^*$ is a solution of Problem \eqref{Cheby}
\end{theorem}
Convexity of the MOO problem is needed in order to guarantee that every Pareto optimal solution can be found by Problem \eqref{Cheby} (see p.81 of \citep{sawaragi1985theory}. 

Now, we will show that $s\left(\mathcal{F}\left(\mathbf{x}\right),\mathbf{r}\right)$ is both a completed scalarization function and a pseudoconvex function on $X$. Firstly, $f_i$ is a convex function, we have that $s$ is also a convex function on $X$ \citep{nguyen2014cac}. Hence, we imply that $s\left(\mathcal{F}\left(\mathbf{x}\right),\mathbf{r}\right)$ is a pseudoconvex function on $X$. Besides, $s$ is monotonically increasing function with $\mathbf{r}>0$. Moreover, by Proposition 5.22 \citep{dinh2005generalized}, \eqref{Cheby} is a completed scalarization function. Applying Theorem \ref{theorem3.1}, we confirmed that all the Pareto optimal solutions of Problem \eqref{MOP} could be found by Problem \eqref{Cheby}.

\subsection{Preference-Based Inverse Utility function}
Sometimes the term utility function is used instead of the value function, which is often assumed that the decision-maker makes decisions based on an underlying function of some kind.
\begin{definition}\citep{miettinen2012nonlinear}
    A utility function $U:\mathbb{R}^m\to\mathbb{R}$ representing the preferences of the decision maker among the objective vectors is called a value function or utility function.
\end{definition}
Our inverse utility function is defined as:
\begin{align*}\label{Utility}\tag{Utility}
    s\left(\mathcal{F}\left(\mathbf{x}_{\mathbf{r}}\right),\mathbf{r}\right) = \dfrac{1}{\Pi_{i=1}^m\left(u_i-f_i\left(\mathbf{x}_{\mathbf{r}}\right)\right)^{r_i}},
\end{align*}
where $\mathbf{x}_{\mathbf{r}} = h(\mathbf{r},\phi)$, $\mathbf{r}\in \left\{\mathbf{r} \in \mathbb{R}^m_{>0}: \sum_i^{m} r_i = 1\right\}$, $u_i$ is an upper bound of $f_i$.
\begin{theorem}\citep{miettinen2012nonlinear}\label{theorem5.3}
    Let the utility function $U:\mathbb{R}^m\to\mathbb{R}$ be strongly decreasing. Let $U$ attain its maximum at $\mathbf{x}^*\in X$. Therefore, $\mathbf{x}^*$ is Pareto optimal solution of Problem \eqref{MOP}.
\end{theorem}
Let $U = \Pi_{i=1}^m\left(u_i-f_i\left(\mathbf{x}_{\mathbf{r}}\right)\right)^{r_i}$ is well known Cobb–Douglas production function \citep{cobb1928theory}. Simple to verify $U$ is strongly decreasing function with $r_i>0$ and $u_i-f_i\left(\mathbf{x}_{\mathbf{r}}\right)>0, i=1,\dots,m$. Then, $s=\dfrac{1}{U}$ is an increasing function. Hence, from Theorem \ref{theorem5.3}, all the Pareto optimal solutions of Problem \eqref{MOP} can be found by Problem \eqref{Utility}.

Indeed, with $u_i$ calculated using the technique of \citep{benson1998outer}, by Corollary 5.18 \citep{avriel2010generalized}, we acquire $s$ is a convex function and also a pseudoconvex function on $X$. Therefore, \eqref{Utility} is a completed scalarization function. The proposed function $s$ has significant implications for economics. With assuming that we need to minimize the cost funtion $f_i$ to maximize the profit $\Pi_{i=1}^m\left(u_i-f_i\left(\mathbf{x}_{\mathbf{r}}\right)\right)^{r_i}$.
\section{Application of Controllable Pareto Front Learning in Multi-task Learning}
\subsection{Multi-task Learning as Multi-objectives optimization.}
In machine learning, Multi-task learning (MTL) is a part of Meta-Objectives, one of three main fields of Meta-Learning. The application of Multi-task learning is successful in computer vision \citep{bilen2016integrated, misra2016cross, anh2022multi}, in natural languages \citep{dong2015multi, hashimoto2016joint} and recommender system \citep{le2020stochastically, le2021efficient}. 
In the early days, MTL algorithms often are heuristic in order to dynamically balance the loss terms based on gradient magnitude \citep{chen2018gradnorm}, the rate of change in losses \citep{liu2019end}, task uncertainty \citep{kendall2018multi}. Hence, \citep{sener2018multi} formulated MTL as a MOO and proposed using MDGA \citep{desideri2012multiple} for MTL. 

Denotes a supervised dataset $\left(\mathbf{x},\mathbf{y}\right)=\left\{\left(x_j,y_j\right)\right\}_{j=1}^N$  where $N$ is the number of data points. They specified the MOO formulation of Multi-task learning from the empirical loss $\mathcal{L}^i(\mathbf{y},g(\mathbf{x},\boldsymbol{\theta}))$ using a vector-valued loss $\mathcal{L}$:
\begin{align*}
    \boldsymbol{\theta} &= \argmin_{\boldsymbol{\theta}} \mathcal{L}\left(\mathbf{y},g\left(\mathbf{x},\boldsymbol{\theta}\right)\right), \\
    \mathcal{L}\left(\mathbf{y}, g\left(\mathbf{x}, \boldsymbol{\theta}\right)\right) &= \left(\mathcal{L}_1\left(\mathbf{y},g\left(\mathbf{x},\boldsymbol{\theta}\right)\right),\dots,\mathcal{L}_m\left(\mathbf{y},g\left(\mathbf{x},\boldsymbol{\theta}\right)\right)\right)^T
\end{align*}
where $g\left(\mathbf{x}; \boldsymbol{\theta} \right): \mathcal{X}\times\Theta \rightarrow \mathcal{Y}$ represents to a Target network with parameters $\boldsymbol{\theta}$. 
\begin{definition}[Dominance] A solution $\boldsymbol{\theta}^a$ dominates $\boldsymbol{\theta}^b$ if $\mathcal{L}_i\left(\mathbf{y}, g\left(\mathbf{x}, \boldsymbol{\theta}^a\right)\right) \leq \mathcal{L}_i\left(\mathbf{y}, g\left(\mathbf{x}, \boldsymbol{\theta}^b\right)\right) \forall i \in \{1,\dots,m\}$ and $\mathcal{L}\left(\mathbf{y}, g\left(\mathbf{x}, \boldsymbol{\theta}^a\right)\right) \neq \mathcal{L}\left(\mathbf{y}, g\left(\mathbf{x}, \boldsymbol{\theta}^b\right)\right)$. Denote $\mathcal{L}(\mathbf{y}, g(\mathbf{x}, \boldsymbol{\theta}^a)) \prec \mathcal{L}(\mathbf{y}, g(\mathbf{x}, \boldsymbol{\theta}^b))$ or $\mathcal{L}^a \prec \mathcal{L}^b$.
\end{definition}
\begin{definition}[Pareto optimal solution] A solution $\boldsymbol{\theta}^a$ is called Pareto optimal solution if $\nexists \boldsymbol{\theta}^b: \mathcal{L}^b \prec \mathcal{L}^a$.
\end{definition}
\begin{definition}[Pareto front] The set of Pareto optimal solutions is Pareto set, denoted by $P_s$, and the corresponding images in objectives space are Pareto front, denoted by  $P_f=\mathcal{L}\left(\mathbf{y}, g\left(\mathbf{x}, \Theta\right)\right)$.
\end{definition}
\subsection{Controllable Pareto Front Learning in Multi-task Learning.}
\begin{algorithm}[ht]
\KwIn{Supervised dataset $(\boldsymbol{x}, \boldsymbol{y})$}
\KwOut{$\phi^*$}

\While{not converged}{
$\mathbf{r}\sim \text{Dirichlet}(\alpha)$

Sample mini-batch $\left(x_1, {y}_1\right),\dots,\left(x_B,{y}_B\right)$

$\boldsymbol{\theta}_{\mathbf{r}} = h(\mathbf{r}, \phi)$

$\mathcal{L}(\boldsymbol{\theta}_{\boldsymbol{r}}) = \left[\sum_{i=1}^B\mathcal{L}^j\left(y_i,g\left(x_i,\boldsymbol{\theta}_{\mathbf{r}}\right)\right)\right]_{j=1}^m$ 

$\phi = \phi - \xi\nabla_{\phi}s\left(\mathcal{L}\left(\boldsymbol{\theta}_{\mathbf{r}}\right),\mathbf{r}\right)$
}
$\phi^* = \phi$
\caption{: Hypernetwork for Multi-task Learning.}
\end{algorithm}

Pareto front learning in Multi-task Learning by solving the following:
\begin{align*}
    & \phi^* = \argmin_{\phi} \underset{\mathbf{r} \sim p_{\mathcal{P}},(\mathbf{x}, \mathbf{y}) \sim p_D}{\mathbb{E}} s(\mathcal{L}(\mathbf{y},g(\mathbf{x}, \boldsymbol{\theta}_{\mathbf{r}}), \mathbf{r}) \\
    & \text{ s.t. } \boldsymbol{\theta}_{\mathbf{r}} = h(\mathbf{r}, \phi^*),  h(\mathcal{P}, \phi^*) = P_s,
\end{align*}
where $h: \mathcal{P} \times \Phi \rightarrow \Theta$ represents to a Hypernetwork, random variable $\mathbf{r}$ is preference vector which formulate trade-off between objective functions, $p_{\mathcal{P}}$ is a random distribution on $\mathcal{P}$.

\section{Computational experiments}
The code is implemented in Python language programming and Pytorch framework \citep{paszke2019pytorch}. We compare the performance of our method with the baseline methods: PHN-LS and PHN-EPO \citep{navon2020learning}. Due to the page limitation, we mention the setting details and additional experiments in Appendix. Our source code is available at \revise{\url{https://github.com/tuantran23012000/PHN-CSF.git}}.
\subsection{Evaluation metrics}
\revise{\textbf{Mean Euclid Distance (MED).} To evaluate how error the obtained Pareto optimal solutions by the hypernetwork $\mathcal{F}^*=\{\mathcal{F}^*_1,\dots,\mathcal{F}^*_{|\mathcal{F}^*|}\}$ to the Pareto optimal solutions on the Pareto front $\hat{\mathcal{F}}=\{\hat{\mathcal{F}}_1,\dots,\hat{\mathcal{F}}_{|\hat{\mathcal{F}}|}\}$ corresponding to reference vector set $\{\mathbf{r} \in \mathbb{R}^m_{>0}: \sum_i^{m} r_i = 1\}$. Then the performance indicator is defined as:
\begin{align*}
    MED(\mathcal{F}^*,\hat{\mathcal{F}}) = \dfrac{1}{|\mathcal{F}^*|}\left(\sum_{i=1}^{|\mathcal{F}^*|}\left\Vert \mathcal{F}^*_i - \hat{\mathcal{F}}_i\right\Vert_2\right).
\end{align*}
}
\revise{In contrast to the IGD and GD metrics (mentioned by \citep{li2019quality}, which measure the minimized distance between a point in the reference points and the Pareto solution set. Our MED calculates the distance between pairings of predictions and targets in the set of predicted and corresponding Pareto optimal solutions.} 

\textbf{Hypervolume (HV).} Hypervolume \citep{zitzler1999multiobjective} is the area dominated by the Pareto front. Therefore the quality of a Pareto front is proportional to its hypervolume. Given a set of $k$ points $\mathcal{M} = \{m^j | m^j \in \mathbb{R}^m; j=1,\dots, k\}$ and a reference point $\rho\in\mathbb{R}^m_{+}$ \revise{(Appendix \ref{sub:hv} provides the way to choose a reference point)}, the Hypervolume of $\mathcal{S}$ is measured by the region of non-dominated points bounded above by $m \in \mathcal{M}$, then the hypervolume metric is defined as follows:
\begin{align*}
    HV(S) = VOL\left(\underset{m \in \mathcal{M}, m \prec \rho}{\bigcup}\displaystyle{\Pi_{i=1}^m}\left[m_i,\rho_i\right]\right)
\end{align*}
Moreover, we also adopt metrics to evaluate the model performance, including mean accuracy, precision, recall, and f1 score in Multi-label classification.

\textbf{Mean Accuracy score (mA).}
\begin{align*}
    mA = \dfrac{1}{K}\dfrac{1}{M}\displaystyle\sum^{K}\sum_{j=1}^M\dfrac{1}{2}\left(\dfrac{TP^j}{TP^j+FN^j} + \dfrac{TN^j}{TN^j+FP^j}\right)
\end{align*}

\textbf{Precision score (Pre).}
\begin{align*}
    Pre = \dfrac{1}{K}\dfrac{1}{M}\displaystyle\sum^{K}\sum_{j=1}^M\dfrac{TP^j}{TP^j+FP^j}
\end{align*}

\textbf{Recall score (Recall).}
\begin{align*}
    Recall = \dfrac{1}{K}\dfrac{1}{M}\displaystyle\sum^{K}\sum_{j=1}^M\dfrac{TP^j}{TP^j+FN^j}
\end{align*}

\textbf{F1 score (F1).}
\begin{align*}
    F1 = \dfrac{1}{K}\dfrac{1}{M}\displaystyle\sum^{K}\sum_{j=1}^M\dfrac{2*Pre*Recall}{Pre+Recall}
\end{align*}
 where $K$ is the number of test preference vectors, $M$ is the number of labels. 

\subsection{MOO problems}
In the following, we investigate Controllable Pareto Front Learning methods in several convex MOO problems, in which objective functions are convex and constraint space is a convex set \revise{(Non-Convex MOO experiments are provided in the Appendix \ref{sec:non-convex})}. In the optimization process, we sample use Algorithm \ref{alg:hypermoo} with $20000$ iterations and compute mean, variation of MED score by 30 executions as the main metric the evaluate the methods.

\noindent \textbf{Example 7.1:}
\begin{align*}
\min & \left\{x,(x-1)^2\right\}\\
\text{s.t. } & 0\le x\le 1.
\end{align*}

\noindent \textbf{Example 7.2} \label{ex7.2} \citep{binh1997mobes}:
\begin{align*}
    \min & \left\{f_1,f_2\right\}\\
\text{s.t. } & x_i\in[0,5], i=1,2
\end{align*}
where \begin{align*}
    f_1 = \dfrac{x_1^2 + x_2^2}{50}, \textbf{ } f_2 = \dfrac{(x_1-5)^2 + (x_2-5)^2}{50}.
\end{align*}

\noindent \textbf{Example 7.3}\label{ex7.3} \citep{thang2020monotonic}:
\begin{align*}
    \min & \left\{f_1,f_2,f_3\right\}\\
\text{s.t. } & x_{1}^2+x_{2}^2+x_3^2 = 1\\
& x_i\in[0,1], i=1,2,3
\end{align*}
where \begin{align*}
    & f_1 = \dfrac{x_1^2 + x_2^2 + x_3^2 +x_2 - 12x_3 + 12}{14},\\
    & f_2 = \dfrac{x_1^2 + x_2^2 + x_3^2 + 8x_1 - 44.8x_2 + 8x_3 + 44}{57},\\
    & f_3 = \dfrac{x_1^2 + x_2^2 + x_3^2 - 44.8x_1 + 8x_2 + 8x_3 + 43.7}{56}.
\end{align*}

\begin{table}[ht]%
\caption{MED score of methods in example 7.1. }
\label{tb1}
\centering
\resizebox{0.5\textwidth}{!}{\begin{tabular}{|c|c|c|c|c|}
\toprule
& \multicolumn{1}{c|}{PHN-EPO} &\multicolumn{1}{c|}{PHN-LS} & \multicolumn{1}{c|}{PHN-Cheby} & \multicolumn{1}{c|}{PHN-Utility} \\
\cmidrule(lr){2-3}\cmidrule(lr){3-4}\cmidrule(lr){4-5}
rays & MED $\Downarrow$& MED$\Downarrow$& MED$\Downarrow$ & MED$\Downarrow$   \\ 
\midrule 
5 & $0.0082\pm  0.0023$ & $0.0035\pm  0.0028$ & $0.0087\pm  0.0013$ & $\bf 0.0024\pm  0.0007$ \\  
10 & $0.0087\pm  0.0013$ & $0.0043\pm  0.0026$ &$0.0086\pm  0.0009$ & $\bf 0.0025\pm  0.0006$\\
50 & $0.0088\pm  0.0006$ & $0.0042\pm  0.0008$ &$0.0084\pm  0.0005$ & $\bf 0.0025\pm  0.0003$\\
100 & $0.0087\pm  0.0004$ & $0.0038\pm  0.0003$ &$0.0084\pm  0.0005$ & $\bf 0.0025\pm  0.0002$\\
300 & $0.0086\pm  0.0003$ & $0.0041\pm  0.0004$ &$0.0086\pm  0.0002$ & $\bf 0.0025\pm  0.0001$\\
600 & $0.0086\pm  0.0001$ & $0.0041\pm  0.0002$ &$0.0085\pm  0.0001$ & $\bf 0.0024\pm  0.0001$\\
\bottomrule
\end{tabular}}
\end{table}


\begin{table}[ht]%
\caption{MED score of methods in example 7.2. }
\label{tb2}
\centering
\resizebox{0.5\textwidth}{!}{\begin{tabular}{|c|c|c|c|c|}
\toprule
& \multicolumn{1}{c|}{PHN-EPO} &\multicolumn{1}{c|}{PHN-LS} & \multicolumn{1}{c|}{PHN-Cheby} & \multicolumn{1}{c|}{PHN-Utility} \\
\cmidrule(lr){2-3}\cmidrule(lr){3-4}\cmidrule(lr){4-5}
rays & MED $\Downarrow$& MED$\Downarrow$& MED$\Downarrow$ & MED$\Downarrow$   \\ 
\midrule
5 & $0.0017\pm  0.0005$ & $0.0017\pm  0.0006$ & $0.0018\pm  0.0005$ & $\bf 0.0013\pm  0.0005$ \\  
10 & $0.0017\pm  0.0004$ & $0.0018\pm  0.0004$ &$0.0020\pm  0.0003$ & $\bf 0.0014\pm  0.0004$\\
50 & $0.0017\pm  0.0002$ & $0.0017\pm  0.0002$ &$0.0019\pm  0.0002$ & $\bf 0.0013\pm  0.0001$\\
100 & $0.0018\pm  0.0001$ & $0.0018\pm  0.0001$ &$0.0019\pm  0.0001$ & $\bf 0.0014\pm  0.0001$\\
300 & $0.0018\pm  0.0001$ & $0.0018\pm  0.0001$ &$0.0019\pm  0.0001$ & $\bf 0.0014\pm  0.0001$\\
600 & $0.0018\pm  0.0001$ & $0.0018\pm  0.0001$ &$0.0019\pm  0.0001$ & $\bf 0.0014\pm  0.0001$\\
\bottomrule
\end{tabular}}
\end{table}

\begin{table}[ht]%
\footnotesize
\caption{MED score of methods in example 7.3. }
\label{tb3}
\centering
\resizebox{0.5\textwidth}{!}{\begin{tabular}{|c|c|c|c|c|}
\toprule
& \multicolumn{1}{c|}{PHN-EPO} &\multicolumn{1}{c|}{PHN-LS} & \multicolumn{1}{c|}{PHN-Cheby} & \multicolumn{1}{c|}{PHN-Utility} \\
\cmidrule(lr){2-3}\cmidrule(lr){3-4}\cmidrule(lr){4-5}
rays & MED $\Downarrow$& MED$\Downarrow$& MED$\Downarrow$ & MED$\Downarrow$   \\ 
\midrule
5 & $0.0828\pm  0.0162$ & $0.0531\pm  0.0137$ & $0.0397\pm  0.0117$ & $\bf 0.0244\pm  0.0054$ \\  
10 & $0.0809\pm  0.0105$ & $0.0490\pm  0.0102$ &$0.0396\pm  0.0090$ & $\bf 0.0238\pm  0.0040$\\
50 & $0.0808\pm  0.0052$ & $0.0494\pm  0.0043$ &$0.0395\pm  0.0036$ & $\bf 0.0201\pm  0.0022$\\
100 & $0.0824\pm  0.0041$ & $0.0485\pm  0.0035$ &$0.0395\pm  0.0022$ & $\bf 0.0244\pm  0.0013$\\
300 & $0.0819\pm  0.0027$ & $0.0483\pm  0.0015$ &$0.0396\pm  0.0011$ & $\bf 0.0246\pm  0.0008$\\
600 & $0.0838\pm  0.0020$ & $0.0483\pm  0.0011$ &$0.0394\pm  0.0009$ & $\bf 0.0243\pm  0.0004$\\
\bottomrule
\end{tabular}}
\end{table}

\begin{figure*}[ht]
     \centering
        \begin{subfigure}[b]{0.24\textwidth}
         \centering
         \includegraphics[width=\textwidth]{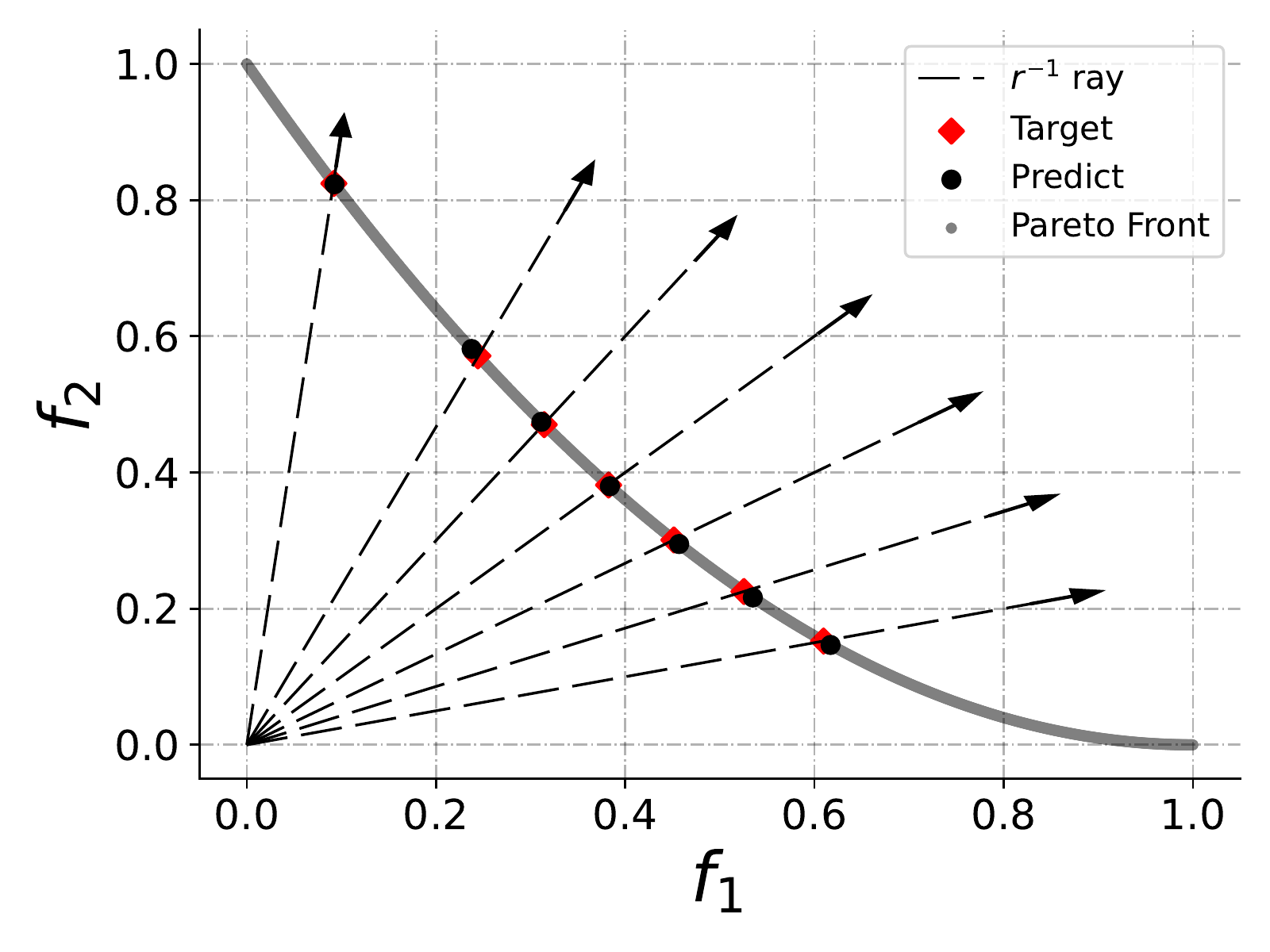}
     \end{subfigure}
     \hfill
     \begin{subfigure}[b]{0.24\textwidth}
         \centering
         \includegraphics[width=\textwidth]{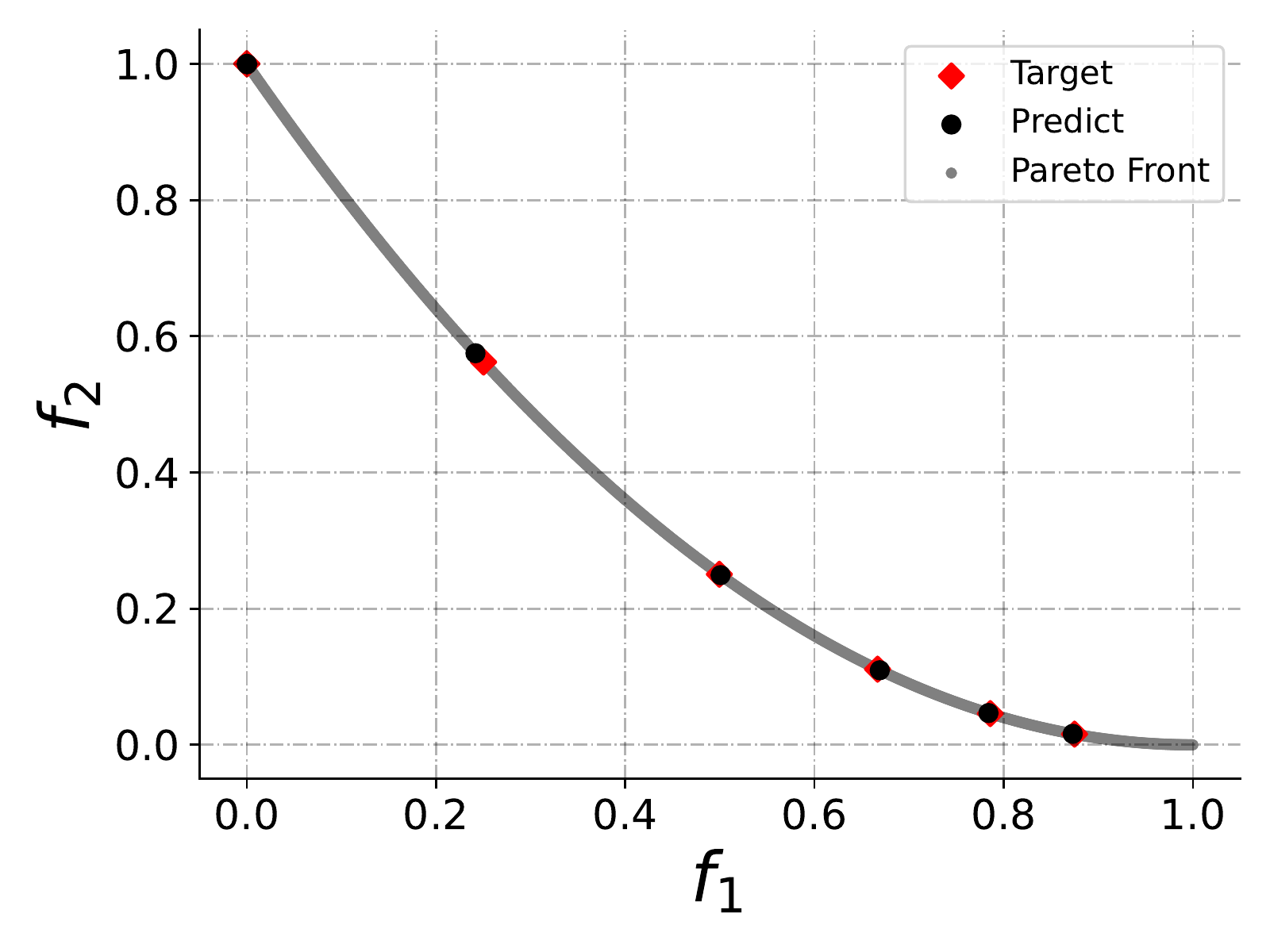}
     \end{subfigure}
     \hfill
     \begin{subfigure}[b]{0.24\textwidth}
         \centering
         \includegraphics[width=\textwidth]{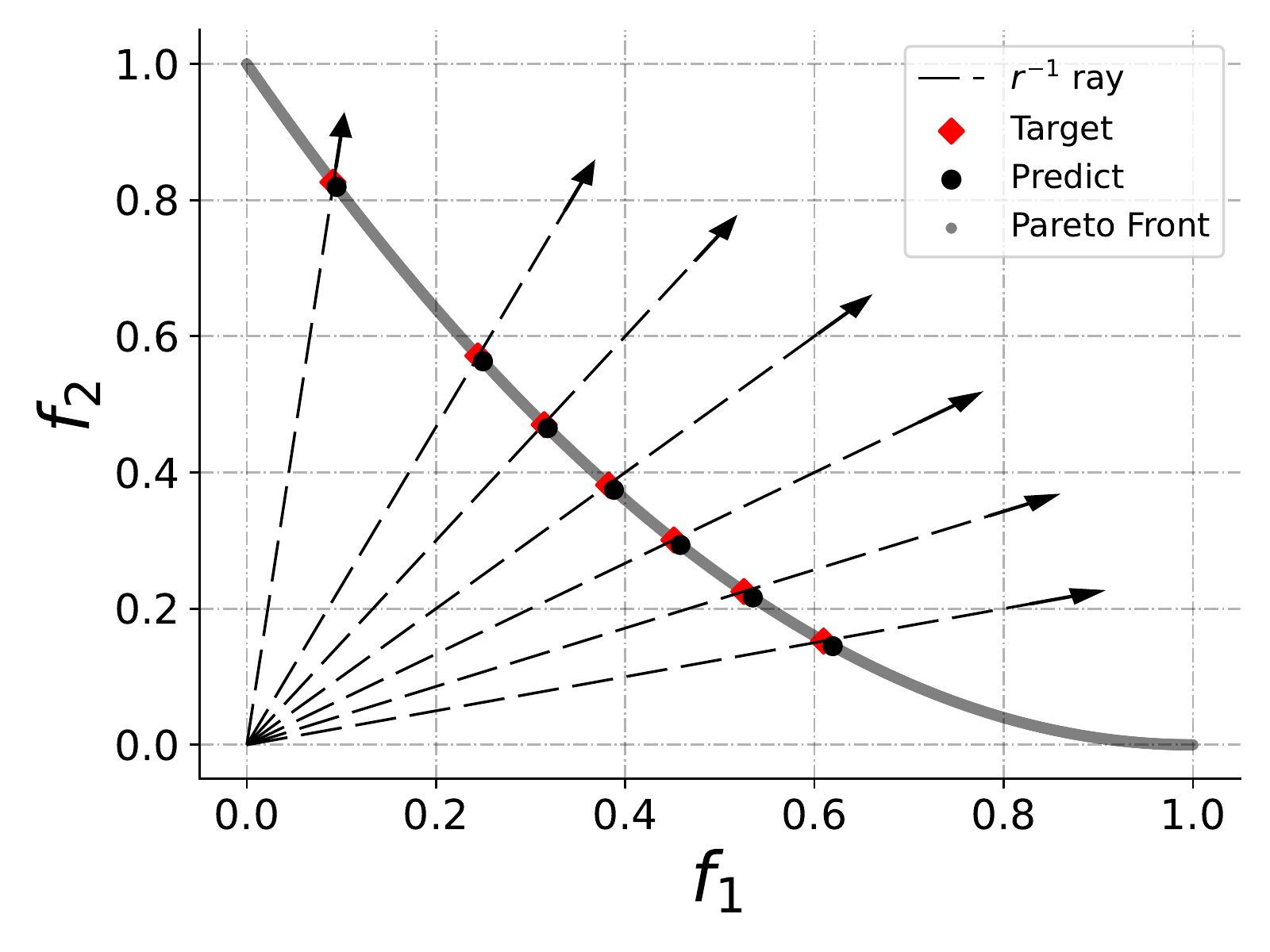}
     \end{subfigure}
     \hfill
     \begin{subfigure}[b]{0.24\textwidth}
         \centering
         \includegraphics[width=\textwidth]{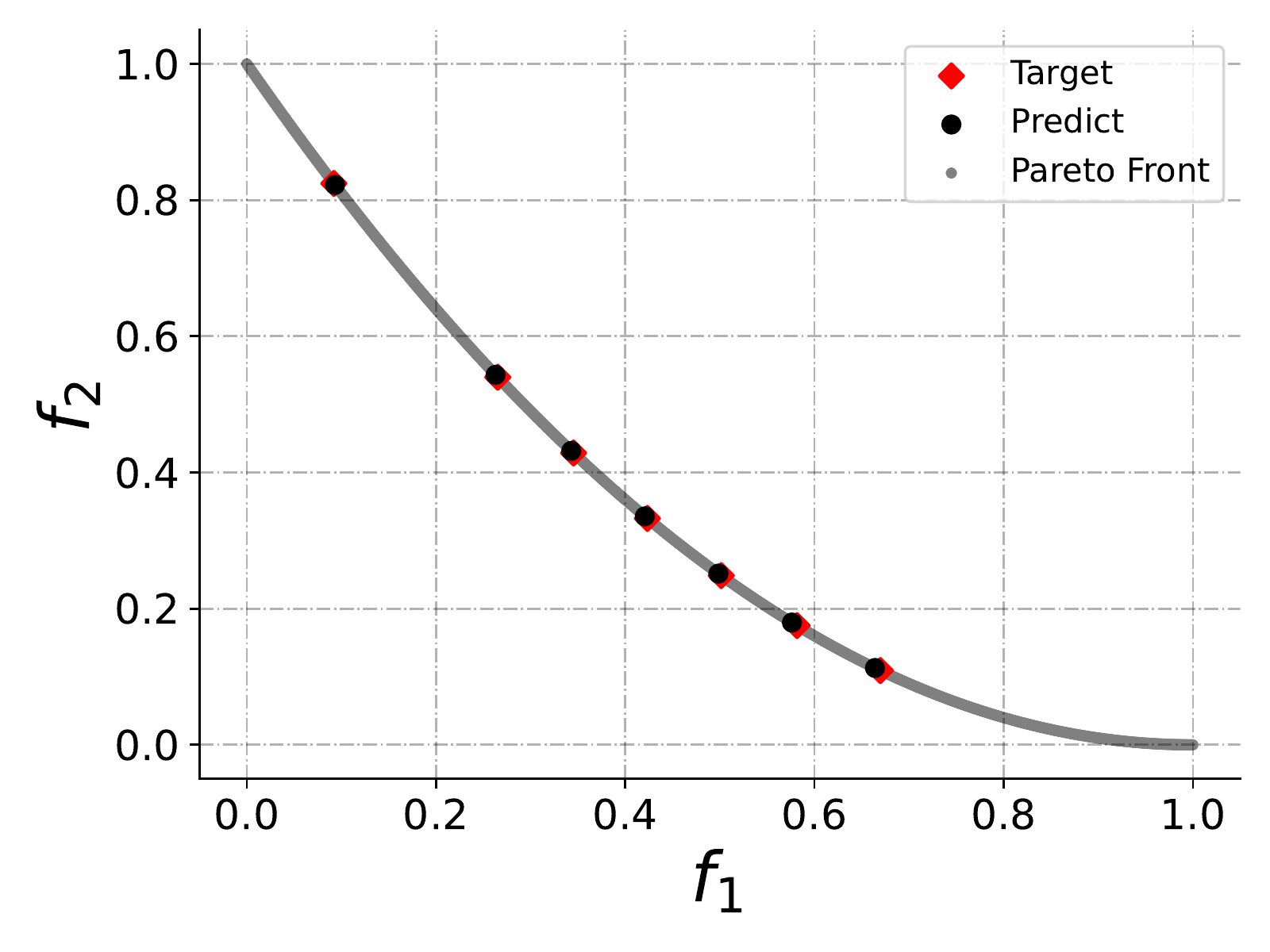}
     \end{subfigure}

     \centering
     \begin{subfigure}[b]{0.24\textwidth}
         \centering
         \includegraphics[width=\textwidth]{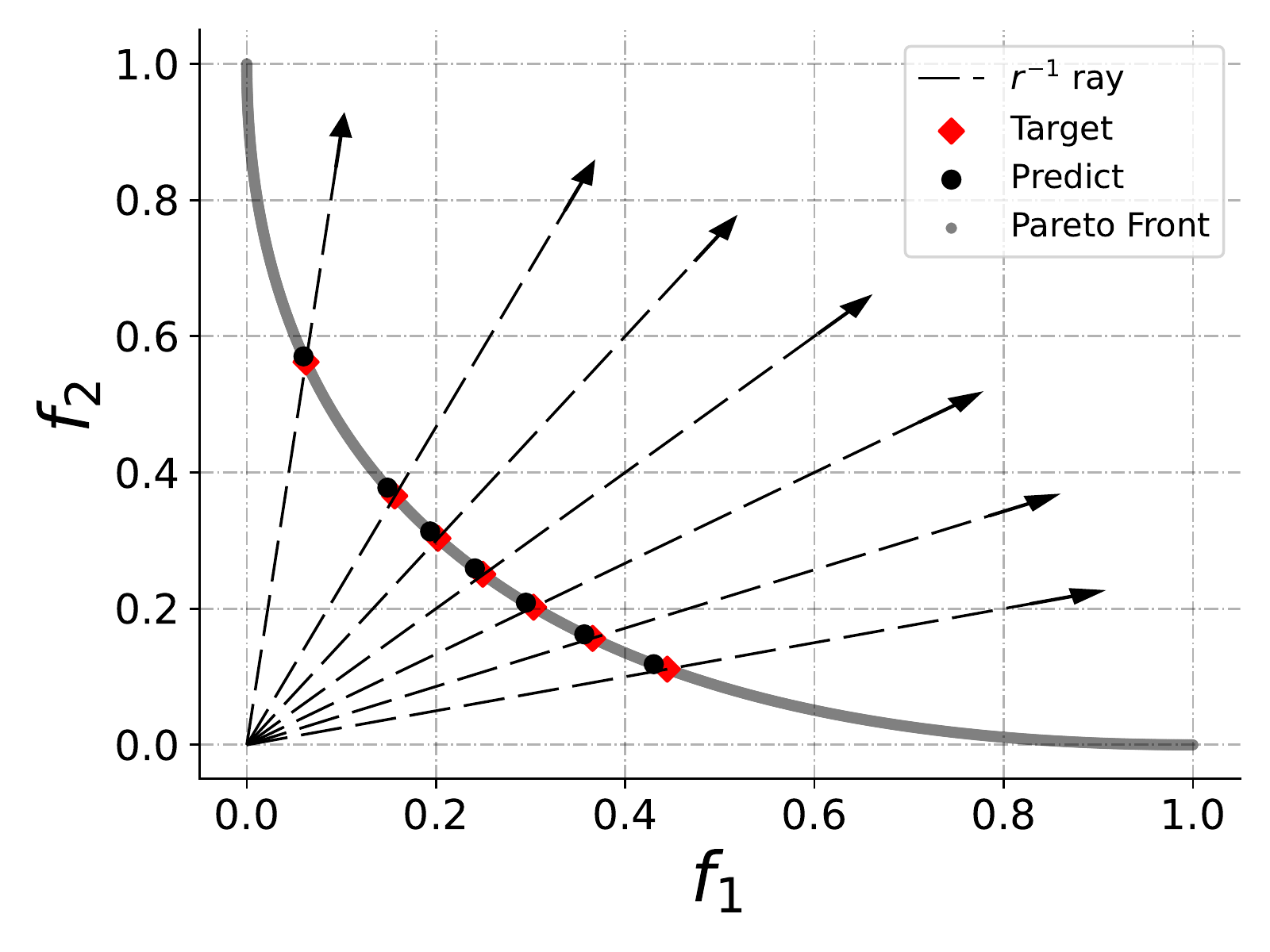}
     \end{subfigure}
     \hfill
     \begin{subfigure}[b]{0.24\textwidth}
         \centering
         \includegraphics[width=\textwidth]{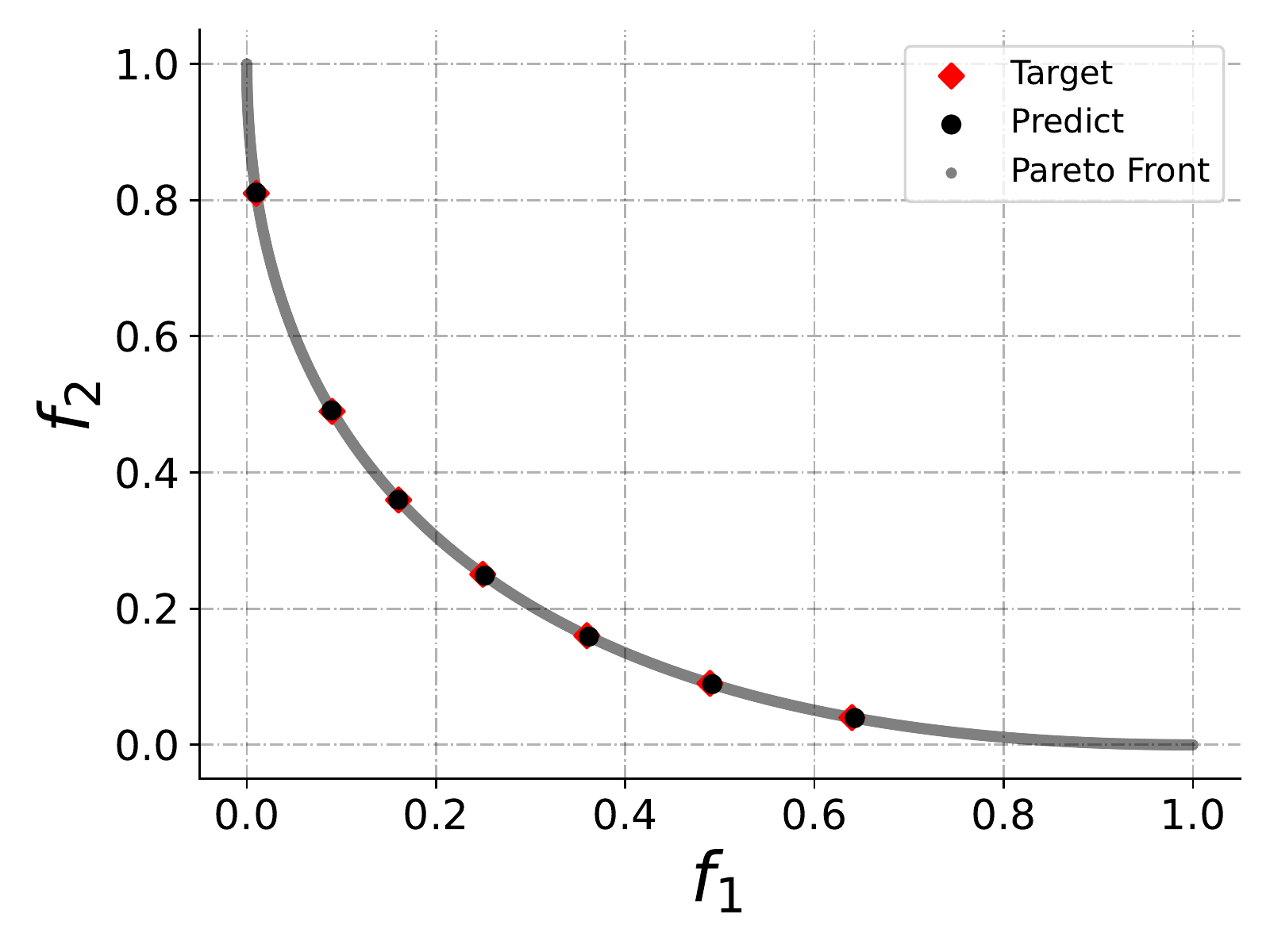}
     \end{subfigure}
     \hfill
     \begin{subfigure}[b]{0.24\textwidth}
         \centering
         \includegraphics[width=\textwidth]{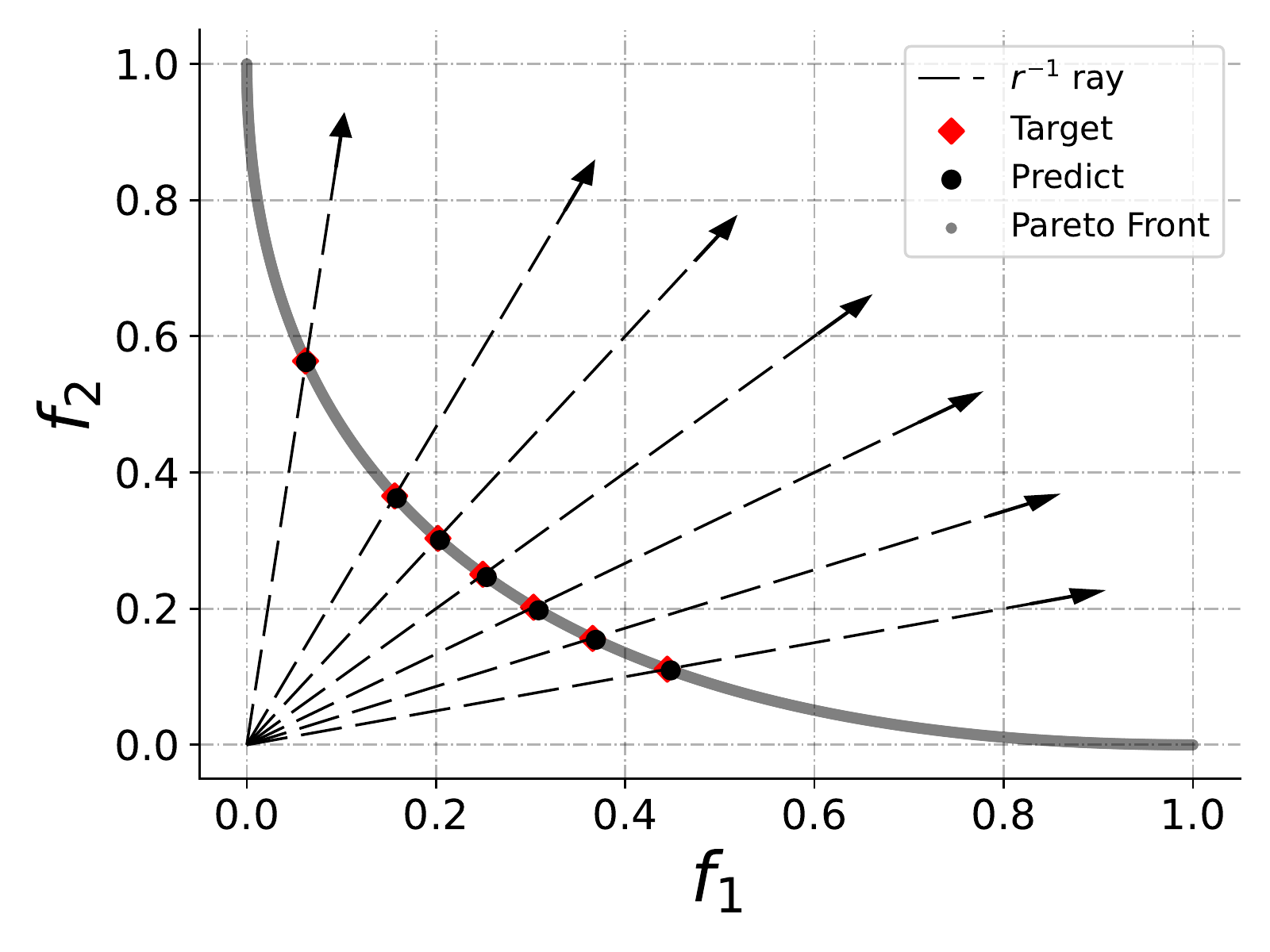}
     \end{subfigure}
     \hfill
     \begin{subfigure}[b]{0.24\textwidth}
         \centering
         \includegraphics[width=\textwidth]{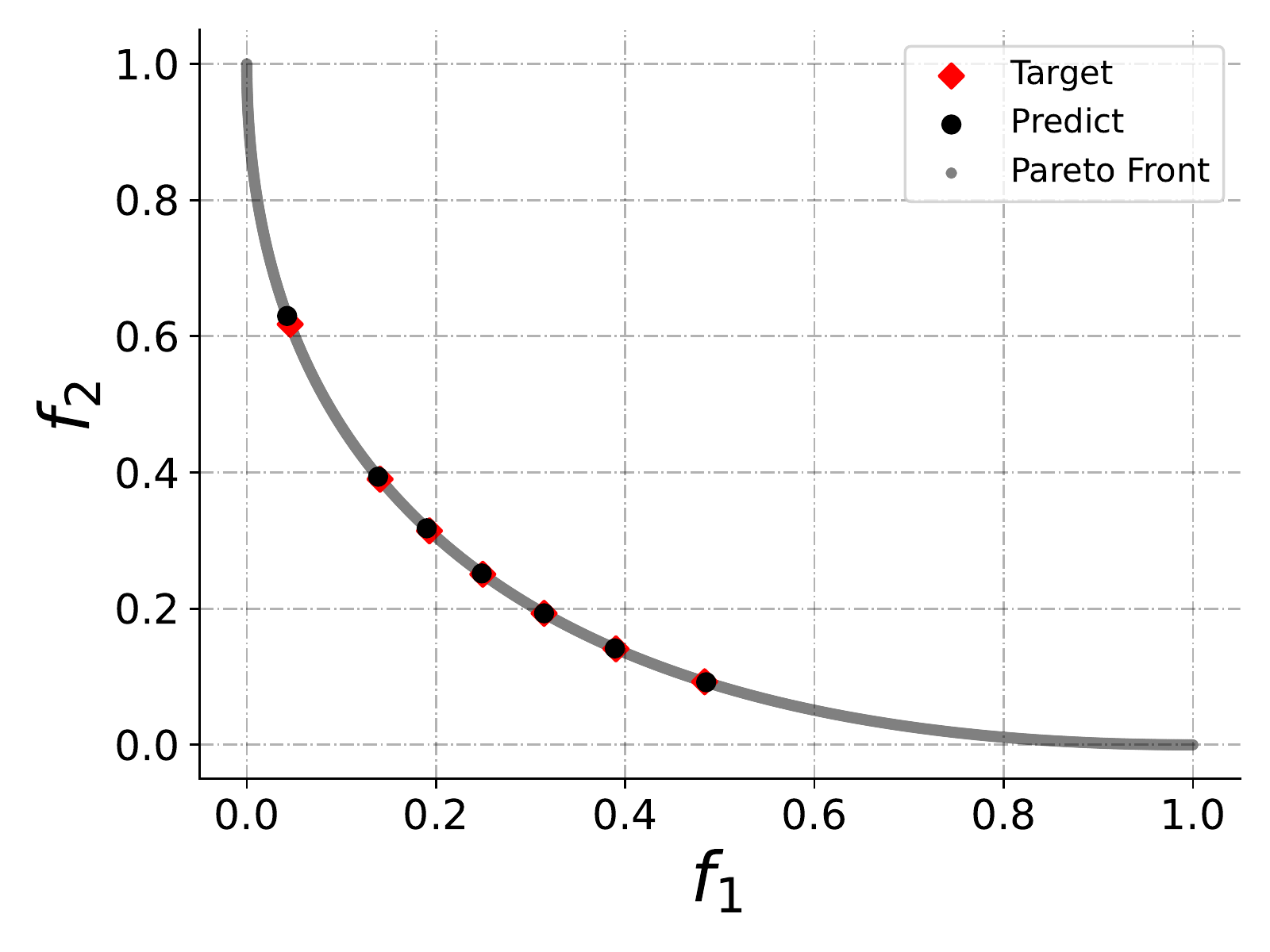}
     \end{subfigure}

    \begin{subfigure}[b]{0.24\textwidth}
         \centering
         \includegraphics[width=\textwidth]{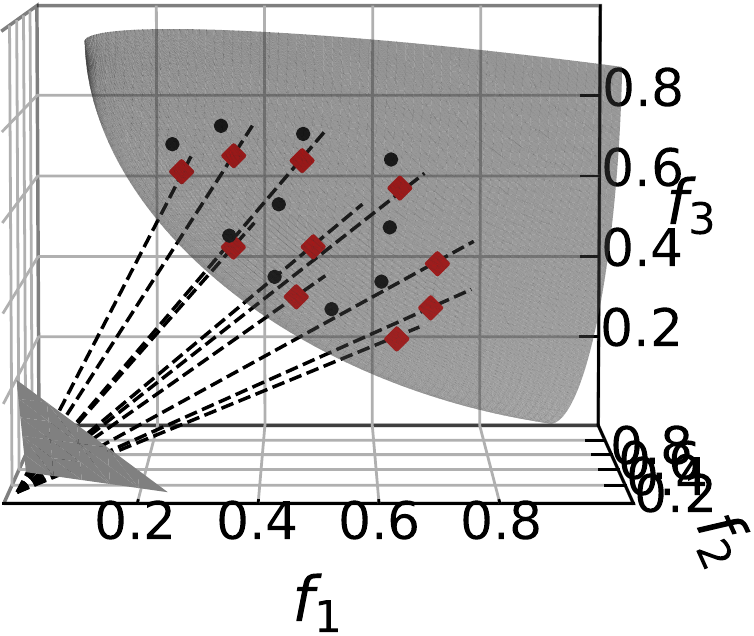}
         \caption{PHN-EPO}
         \label{3c}
     \end{subfigure}
     \hfill
     \begin{subfigure}[b]{0.24\textwidth}
         \centering
         \includegraphics[width=\textwidth]{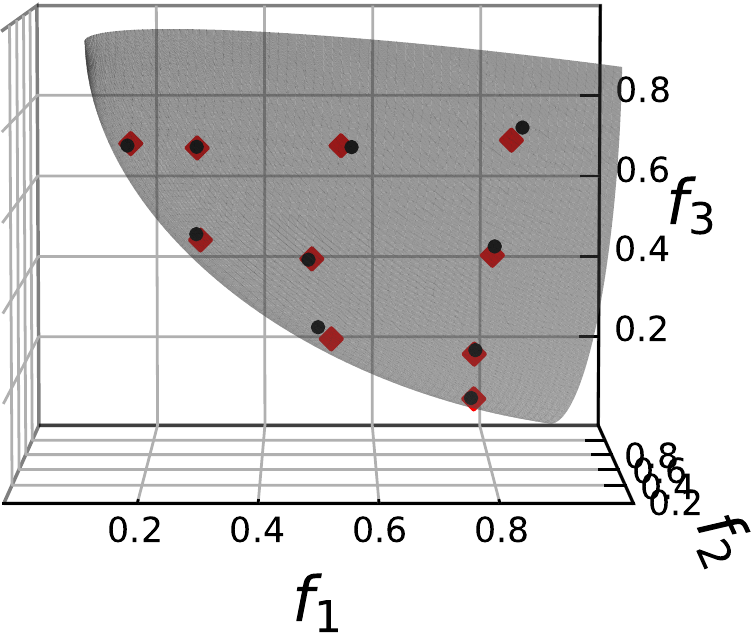}
         \caption{PHN-LS}
     \end{subfigure}
     \hfill
     \begin{subfigure}[b]{0.24\textwidth}
         \centering
         \includegraphics[width=\textwidth]{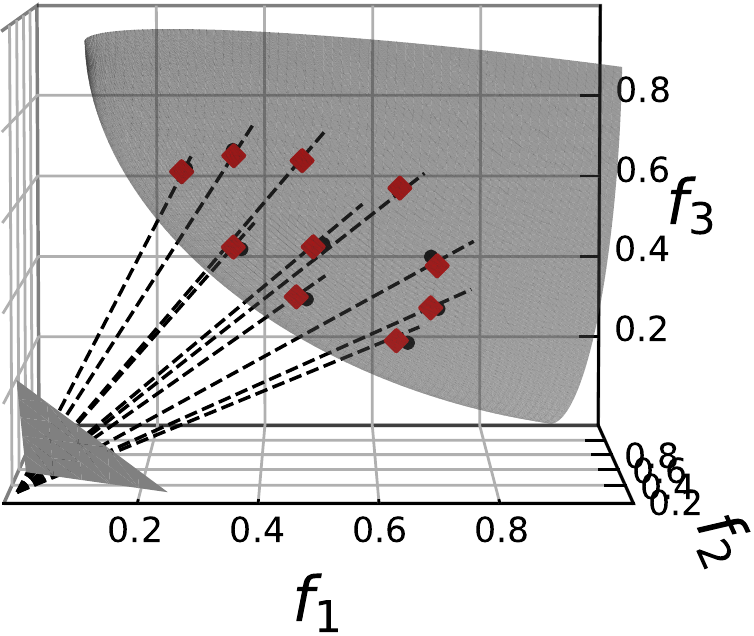}
         \caption{PHN-Cheby}
         \label{3b}
     \end{subfigure}
     \hfill
     \begin{subfigure}[b]{0.24\textwidth}
         \centering
         \includegraphics[width=\textwidth]{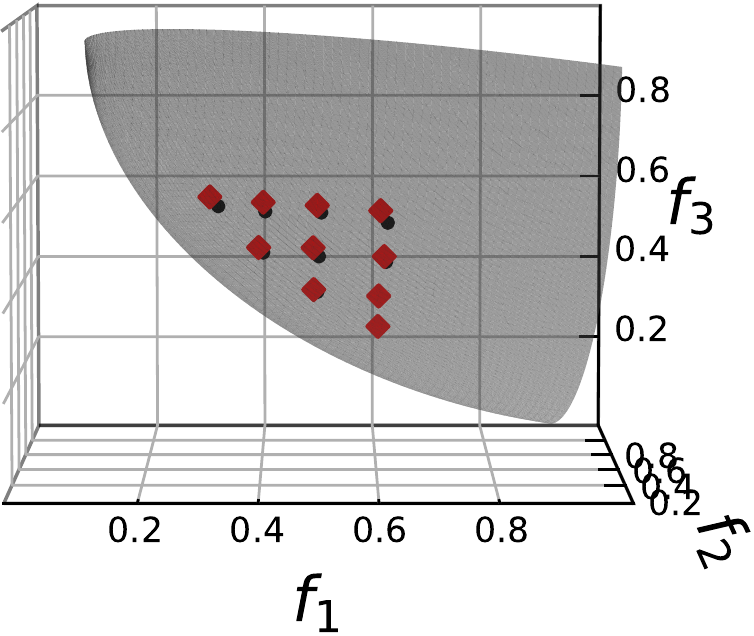}
          \caption{PHN-Utility}
     \end{subfigure}
      \caption{Output of Controllable Pareto Front Learning method in example 7.1 (top), example 7.2 (middle), and example 7.3 (bottom). The black dashed line represents the $\mathbf{r}^{-1}$ ray in Figures (a) and Figures (c). The \textcolor{red}{red} and \textcolor{black}{black} points are the target and predicted points respectively.}
      \label{fig3}
\end{figure*}






\subsection{Multi-task Learning problems}
On MTL problems, the dataset is split into three subsets: training, validation, and testing. The model with the highest HV in the validation set will be evaluated. All methods are evaluated with the same well-spread preference vectors based on \citep{das}.
\subsubsection{Image Classification.}

\begin{table}[htbp]
\caption{Testing hypervolume on Multi-MNIST, Multi-Fashion, and Multi-Fashion+MNIST datasets.}
\label{tb5}
\centering
\resizebox{0.5\textwidth}{!}{\begin{tabular}{| c | >{\centering\arraybackslash}m{1.cm} | >{\centering\arraybackslash}m{1.cm} | >{\centering\arraybackslash}m{1.cm} |>{\centering\arraybackslash}m{2.cm} | }
\toprule
                 &  {\bf Multi-MNIST} & {\bf Multi-Fashion} & {\bf Fashion-MNIST} &   \\ \midrule  

{\bf Method} & HV $\Uparrow$ & HV$\Uparrow$ & HV$\Uparrow$ &  {\bf \revise{Training-time (hour.)}} \\
\midrule
PHN-EPO  & 2.866  & 2.203  & 2.789 &  1.38 \\ 
PHN-LS  & 2.862   & \textbf{2.204}  & 2.768  &\bf 1.15  \\ 
\midrule
PHN-Cheby (\textbf{ours})  & 2.870   & 2.173 & \bf 2.807  & 1.45 \\
PHN-Utility  (\textbf{ours})  &  \bf 2.874  & 2.169  & 2.795 &  1.23 \\ 
\bottomrule
\end{tabular}}
\end{table}

With the image classification task, we utilized three benchmark datasets in Multi-task Learning, including Multi-MNIST \citep{sabour2017}, Multi-Fashion is generated similarly from the dataset \citep{xiao2017fashion}, and Multi-Fashion+MNIST \citep{lin2019pareto}. For each dataset, we have two tasks that ask us to categorize the top-left (task left) and bottom-right (task right) items (task right). Each dataset has 20,000 test set instances and 120,000 training examples; then, $10\%$ of the training data were split into the validation set. A result of the findings is shown in Table \ref{tb5}, we compare with PHN-EPO of \citep{navon2020learning}, and we evaluate over 25 preference vectors. The hypervolume's reference point is (2, 2).
\subsubsection{Scene Understanding}
\begin{table}[htbp]
\caption{Testing hypervolume on NYUv2 dataset.}
\label{tb4}
\centering
\resizebox{0.4\textwidth}{!}{\begin{tabular}{| c | >{\centering\arraybackslash}m{1.cm} |>{\centering\arraybackslash}m{2.cm} | }
\toprule
\multicolumn{3}{|c|}{NYUv2} \\ 
\midrule
{\bf Method} & HV $\Uparrow$& {\bf \revise{Training-time (hour.)}} \\
\midrule
PHN-EPO &  7.385 & 16.75 \\ 
PHN-LS &  9.204 & \bf 10.25  \\ 
\midrule
PHN-Cheby (\textbf{ours})  &  8.865 & 10.36 \\
PHN-Utility  (\textbf{ours})  & \bf 9.237 &  10.35 \\ 
\bottomrule
\end{tabular}}
\end{table}

The NYUv2 dataset \citep{silberman2012indoor} serves as the basis experiment for our method. This dataset is a collection of 1449 RGBD images of an indoor scene that have been densely labeled at the per-pixel level using 13 different classifications. We use this dataset as a 3 tasks MTL benchmark for normal surface prediction, depth estimation, and semantic segmentation. The results are presented in Table \ref{tb4} with (3, 3, 3) as hypervolume's reference point. Our method, PHN-Utility, achieves the best HV in NYUv2 dataset with a 10.25 hour Training-time training. We use W$\&$B \citep{wandb}, a tool used during the training process to efficiently keep tabs on experiments, modify and iterate on datasets, and assess model performance. Additional details about this experiment is available at Appendix and \url{https://api.wandb.ai/report/tuantran23012000/f0y7l0oi}.

\subsubsection{Multi-Output Regression.} 
We conduct experiments using the SARCOS dataset \citep{Sethu2000}  to illustrate the feasibility of our methods in high-dimensional space. The objective is to predict seven relevant joint torques from a 21-dimensional input space (7 tasks) (7 joint locations, seven joint velocities, seven joint accelerations). There are 4449 and 44484 examples in testing/training set. 10\% of the training data is used as a validation set. 

\begin{table}[ht]
\caption{Testing hypervolume on SARCOS dataset. }
\label{tb7}
\centering
\resizebox{0.4\textwidth}{!}{\begin{tabular}{| c | >{\centering\arraybackslash}m{1.cm} |>{\centering\arraybackslash}m{2.cm} | }
\toprule
\multicolumn{3}{|c|}{SARCOS} \\ 
\midrule
{\bf Method} & HV $\Uparrow$& {\bf \revise{Training-time (hour.)}} \\
\midrule
PHN-EPO &  0.855 & 1.73 \\ 
PHN-LS &  0.850 & 1.44  \\ 
\midrule
PHN-Cheby (\textbf{ours})  &  0.828 &  1.64\\
PHN-Utility  (\textbf{ours})  & 
\bf 0.856 &  \bf 1.42  \\ 
\bottomrule
\end{tabular}}
\end{table}

\subsubsection{Multi-Label Classification}
\begin{table}[ht]
\caption{Performance comparison of methods on the CelebA dataset.}
\label{tb6}
\centering
\resizebox{0.5\textwidth}{!}{\begin{tabular}{| c | >{\centering\arraybackslash}m{1.cm} | >{\centering\arraybackslash}m{1.cm} | >{\centering\arraybackslash}m{1.cm} |>{\centering\arraybackslash}m{1.cm} | >{\centering\arraybackslash}m{2.cm} |}
\toprule
\multicolumn{6}{|c|}{CelebA} \\ 
\midrule
{\bf Method} & mA  & Pre & Recall &  F1 & \bf \revise{Training-time (hour.)}\\
\midrule
PHN-EPO & 83.35  &  93.43 & 71.75 & 81.17 & 66.24\\ 
PHN-LS &  \bf 84.85 & 93.25 & \bf 75.16 & \bf 83.23 & 7.44\\ 
\midrule
PHN-Cheby (\textbf{ours})    & 81.78 & \bf 93.88 & 68.01 & 78.87 & \bf 7.20\\
PHN-Utility  (\textbf{ours})    &84.46  & 93.71&  73.87 & 82.62 & 7.44\\ 
\bottomrule
\end{tabular}}
\end{table}

Continually investigate Controllable Pareto Front Learning in large scale problem, we solve the problem of recognizing 40 facial attributes (40 tasks) in 200K face images on CelebA dataset \citep{liu2015deep} using a big Target network: Resnet18 (11M parameters) of \citep{he2016deep}. Due to very high dimensional scale (40 dimension), we sample randomly 1000 testing preference rays and evaluate methods by \textbf{mA, Pre, Recall, and F1 score} in Table \ref{tb6}. Additional details about this experiment is available at Appendix and \url{https://api.wandb.ai/report/tuantran23012000/xn2gxdwu}.

\subsection{Computational Analysis}
Experiments on \textbf{MOO problems} restated the efficacy of exact mapping in our framework. The MED scores in Table \ref{tb1}–\ref{tb3} are relatively small, indicating that the truth optimal solutions and the predictions of Controllable Pareto Front Learning methods are almost equivalent, particularly the Utility function, the novel scalarization of Pareto Hypernetwork. In the meanwhile, as shown in Figure \ref{fig3}(a) and Figure \ref{fig3}(c), although PHN-EPO and PHN-Cheby predictions are virtually identical, their essences are different. To solve the problem, EPO need to solve a linear programming every iteration to guarantee the generated solutions lie along the inverse preference vectors $\mathbf{r}^{-1}$, while the attribute is the nature of Chebyshev function. That is the reason why PHN-EPO takes many training time, especially on CelebA dataset (the MTL problem with 40 tasks).

\begin{table}[htbp]
\caption{\revise{Inference time on Multi-MNIST, Multi-Fashion, and Multi-Fashion+MNIST datasets by a random reference vector.}}
\label{tb8}
\centering
\resizebox{0.5\textwidth}{!}{\begin{tabular}{| c |>{\centering\arraybackslash}m{1.cm}| >{\centering\arraybackslash}m{2.cm} | >{\centering\arraybackslash}m{2.cm} | >{\centering\arraybackslash}m{2.cm} | }
\toprule
                &  &  {\bf Multi-MNIST} & {\bf Multi-Fashion} & {\bf Fashion-MNIST}    \\ \midrule  

{\bf Method} & params & Infer-time (min.) & Infer-time (min.) & Infer-time (min.)\\
\midrule
LS & 32K &3.5  & 3.3  & 3.2  \\ 
EPO & 35K & 8.1  & 8.3  & 8.2  \\ 
PMTL & 33K & 11.3  & 11.7  & 11.4 \\ 
PHN-EPO & 3.3M & 0.017 & 0.018  & 0.018  \\ 
PHN-LS  & 3.3M & 0.022   & 0.019  & 0.018  \\ 
\midrule
PHN-Cheby (\textbf{ours}) & 3.3M & \bf 0.016   & 0.018 & 0.019  \\
PHN-Utility  (\textbf{ours}) & 3.3M &  0.019  & \bf 0.017  & \bf 0.017 \\ 
\bottomrule
\end{tabular}}
\end{table}
\revise{The MTL experiments demonstrated the superior inference of hypernetwork-based methods. Table \ref{tb8} shows the inference time of the Controllable Pareto Front Learning framework, which maps a preference vector to a point on the Pareto Front corresponding. We define the inference time as the duration from the input of the reference vector until the algorithm output the Pareto optimal solution. Our proposed framework is faster 100 than prior methods such as LS, EPO, and PMTL, which must train a CNN model to obtain a solution corresponding to a preference vector. The main reason is that the framework is pre-trained (offline learning) with extended priority vector data instead of directly calculating it to find the Pareto optimal solution from a corresponding reference vector (online learning).}

\section{Conclusion and Future Work}
In this paper, we have proved the convergence of Completed Scalarization Functions for convex MOO problem in order to develop Controllable Pareto Front Learning, a brilliant framework to approximate the entire Pareto front, with a comprehensive mathematical explanation of why Hypernetwork may give a precise mapping between a preference vector and the corresponding Pareto optimal solution. \revise{However, in some cases, decision-makers cannot choose an explicit preference vector, and only express relative importance between objectives. We provide further discussions about this situation in the Appendix \ref{sec:orderconstr}}. Controllable Pareto Front Learning offers significant promise for real-world Multi-Objective Systems that require real-time control and can integrate with the newest research in Machine Learning such as Graph Neural Networks, Knowledge Graph \citep{pham2021hierarchical, pham2022proposal}. Future work might relate to the provision of the Scalarization functions on the premise that they are monotonically increasing and pseudoconvex. In addition, we also provide the option to expand the class of the Scalarization function, which can be found in Appendix \ref{appendixB.3}. 
\section*{Acknowledgments}
This work was supported by Vietnam Ministry of Education and Training under  Grant number [B2023-BKA-07].
\printcredits

\appendix
\section{PREFERENCE ORDER CONSTRAINTS}
\label{sec:orderconstr}
\revise{In many cases, decision-makers are unable to provide an exact preference vector and instead offer a general description of preference-order constraints, often in the form of a rank list indicating the relative importance of different tasks. Notably, according to \cite{abdolshah2019multi}, preference-order constraints are closely linked to the definition of Pareto Stationary points (Definition \ref{def:paretostationary}), which satisfy the stationary equation where the weighted sum of gradients equals zero. Specifically, a point satisfies the condition of task 1 being more important than task 0 if the weight assigned to task 1 in the Pareto stationary equation is greater than the weight assigned to task 0. Consequently, this constraint creates a subset within the Pareto set, as depicted by the red region in figure \ref{fig:ordercons}.
}
\begin{figure}[h]
    \centering
    \includegraphics[width=0.4\textwidth]{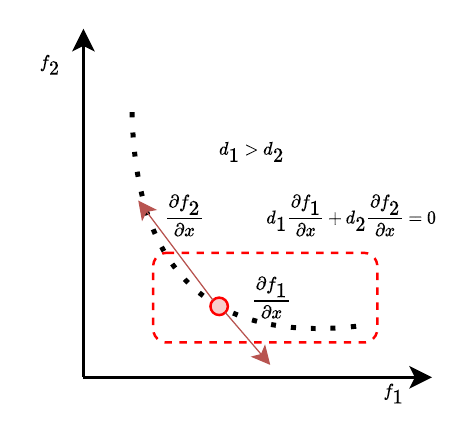}
    \caption{The weight $d_1, d_2$ of stationary equation representing that objective 1 is more important than objective 2}
    \label{fig:ordercons}
\end{figure}

On the other hand, our framework is designed to discover the entire Pareto set. Therefore, after optimizing the hypernetwork using the Equation \eqref{CS-PHN} described in the paper, we can identify the subset that aligns with the preference-order constraints specified by the decision-maker. As a result, our model remains valid and capable of accommodating these constraints effectively.
\section{EXPERIMENTAL DETAILS}\label{ED}
\subsection{MOO problems}
The experiments MOO were implemented on a computer with Intel(R) Core(TM) i7-10700, 64-bit CPU $@ 2.90$GHz, and RAM 16GB.

\textbf{Example 7.1:} In this example, we use a multi-layer perceptron (MLP) as the Hypernetwork model, which has the following structure:
\begin{align*}
    h(\mathbf{r},\phi): &\textbf{Input}(\mathbf{r}) \to \textbf{Linear}(2, 100) \to\textbf{ReLU}\\
    &\to \textbf{Linear}(100, 100)
\to \textbf{ReLU} \\
&\to \textbf{Linear}(100, 1) \to \textbf{Sigmoid}\to \textbf{Output}(\mathbf{x}_{\mathbf{r}}).
\end{align*}
We utilize Hypernetwork to generate an approximate efficient solution from a reference vector created by Dirichlet distribution with $\alpha = 0.6$. We trained all completion functions using an Adam optimizer \citep{kingma2014adam} with a learning rate of $1e-3$ and $10000$ iterations. In the test phase, we sampled $[5,10,50,100,300,600]$ preference vectors based on the method of \citep{das1997closer}. Besides, we also illustrated target points and predicted points from the pre-trained Hypernetwork in Figure \ref{fig3} (top).

\textbf{Example 7.2:} The structure of Hypernetwork has the following structure:
\begin{align*}
    h(\mathbf{r},\phi): &\textbf{Input}(\mathbf{r}) \to \textbf{Linear}(2, 100) \to\textbf{ReLU}\\
    &\to \textbf{Linear}(100, 100)
\to \textbf{ReLU} \\
    &\to \textbf{Linear}(100, 100)
\to \textbf{ReLU} \\
&\to \textbf{Linear}(100, 2) \to\textbf{Output}(\mathbf{x}_{\mathbf{r}}).
\end{align*}
Unlike Example 7.1, in this example, we did not have the Sigmoid layer before the Output layer but instead direct $\mathbf{x}_{\mathbf{r}}$ from the last Linear layer. We use five preference vectors (0.2, 0.8), (0.4, 0.6),(0.3,0.7), (0.5,0.5, (0.7,0.3), (0.6,0.4), (0.9,0.1) to visualize minimized and valued objectives predicted from the Hypernetwork in Figure \ref{fig3} (middle).

\textbf{Example 7.3:} To satisfy the equation constraint, in the last layer, we pass the output of the previous layer via a non-linear layer (Square Root Layer), the following structure:
\begin{align*}
    h(\mathbf{r},\phi): &\textbf{Input}(\mathbf{r}) \to \textbf{Linear}(3, 100) \to\textbf{ReLU}\\
    &\to \textbf{Linear}(100, 100)
\to \textbf{ReLU}\\
 & \to\textbf{Linear}(100, 100)\to \textbf{ReLU}\\
 &\to\textbf{Linear}(100, 100) \to \textbf{ReLU}\\
&\to \textbf{Linear}(100, 3) \to \textbf{Softmax}\\
&\to\textbf{Sqrt}\to
\textbf{Output}(\mathbf{x}_{\mathbf{r}}).
\end{align*}
Base on \citep{das}, we define subregions $\Omega_i$ by points $u=(u_1, \dots, u_p) \in U \subset \mathcal{P}$ such that:
\begin{align*}
    u_1 \in \{0, \delta, 2\delta, \dots, ..., 1\} \text{ s.t }  \frac{1}{\delta}=k \in \mathbb{N}^* 
\end{align*}
If $1<j<m-1$, $n_i = \frac{\delta}{u_i}, 1 \leq i<j-1$, we have:
\begin{align*}
    u_j \in \{0, \delta, \dots, (p-\sum_{i=1}^{j-1}n_i)\delta\} \text{, }  u_m = 1 - \sum_{j=1}^{m-1}u_j
\end{align*}
We used the Delaunay triangulation algorithm \citep{Edelsbrunner1985VoronoiDA} for generating preference vectors in 3D space. Then we conduct experiments on those vectors and illustrate results in Figure \ref{fig3} by setting $k = 10$. Then, the number of points $u \in U$ in the space $\mathcal{P}$ is $\Big (^{m+k-1}_k \Big) = 66$. However, the Pareto optimal solution corresponds to the intersection of the preference vector $\mathbf{r}^{-1}$ and the Pareto front for PHN-Cheby and PHN-EPO. Because of this, to avoid choosing vectors with coordinate values of 0, we only use vectors with different coordinates better than 0.16. Ten preference vectors in the set of vectors will correspond to the value $k=10$.

\subsection{Multi-task Learning problems}
When experimenting, we set up MTL problems on a Linux server with Intel Xeon(R), Silver 4216 64-bit CPU $@ 2.10$GHz, RAM 64GB, and VGA NVIDIA Tesla T4 16GB. 

\textbf{Image Classification:} We use Multi-LeNet \citep{sener2018multi} to train Hypernetwork. The network's architecture is illustrated in Figure \ref{fig8}. In this experiments, we use Cross-Entropy loss, an Adam optimizer with a batch size of 256 and a learning rate of $1e-4$ over 150 epochs. The validation set is then used to decide the optimal Hypervolume for each approach. We evaluate the Hypernetwork using 25 random preference vectors. The results on the test set are visualized Figure \ref{acc}. 

\textbf{Scene Understanding:} 
When it comes to the Target network, we make use of a SegNet architecture \citep{badrinarayanan2017segnet}. Each method is trained for 100 epochs using the Adam optimizer with the following parameters: an initial learning rate of $1e-4$, $\alpha = 0.2$, and a batch size of 14. With 795 train images, 436 validation images, and 218 test images, using the \citep{das} algorithm, 23 evaluated and 28 tested rays were created to achieve the greatest possible hypervolume on the validation set. We use the chunking technique \citep{hypernetwork, von2019continual} to generate and reduce the number of parameters for the Hyper-SegNet architecture as Figure \ref{fig15}.

\textbf{Multi-Output Regression:} The target network is Multi-Layer Perceptron which has 21 input variables, 3 layers with 128 units, activation ReLU, and a fully-connected layer. We parameterize $h(\mathbf{r}, \phi)$ on Hypernetwork using a feed-forward network with a range of outputs. Each output generates a unique weight tensor for the target network. To construct shared features, the input $\mathbf{r}$ is first mapped to the 100-dimensional space using a multi-layer perceptron network. These features are passed across fully connected levels in the target network to produce a weight matrix for each layer. Using an Adam optimizer with a learning rate of $1e-3$ and a batch size of 512, we train all algorithms. We multiply the learning rate by $\sqrt{2}$ if the algorithm fails to update the best model for 40 consecutive epochs. $\alpha$ was set to $0.2$. By dividing them by the quantile of 0.9, we normalize the output variables.

\textbf{Multi-Label Classification:} We conduct an experiment to the CelebA dataset, which includes more than $200000$ face images annotated with 40 binary labels, we use a HyperResnet-18 to generate the weight of layers for TargetResnet-18 as Figure \ref{fig12}. In this experiment, we split the dataset into 162770 train images with sizes of $64\times 64$, 19867 valid images, and 19962 test images. We train using 100 epochs, a $5e-4$ learning rate, and the best weight is stored by evaluating five preference vectors with a Mean Accuracy score. Besides, we report a performance comparison of methods on the CelebA dataset in Table \ref{tb6}. Performance on four metrics, Mean Accuracy, Precision, Recall, and F1 score, are evaluated by 25 random preference vectors.

\subsection{Hypervolume Reference Point}
\label{sub:hv}
\revise{
The $m$-dimensional reference point serves as a crucial hyperparameter in the hypervolume indicator. In scenarios where the true Pareto front is unknown, and no specific objectives are prioritized, the reference point should be set equally on all coordinates. Fortunately, selecting an appropriate reference point that facilitates well-spread predictions to approximate the entire Pareto front is not a complex task. Often, choosing a reference point with coordinates higher than those obtained from a random initialization of network losses proves to be sufficient for evaluating outcomes.
}
\section{ADDITIONAL EXPERIMENTS}
\subsection{Approximation Pareto front.}\label{appendixB.1}
According to Section \ref{sec4}, the Pareto front may include an endless number of optimum values with varying trade-offs. In addition, the Pareto optimal solutions lack a definitive ranking. A Pareto Front Learning model for the Completed Scalarization function should be capable of approximating the whole Pareto Front and should make it simple to investigate any trade-off. In the training procedure, we employ Dirichlet distribution to produce preference vectors (rays) with $\alpha = 0.2$, the number of rays having the effect of approximating the complete Pareto set in order to construct the entire Pareto front. As Figure \ref{fig5}, Figure \ref{fig6}, and Figure \ref{fig7}, when there are a more significant number of rays used for learning, the Pareto front becomes more covered. 
\subsection{Gradient Explainer Hypernetwork}\label{appendixB.2}
Figure \ref{right} and Figure \ref{left} display the influence of preference vectors on Hypernetwork' prediction of the left (right) digit in the Multi-MNIST dataset. For this purpose, we additionally use the SHAP framework \citep{lundberg2017unified}. When the prediction of a class is a higher probability, the pixels will be red, and when it is lower, the pixels will be blue.

In addition, we apply the GradientExplainer to the CelebA dataset to explain the model's behavior as it is being trained. As seen in Figure \ref{gradient}, the probability densities of the red entry points are distributed most heavily in the areas of the face that contain the characteristics whose probability we want to predict.
\subsection{Design Scalarization Functions}\label{appendixB.3}
From Proposition 5.21 of \citep{dinh2005generalized}, scalarization function $s$ be a monotonically increasing function, then every optimal solution of Problem \eqref{CP} be an efficient solution to Problem \eqref{MOP}. However, it can not be a completed scalarization function, hence, we experiment the extended Scalarization functions based on theory of \citep{mahapatra2021exact, kamani2021pareto, chugh2020scalarizing}.

\textbf{Kullback–Leibler divergence function.} KL-divergence is a distance measurement function to minimize distance of two distributions:
\begin{align*}\tag{KL}
    s_{KL}(\mathcal{F},\mathbf{r}) = \sum_{i=1}^m \sigma_i(\mathcal{F},\mathbf{r})\log\left(m\sigma_i(\mathcal{F},\mathbf{r})\right),
\end{align*}
where $\sigma_i(\mathcal{F},\mathbf{r}) = \dfrac{e^{r_i f_i}}{\sum_{i=1}^m e^{r_i f_i}}, \forall i=1,\dots,m$. $s_{KL}$ is a strongly convex on $Y$.

\textbf{Cauchy–Schwarz function.} Cauchy is a proportionality measuring function:
\begin{align*}\tag{Cauchy}
    s_{Cauchy}\left(\mathcal{F},\mathbf{r}^{-1}\right) = 1-\dfrac{\left<\mathcal{F},\mathbf{r}^{-1}\right>^2}{\Vert\mathcal{F}\Vert^2\Vert\mathbf{r}^{-1}\Vert^2},
\end{align*}
where $s_{Cauchy}$ was implied by the Cauchy-Schwarz inequality pertaining to non-zero vectors $\mathcal{F}, \mathbf{r}^{-1}\in\mathbb{R}^m_{+}$.

\indent\textbf{Cosine-Similarity function.} Cosine is a form of the Utility function, which has the following formula:
\begin{align*}\tag{Cosine}
    s_{Cosine}\left(\mathcal{F},\mathbf{r}\right) = -\dfrac{\left<\mathcal{F},\mathbf{r}\right>}{\Vert\mathcal{F}\Vert\Vert\mathbf{r}\Vert}.
\end{align*}
\indent\textbf{Log function.} Log is also a form of the Utility function:
\begin{align*}\tag{Log}
    s_{Log}\left(\mathcal{F},\mathbf{r}\right) = \sum_{i=1}^m r_i\log{(f_i + 1)}.
\end{align*}
\indent\textbf{Prod function.} Prod is weighted product function, which was also called as product of powers:
\begin{align*}\tag{Prod}
    s_{Prod}\left(\mathcal{F},\mathbf{r}\right) = \displaystyle\prod_{i=1}^m\left((f_i+1)^{r_i}\right).
\end{align*}
\indent\textbf{AC function.} AC is augmented chebyshev function, which was added by LS function:
\begin{align*}\tag{AC}
    s_{AC}\left(\mathcal{F},\mathbf{r}\right) = \underset{i=1,\dots,m}{\max}\{r_i f_i\} + \rho\sum_{i=1}^m r_i f_i,
\end{align*}
where $\rho >0$.

\indent\textbf{MC function.} MC is modified chebyshev function, which has form as:
\begin{align*}\tag{MC}
    s_{MC}\left(\mathcal{F},\mathbf{r}\right) = \underset{i=1,\dots,m}{\max}\{r_i f_i + \rho\sum_{i=1}^m r_i f_i\},
\end{align*}
where $\rho >0$.

\indent\textbf{HVI function.} HVI is hypervolume indicator function, which was considered by \citep{hoang2022improving}, has the following formula:
\begin{align*}\tag{HVI}
    s_{HVI}\left(\mathcal{F},\mathbf{r}\right) = -HV(\mathcal{F},\mathbf{ref}) + \rho s_{Cosine}\left(\mathcal{F},\mathbf{r}\right),
\end{align*}
where $s_{Cosine}\left(\mathcal{F},\mathbf{r}\right) = -\dfrac{\left<\mathcal{F},\mathbf{r}\right>}{\Vert\mathcal{F}\Vert\Vert\mathbf{r}\Vert}$, $\mathbf{ref}$ is a reference point, and $\rho >0$.

We simple to verify that all of the consideration Scalarization functions satisfy monotonically increasing assumption on $Y$, the testing procedure yielded outcomes that give us optimism in Table \ref{tb10}.

\begin{table}[htbp]
\caption{We sample $1000$ preference vectors and evaluate $50$ random vectors follow-up time of 30 executions.}
\label{tb10}
\centering
\resizebox{0.5\textwidth}{!}{\begin{tabular}{|c|c|c|c|}
\toprule
 &  Example 7.1 &  Example 7.2 &  Example 7.3  \\

 \bf Method & MED $\Downarrow$ & MED $\Downarrow$& MED$\Downarrow$\\ 
 \midrule
PHN-EPO & $0.0088\pm  0.0006$ & $0.0017\pm  0.0002$ &  $0.0808\pm  0.0052$\\
\midrule   
PHN-LS  & $0.0042\pm  0.0008$ & $0.0017\pm  0.0002$ & $0.0494\pm  0.0043$\\ 
\midrule
PHN-Cheby & $0.0084\pm  0.0005$  &  $0.0019\pm  0.0002$ & $0.0395\pm  0.0036$\\ 
\midrule
PHN-Utility ($ub = 2.01$) & $0.0025\pm  0.0003$  &  $\bf 0.0013\pm  0.0001$ &  $\bf 0.0201\pm  0.0022$ \\ 
\midrule
PHN-KL   & $0.0052\pm  0.0003$  & $0.0124\pm  0.0002$ &  $0.0373\pm  0.0039$\\ 
\midrule
PHN-Cauchy & $0.0037\pm  0.0003$ & $0.0232\pm  0.0007$ &  $0.0642\pm  0.0025$\\
\midrule
PHN-Cosine & $0.0051\pm  0.0007$ & $0.0259\pm  0.0017$ &  $0.0396\pm  0.0028$\\
\midrule
PHN-Log & $0.0321\pm  0.0343$ & $0.0031\pm  0.0002$ &  $0.0225\pm  0.0035$\\
\midrule
PHN-Prod & $0.0432\pm  0.0412$ & $0.0084\pm  0.0008$ &  $0.0385\pm  0.0042$\\
\midrule
PHN-AC ($\rho = 0.0001$) & $0.0052\pm  0.0005$ & $0.0091\pm  0.0014$ &  $0.0222\pm  0.0039$\\
\midrule
PHN-MC ($\rho = 0.0001$) & $0.0078\pm  0.0008$ & $0.0136\pm  0.0008$ &  $0.0267\pm  0.0044$\\
\midrule
PHN-HVI ($\rho = 100, heads = 8$) & $\bf 0.0017\pm \bf 0.0002$ & $0.0092\pm  0.0011$ &  $0.0303\pm  0.0087$\\
\bottomrule
\end{tabular}}

\end{table}
\subsection{Non-Convex MOO Problems}\label{appendixB.4}
\label{sec:non-convex}
\revise{We experiment with the additional test problems, including ZDT1-2 \citep{zitzler2000comparison}, and DTLZ2 \citep{deb2002scalable}. Dimensions and attributes for those test problems were illustrated in Table \ref{tb9}. We use Algorithm \ref{alg:hypermoo} to train Hypernetwork the following structure:
\begin{align*}
     h(\mathbf{r},\phi): &\textbf{Input}(\mathbf{r}) \to \textbf{Linear}(m, 100) \to\textbf{ReLU}\\
    &\to \textbf{Linear}(100, 100)
\to \textbf{ReLU} \\
&\to \textbf{Linear}(100, n) \to \textbf{Sigmoid}\to \textbf{Output}(\mathbf{x}_{\mathbf{r}}),
\end{align*}
with $100000$ iterations, $\alpha = 0.6$, 0.001 learning rate, and Adam optimizer. With the convex Pareto-optimal set in the test problem ZDT1, scalarization functions make a discrepancy between the predicted and optimal solutions. As shown in Figure 5, linear scalarization and utility functions perform poorly with the non-convex Pareto-optimal set in the test problems ZDT2 and DTLZ2.}
\begin{table}[htbp]
\caption{\revise{Dimensions and attributes for the MOO problems.}}
\label{tb9}
\centering
\resizebox{0.5\textwidth}{!}{\begin{tabular}{c|c|c|c|c}
\toprule
Problem & n & m & Objective function & Pareto-optimal\\
 \midrule
ex7.1 & 1 & 2 & convex & convex \\
ex7.2 & 2 & 2 & convex & convex \\
ex7.3 & 3 & 3 & convex & convex \\
ZDT1 & 30 & 2 & non-convex & convex\\
ZDT2 & 30 & 2 & non-convex & non-convex\\
DTLZ2 & 10 & 3 & non-convex & non-convex\\
\bottomrule
\end{tabular}}
\end{table}

\revise{Our experimental results reveal that the hypernet falls short in accurately approximating non-convex functions (ZDT1, ZDT2, and DTLZ2). Specifically, PHN-LS and PHN-Utility encounter challenges in handling Concave-Shape Pareto Fronts, leading to poor performance. Nevertheless, it's worth noting that PHN-Cheby and PHN-EPO continue to perform well in these scenarios.}
\subsection{Controllable Pareto Front Learning with no shared optimized variables}
\revise{
Consider the problem when the two objectives in a synthetic scenario are independent of each other, i.e. the entire Pareto front collapses to one single point:
\begin{align*}
\min & \left\{x_1,x_2\right\}\\
\text{s.t. } & 0\le x_1,x_2\le 1.
\end{align*}
We use Hypernetwork the following structure:
\begin{align*}
     h(\mathbf{r},\phi): &\textbf{Input}(\mathbf{r}) \to \textbf{Linear}(m, 100) \to\textbf{ReLU}\\
    &\to \textbf{Linear}(100, 100)
\to \textbf{ReLU} \\
&\to \textbf{Linear}(100, n) \to \textbf{Sigmoid}\to \textbf{Output}(\mathbf{x}_{\mathbf{r}}),
\end{align*}
}

\revise{If the functions $f_1$ and $f_2$ are independent, meaning they do not share optimized variables, specific ML Pareto solvers like MOO-MTL (\cite{sener2018multi}), PMTL (\cite{lin2020controllable}), or EPO (\cite{mahapatra2021exact}) encounter issues as they rely on a common descent direction for shared variables. Besides, those algorithms did not perform the projection on the constraint space at each iteration. Consequently, these solvers return NaN, a value outside the feasible set or crash in this example. In contrast, our framework, Controllable PFL, operates by optimizing the parameters of the hypernetwork rather than the shared variables, and utilizes the constraint function, such as the Sigmoid function, to control output. This fundamental distinction allows our framework to perform effectively even in such unique scenarios. Interestingly, PHN-EPO also does not fail because it directly applies the EPO solver to the parameters of the hypernetwork.}

\revise{Experiment results of this special example are shown in the Figure \ref{fig_ex4} and \ref{fig_ex4_1}. Our framework Controllable PFL is still work well, and return the optimal value $(0, 0)$ for all preference vectors.}

\begin{figure*}[ht]
     \centering
        \begin{subfigure}[b]{0.24\textwidth}
         \centering
         \includegraphics[width=\textwidth]{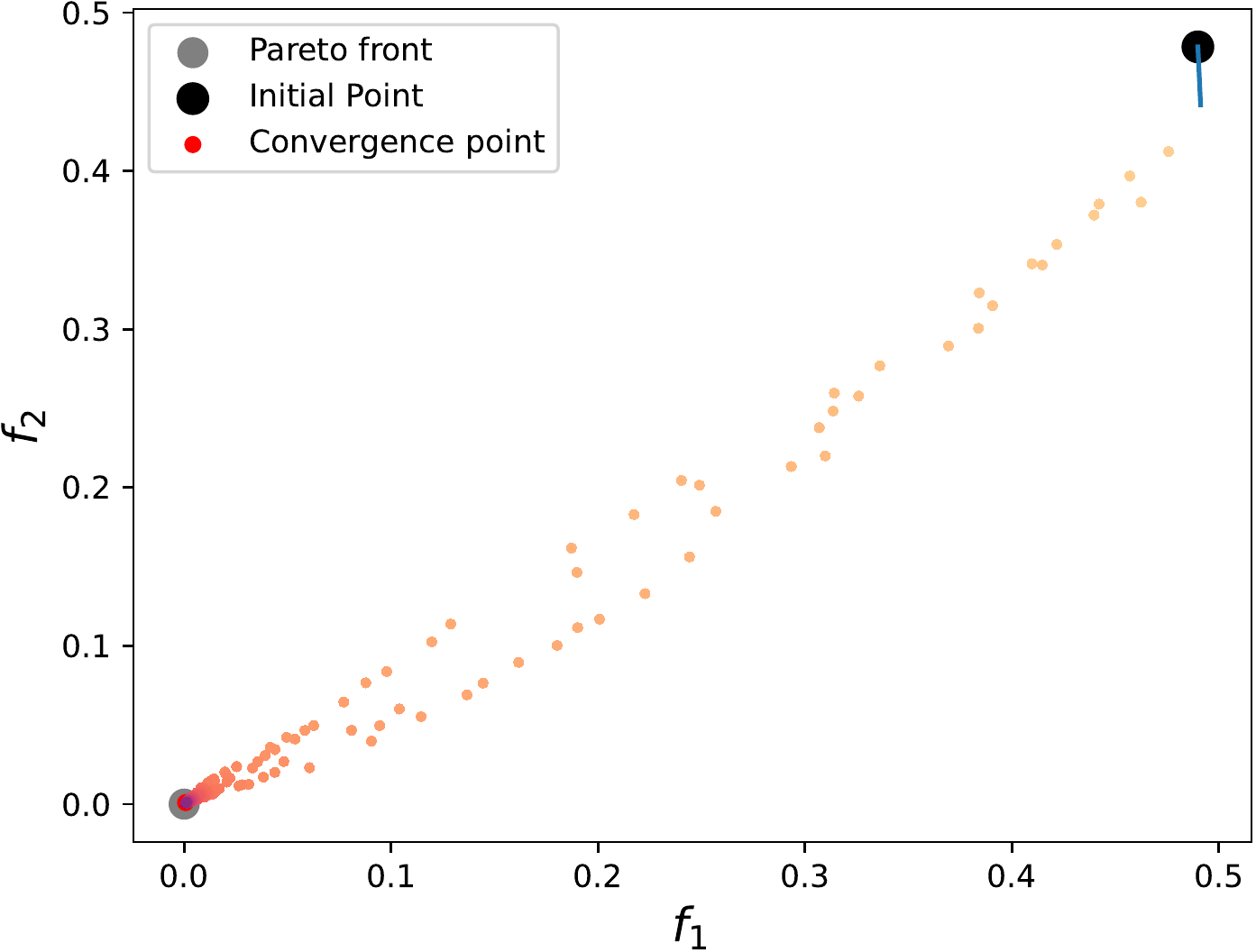}
         \caption{PHN-EPO}
     \end{subfigure}
     \hfill
     \begin{subfigure}[b]{0.24\textwidth}
         \centering
             \includegraphics[width=\textwidth]{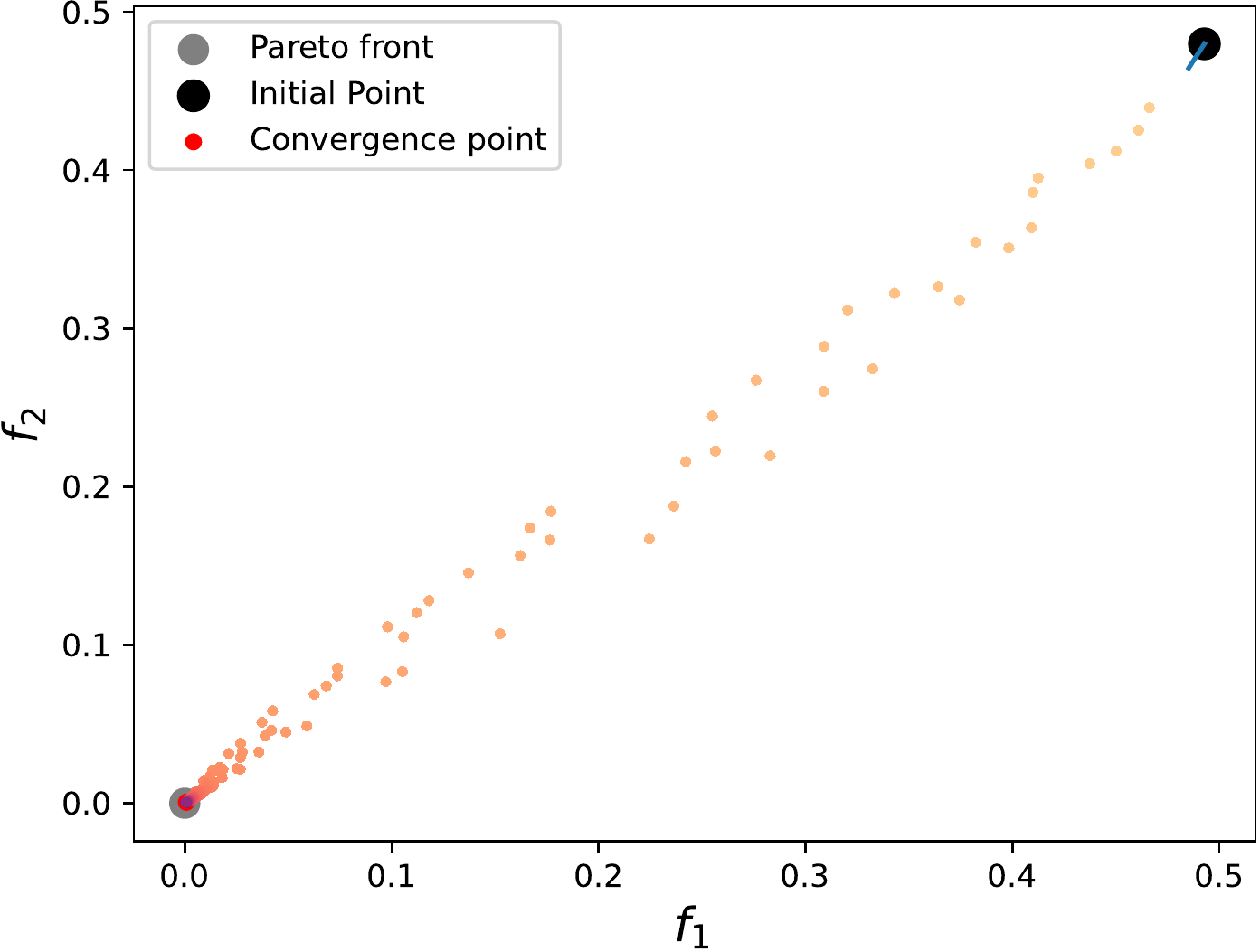}
         \caption{PHN-LS}
     \end{subfigure}
     \hfill
     \centering
     \begin{subfigure}[b]{0.24\textwidth}
         \centering
         \includegraphics[width=\textwidth]{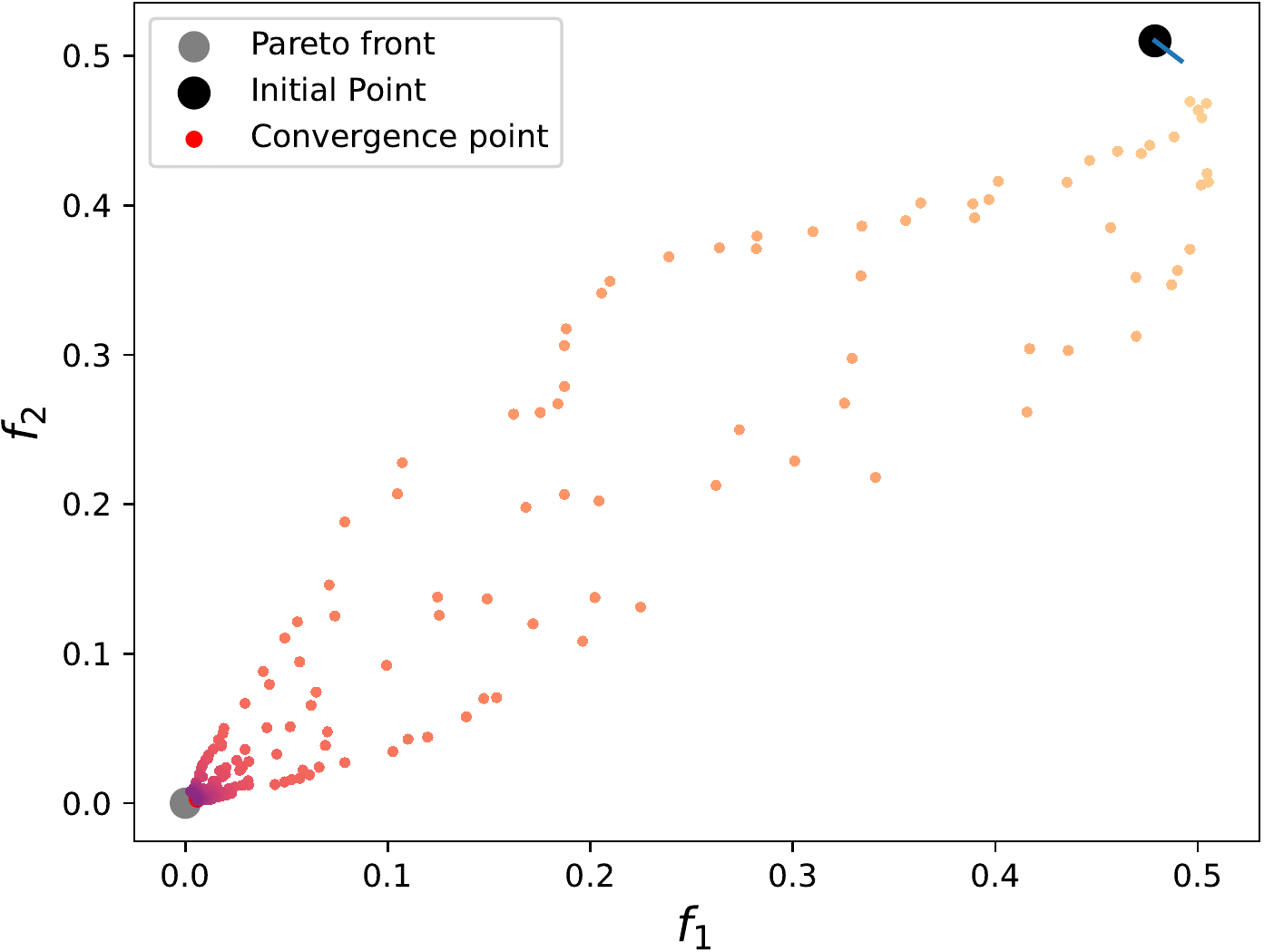}
         \caption{PHN-Cheby}
     \end{subfigure}
     \hfill
     \begin{subfigure}[b]{0.24\textwidth}
         \centering
         \includegraphics[width=\textwidth]{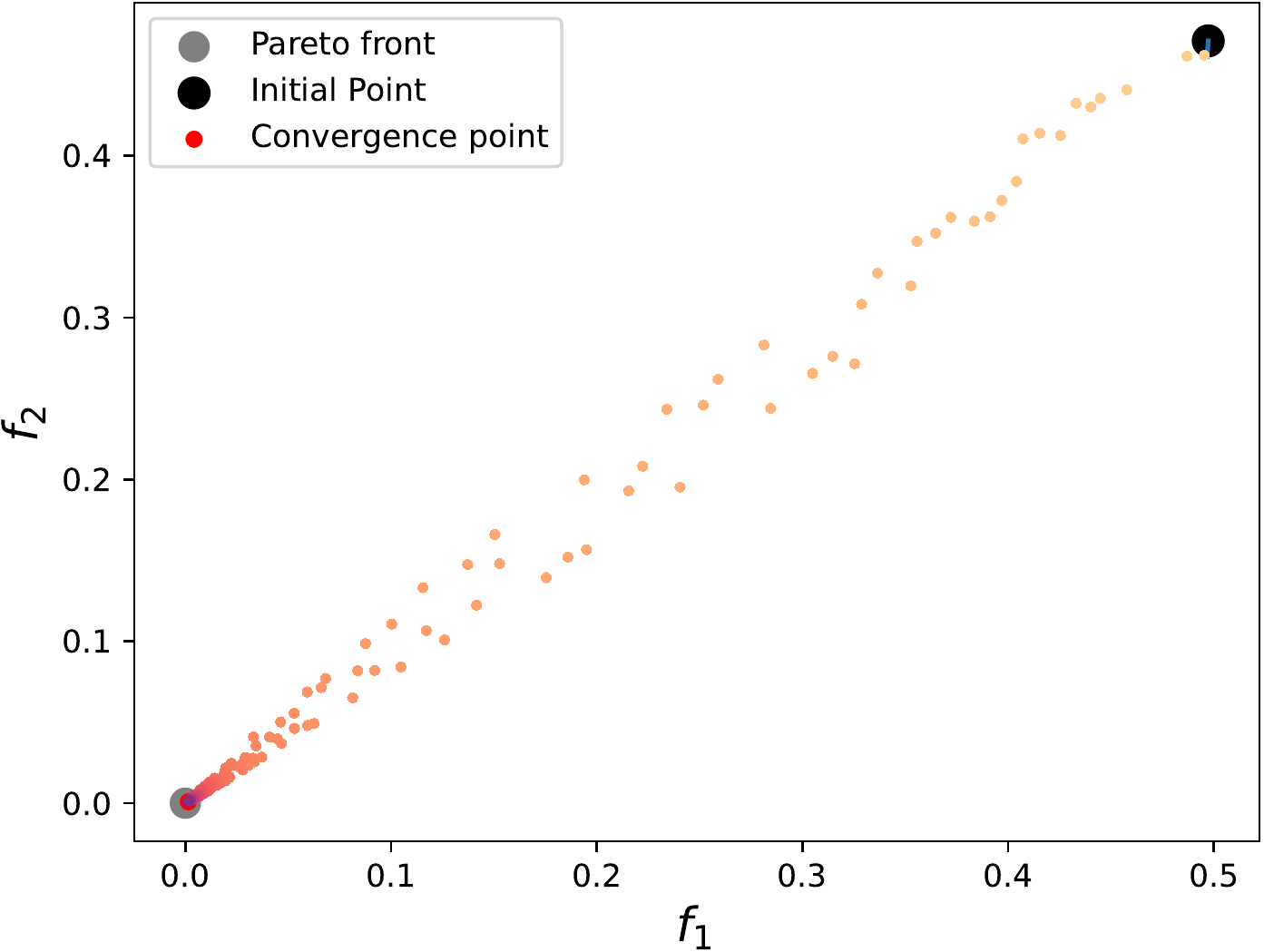}
         \caption{PHN-Utility}
     \end{subfigure}

      \caption{\revise{The convergence trajectory of generated solutions in the optimization process. Light colors are points generated in early iterations, dark colors are points generated in later iterations.}}
      \label{fig_ex4}
\end{figure*}
\begin{figure*}[ht]
     \centering
        \begin{subfigure}[b]{0.3\textwidth}
         \centering
         \includegraphics[width=\textwidth]{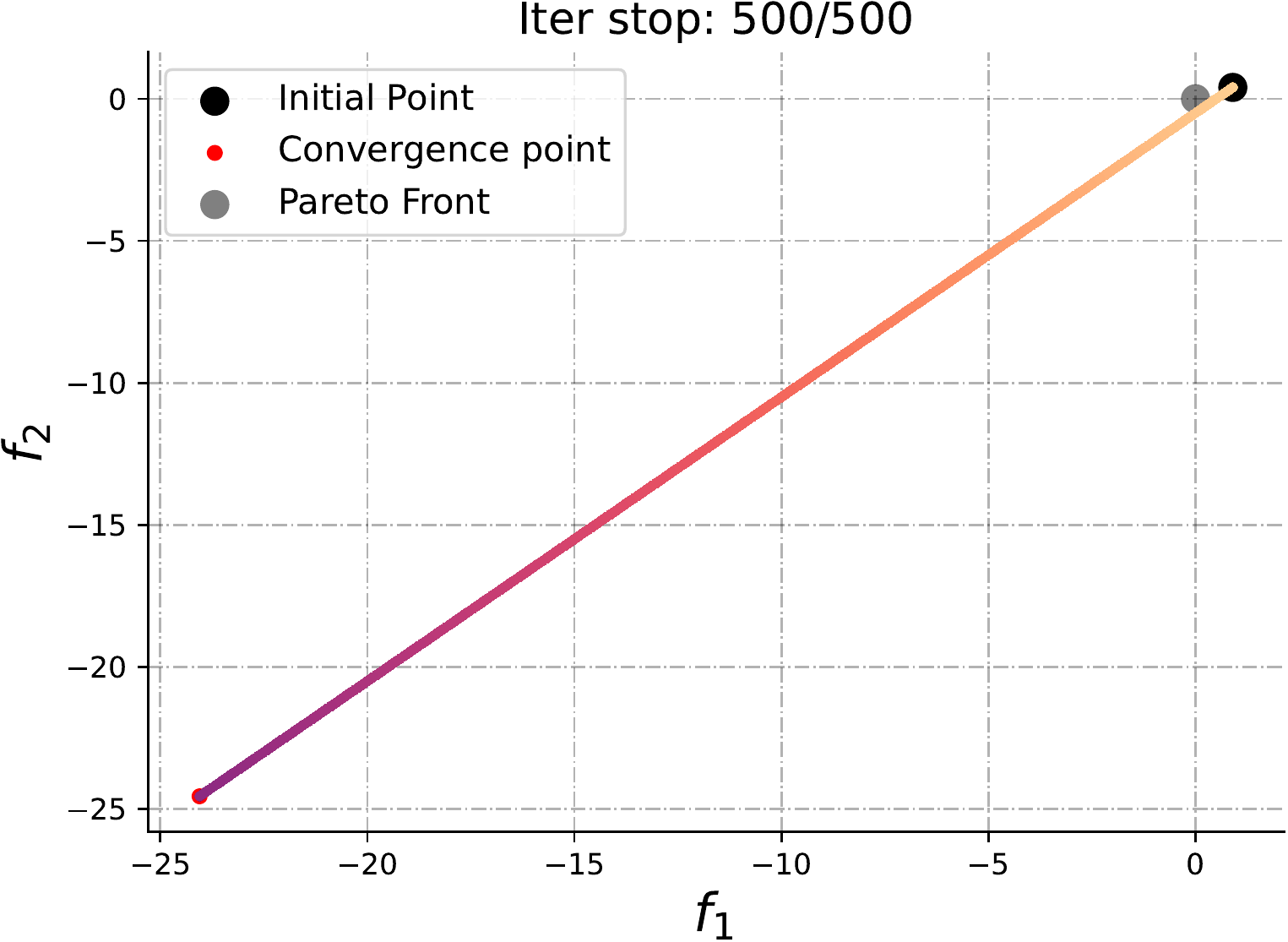}
         \caption{MOO-MTL}
     \end{subfigure}
     \hfill
     \begin{subfigure}[b]{0.3\textwidth}
         \centering
             \includegraphics[width=\textwidth]{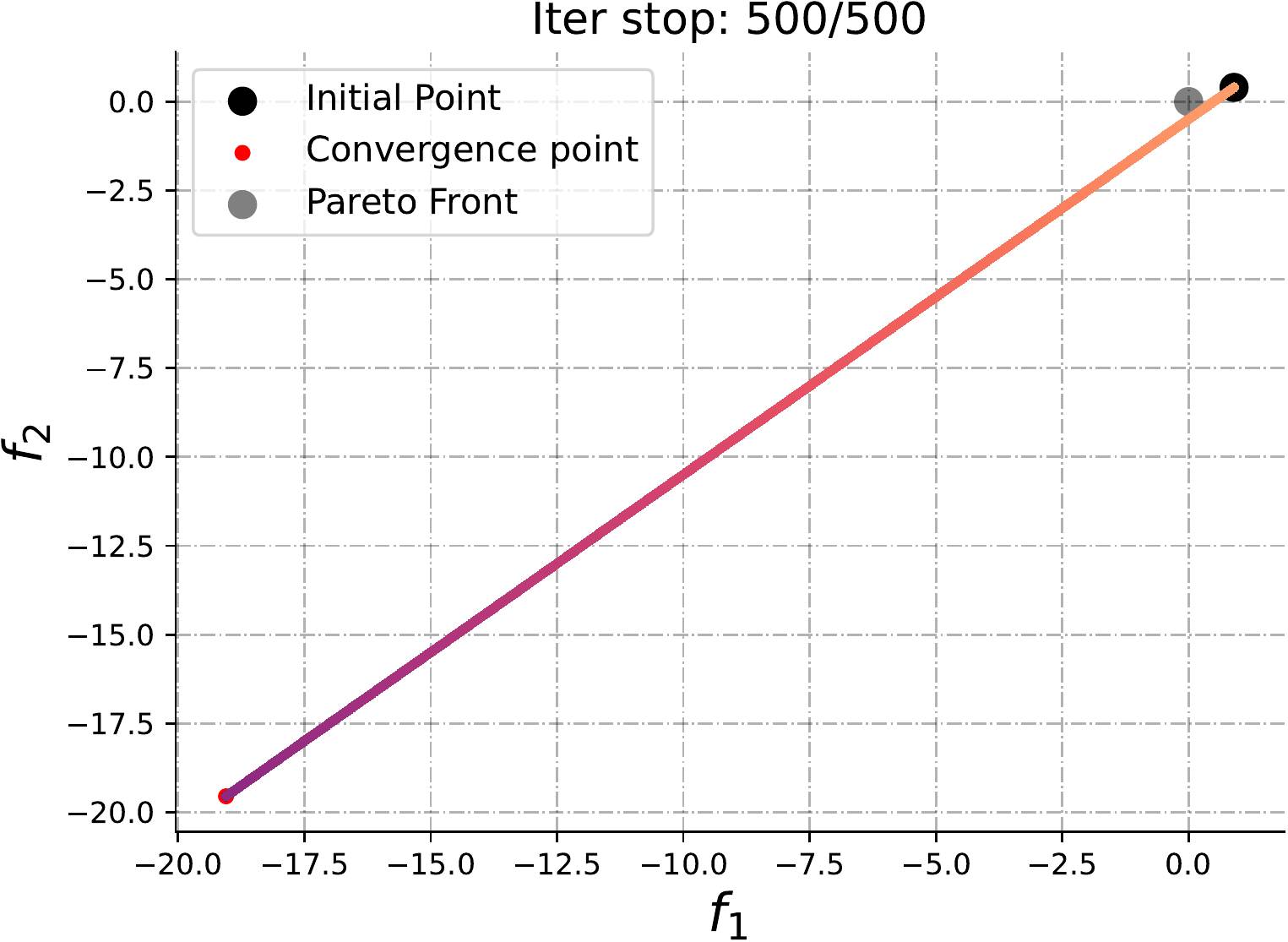}
         \caption{PMTL}
     \end{subfigure}
     \hfill
     \centering
     \begin{subfigure}[b]{0.3\textwidth}
         \centering
         \includegraphics[width=\textwidth]{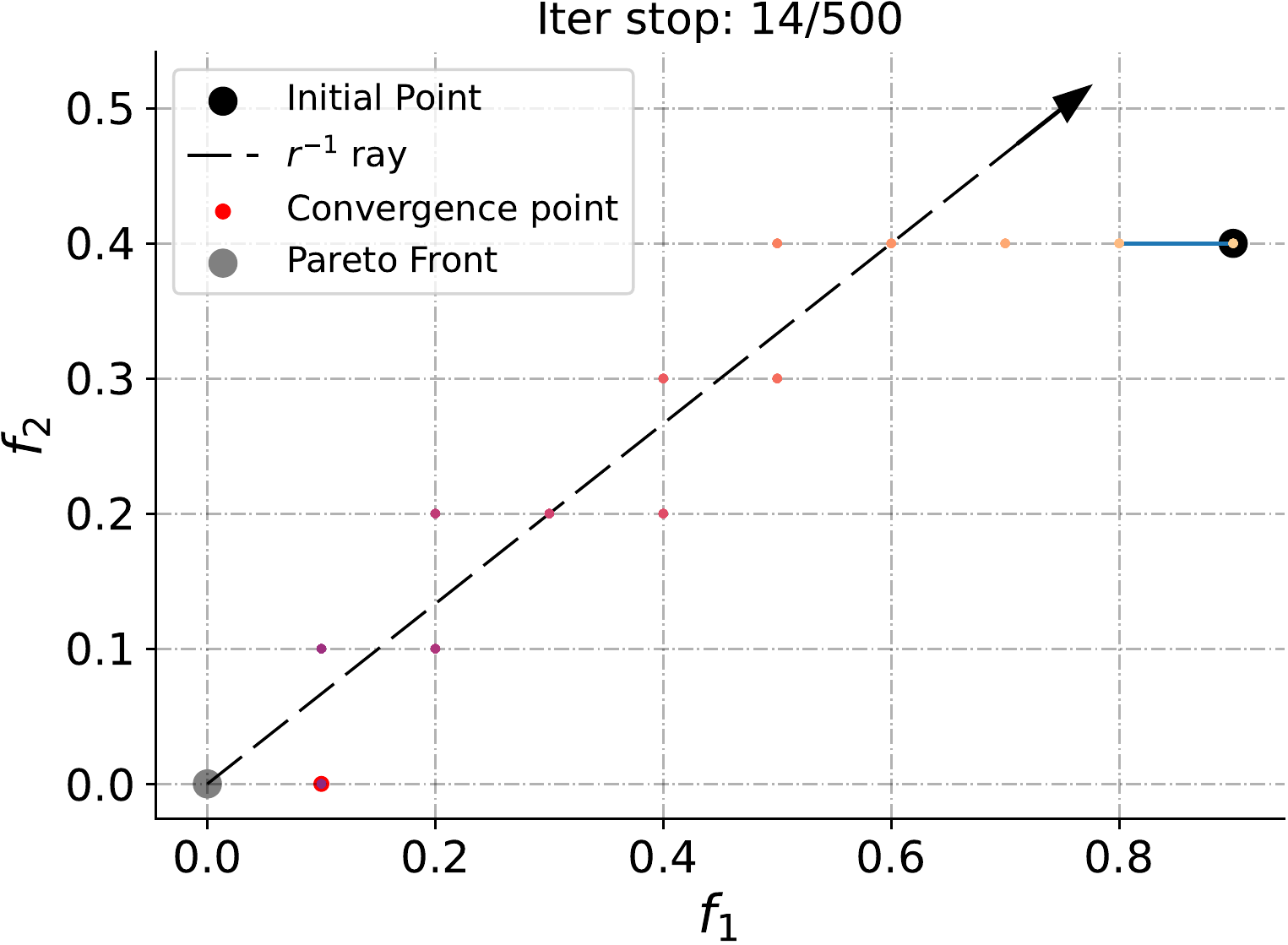}
         \caption{EPO}
     \end{subfigure}
      \caption{\revise{When the functions $f_1$ and $f_2$ are independent, ML Pareto solvers like MOO-MTL, PMTL, or EPO return NaN, a value outside the feasible set, or crash.}}
      \label{fig_ex4_1}
\end{figure*}
\begin{figure*}[ht]
     \centering
        \begin{subfigure}[b]{0.24\textwidth}
         \centering
         \includegraphics[width=\textwidth]{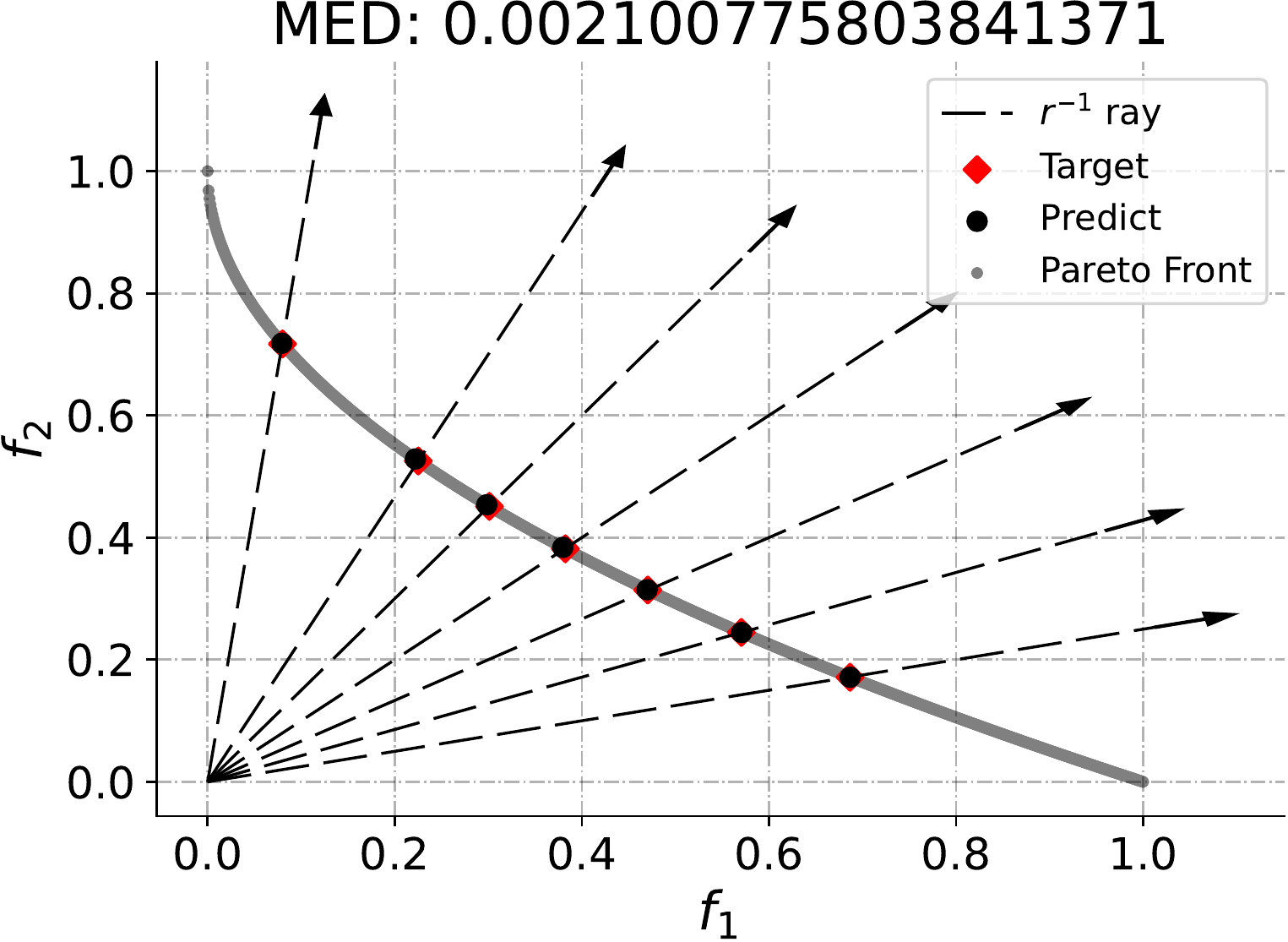}
     \end{subfigure}
     \hfill
     \begin{subfigure}[b]{0.24\textwidth}
         \centering
             \includegraphics[width=\textwidth]{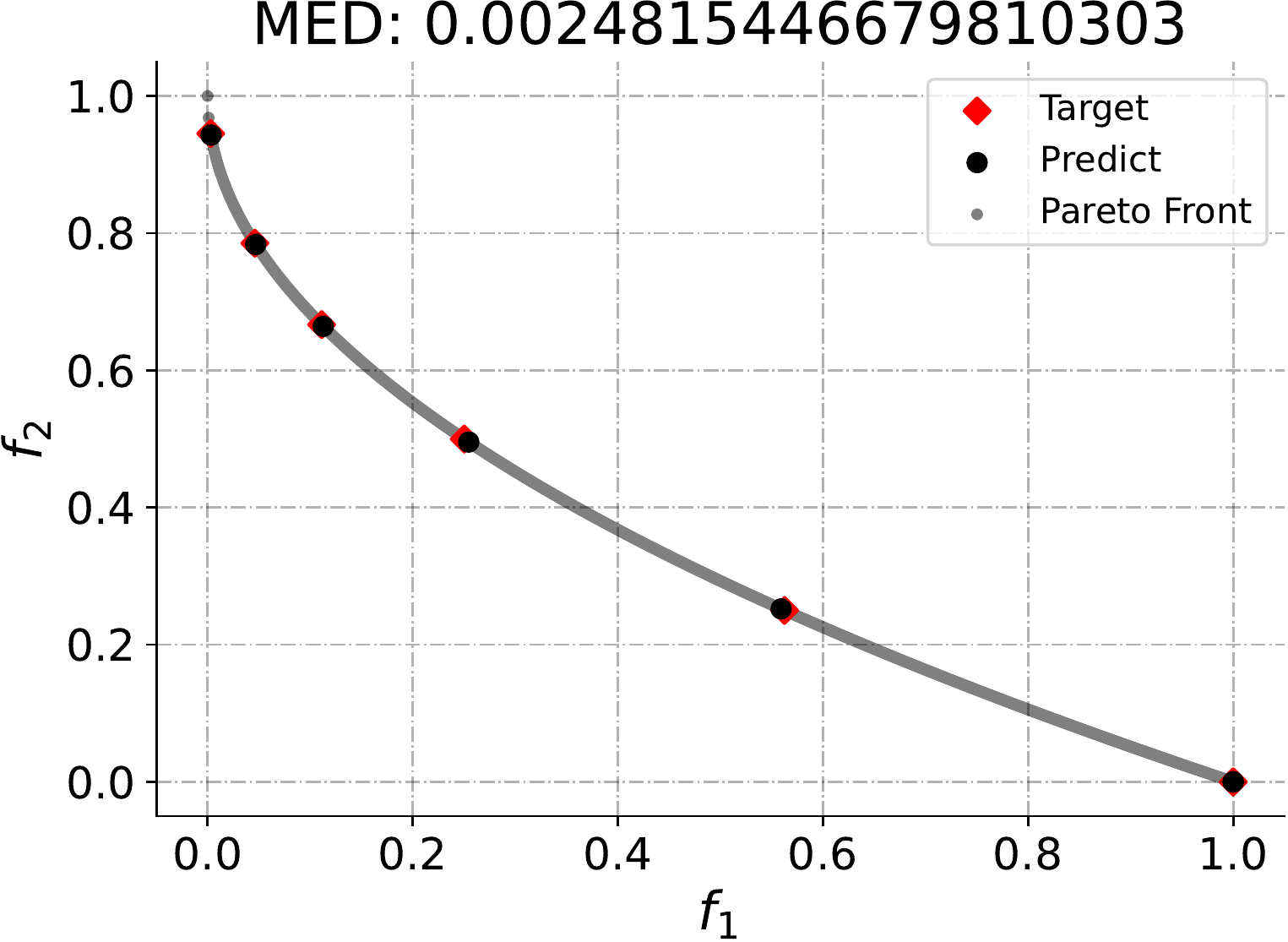}
     \end{subfigure}
     \hfill
     \begin{subfigure}[b]{0.24\textwidth}
         \centering
             \includegraphics[width=\textwidth]{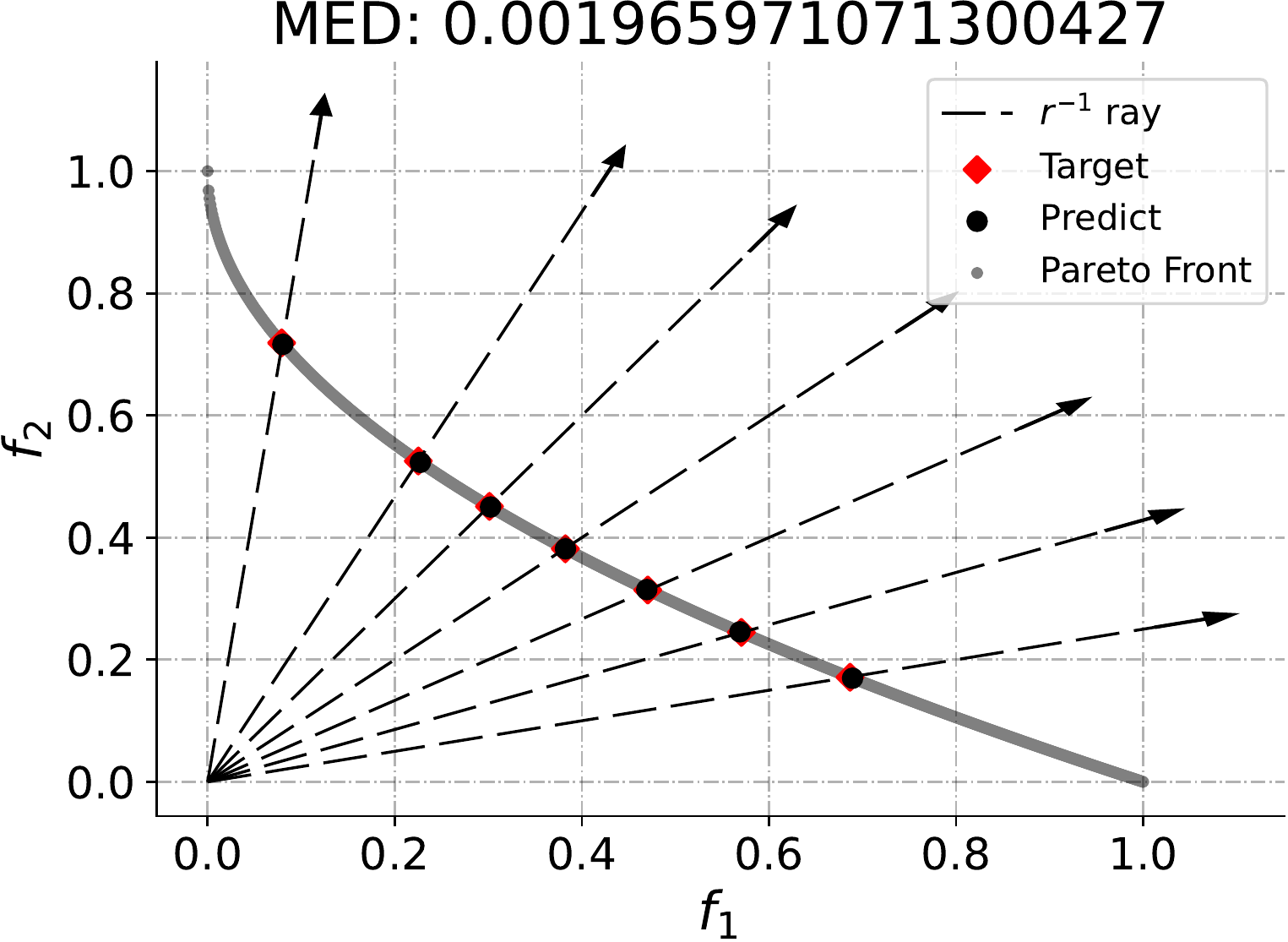}
     \end{subfigure}
     \hfill
     \begin{subfigure}[b]{0.24\textwidth}
         \centering
         \includegraphics[width=\textwidth]{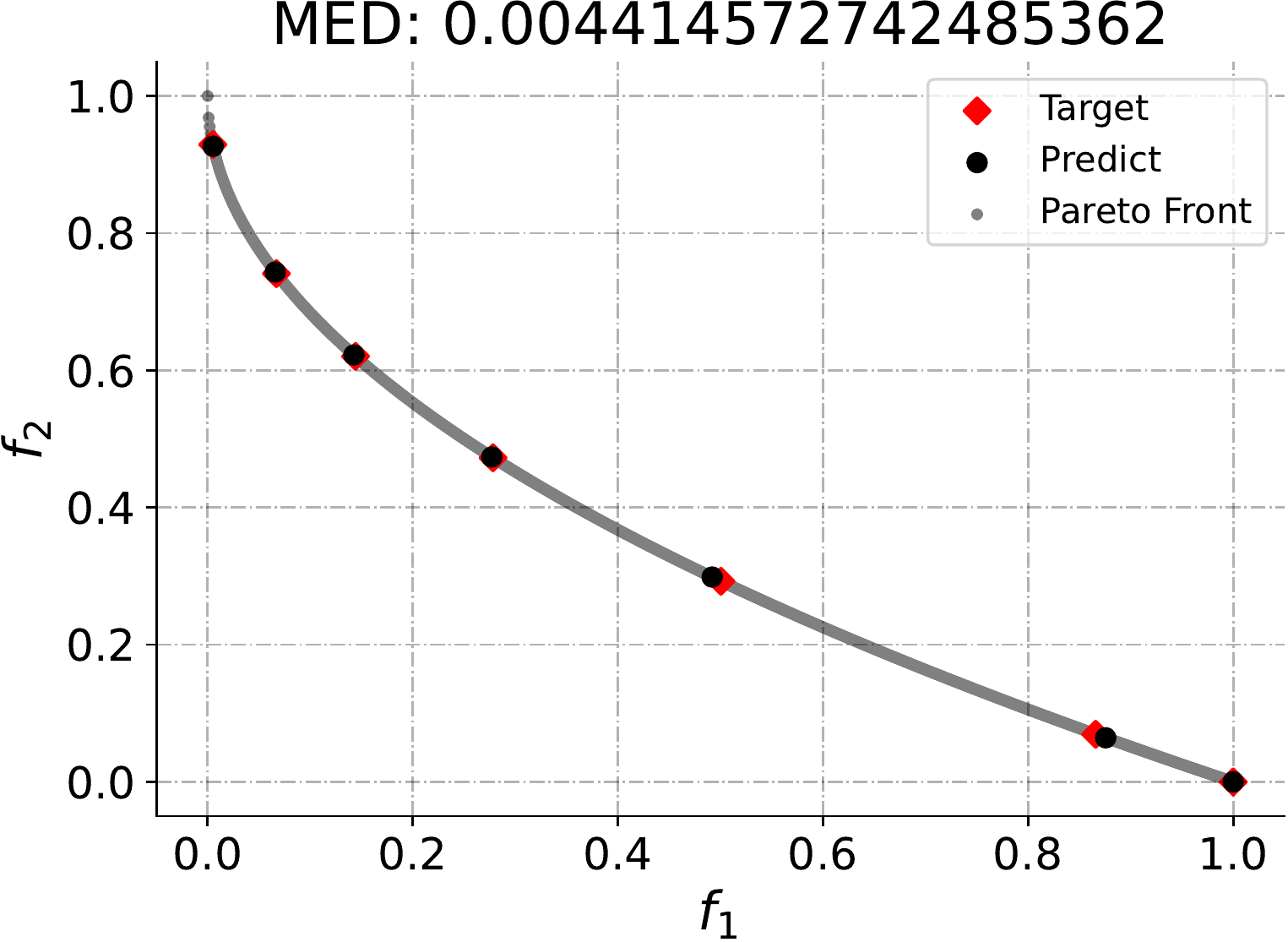}
     \end{subfigure}

     \centering
     \begin{subfigure}[b]{0.24\textwidth}
         \centering
         \includegraphics[width=\textwidth]{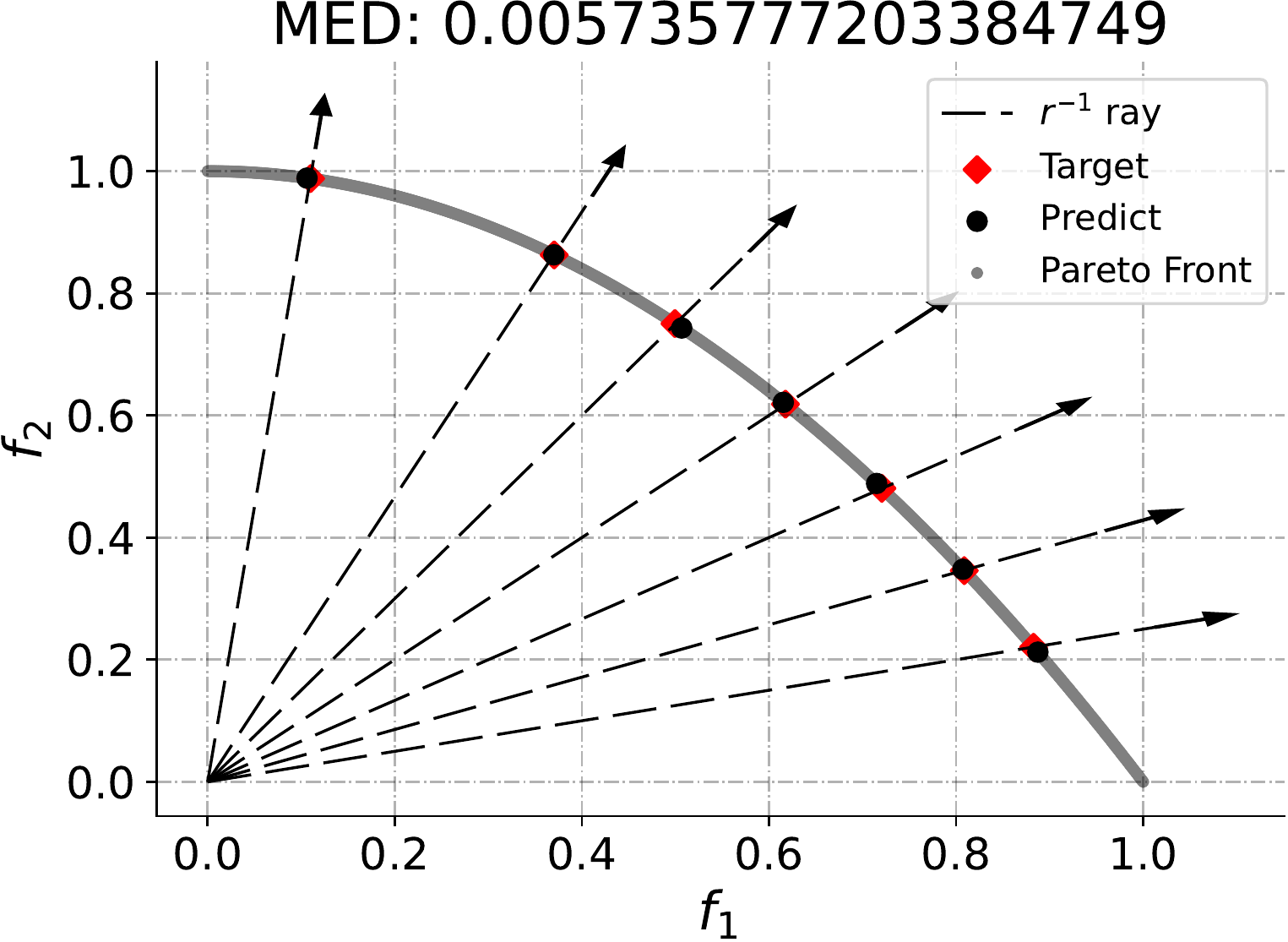}
     \end{subfigure}
     \hfill
     \begin{subfigure}[b]{0.24\textwidth}
         \centering
         \includegraphics[width=\textwidth]{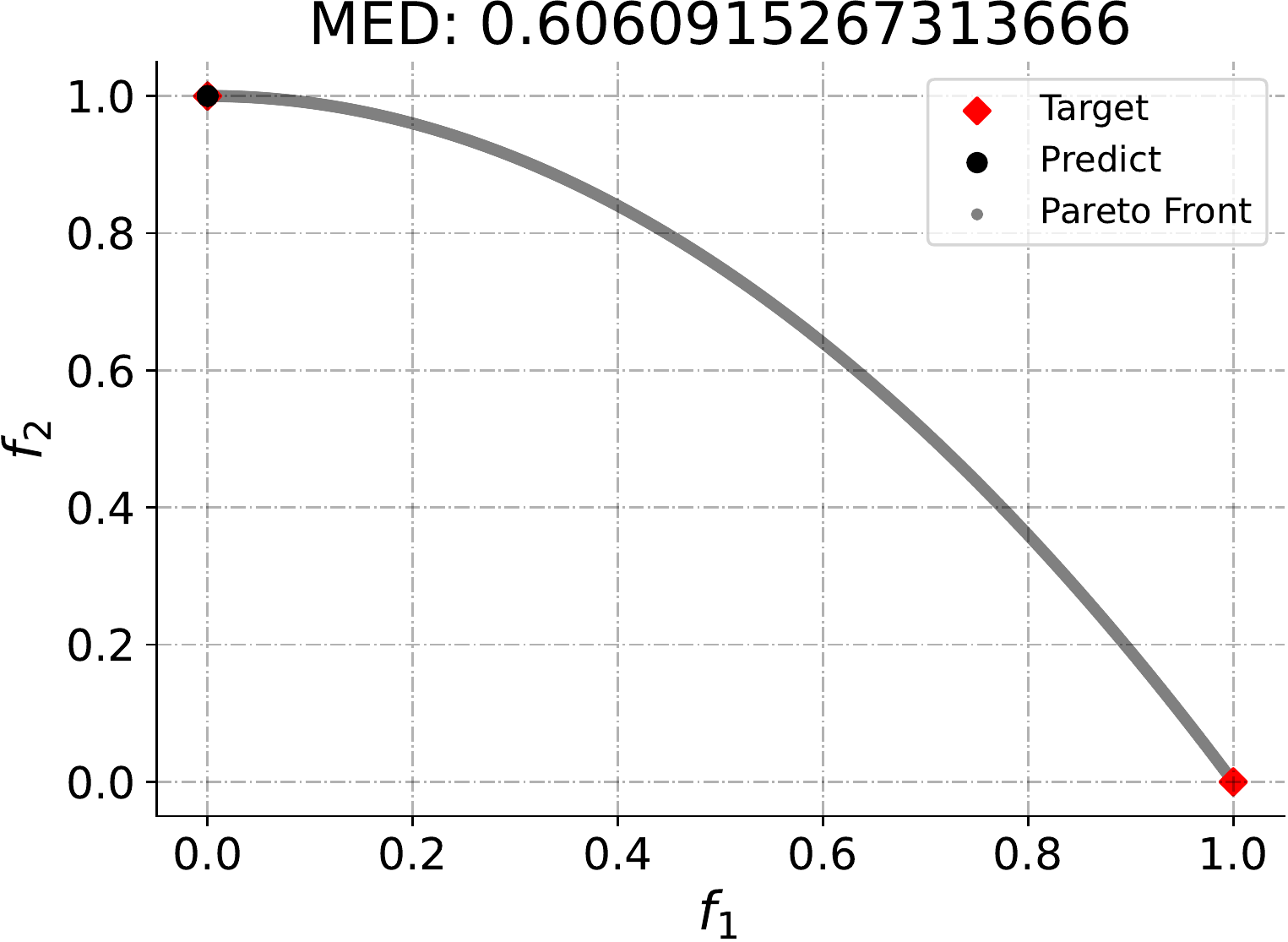}
     \end{subfigure}
     \hfill
     \begin{subfigure}[b]{0.24\textwidth}
         \centering
         \includegraphics[width=\textwidth]{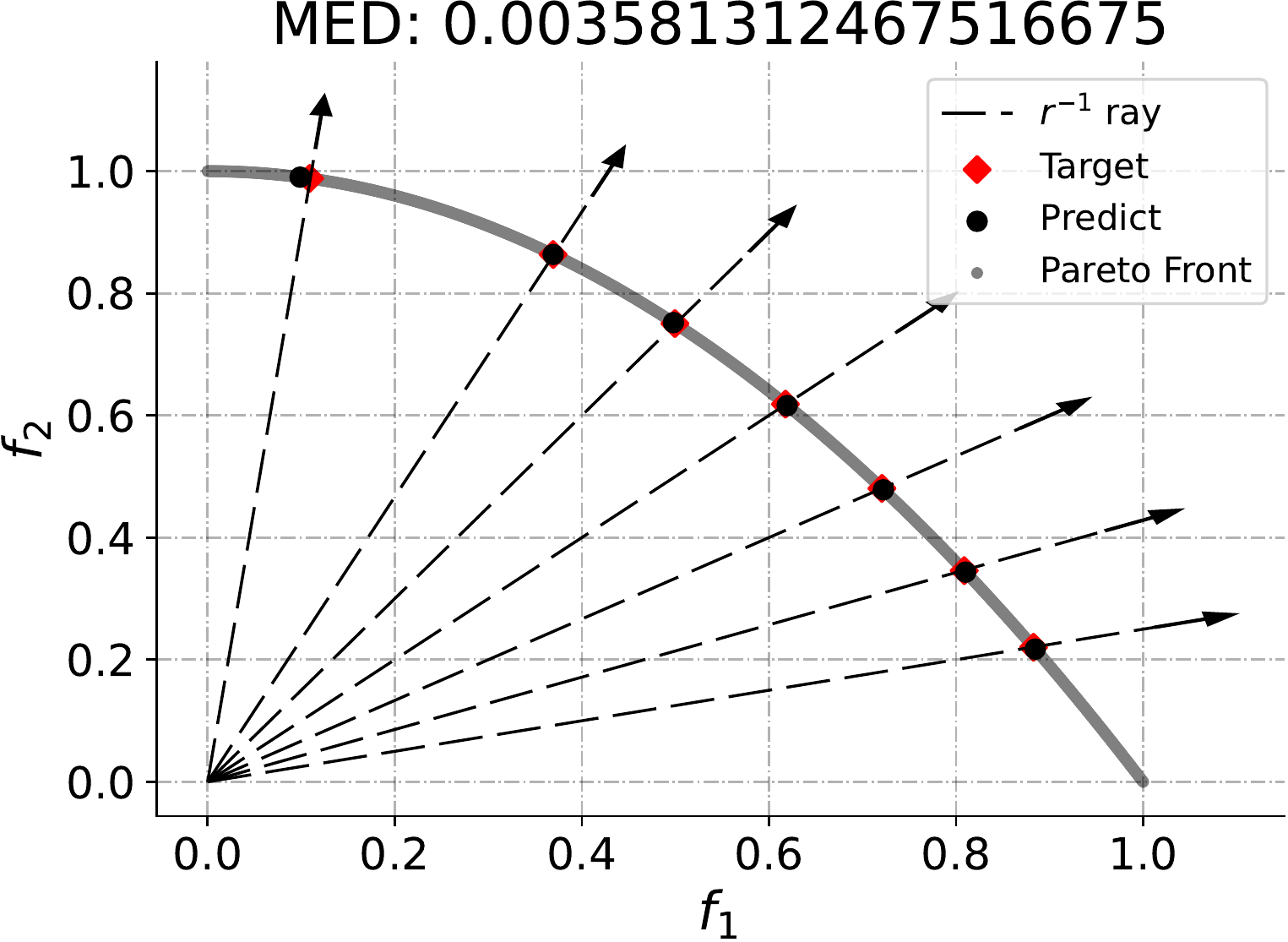}
     \end{subfigure}
     \hfill
     \begin{subfigure}[b]{0.24\textwidth}
         \centering
         \includegraphics[width=\textwidth]{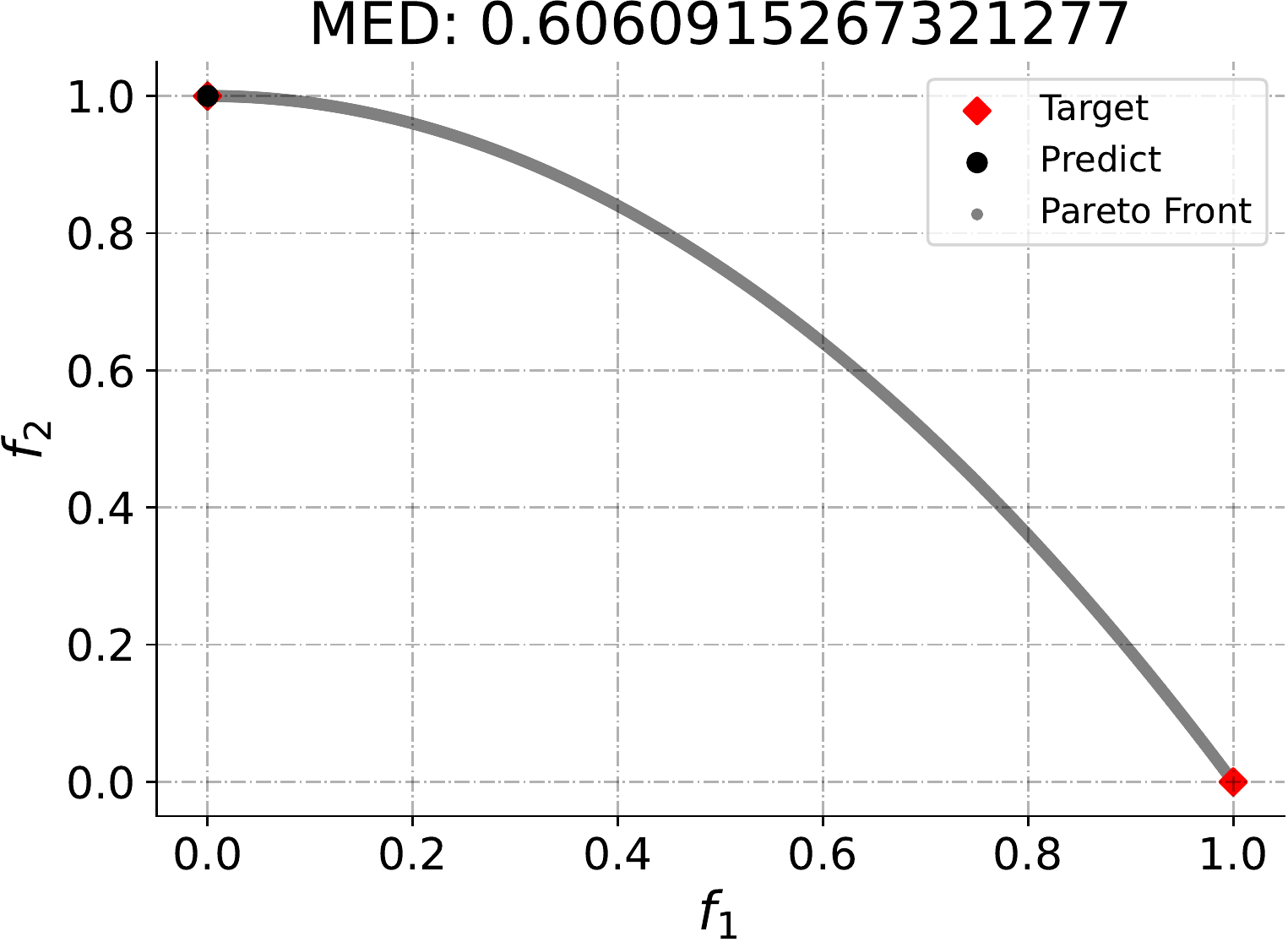}
     \end{subfigure}

    \begin{subfigure}[b]{0.24\textwidth}
         \centering
         \includegraphics[width=\textwidth]{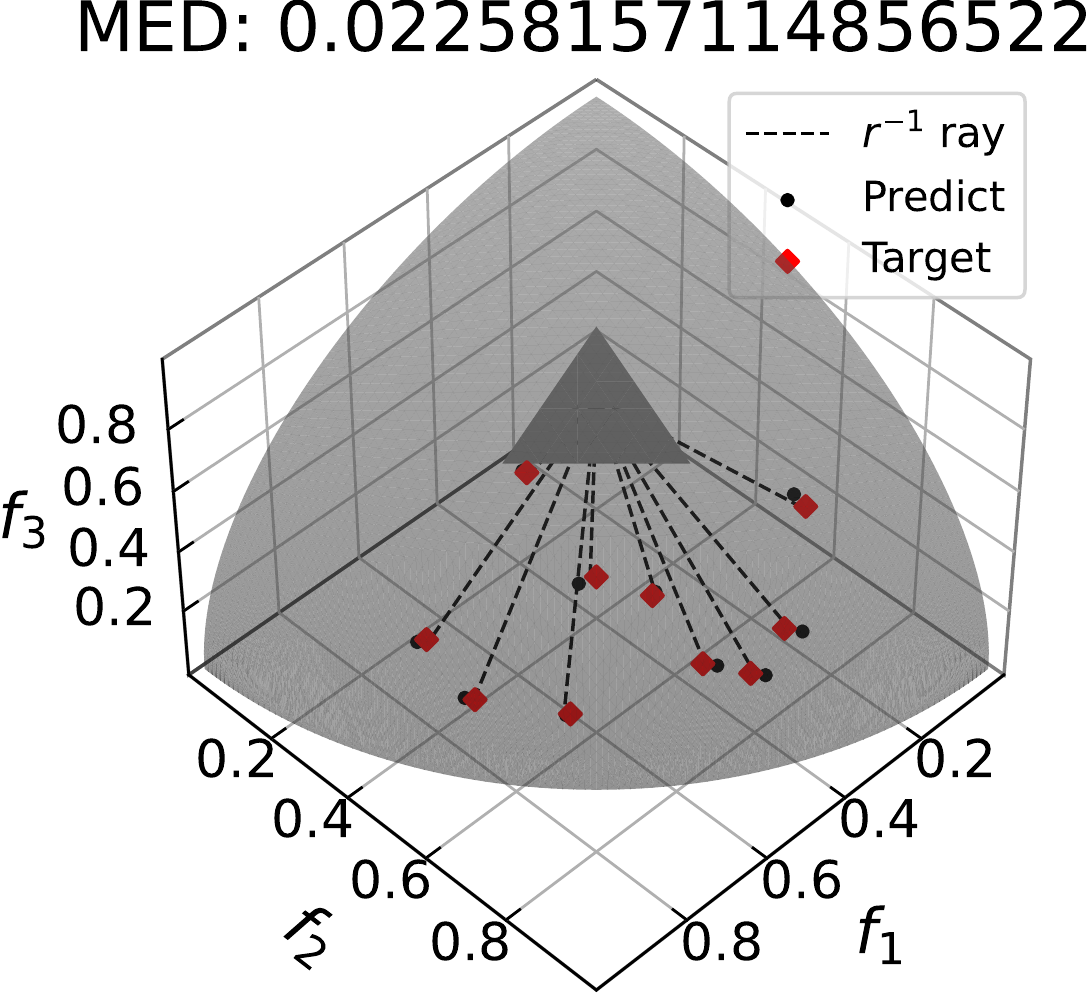}
         \caption{PHN-EPO}
     \end{subfigure}
     \hfill
     \begin{subfigure}[b]{0.24\textwidth}
         \centering
         \includegraphics[width=\textwidth]{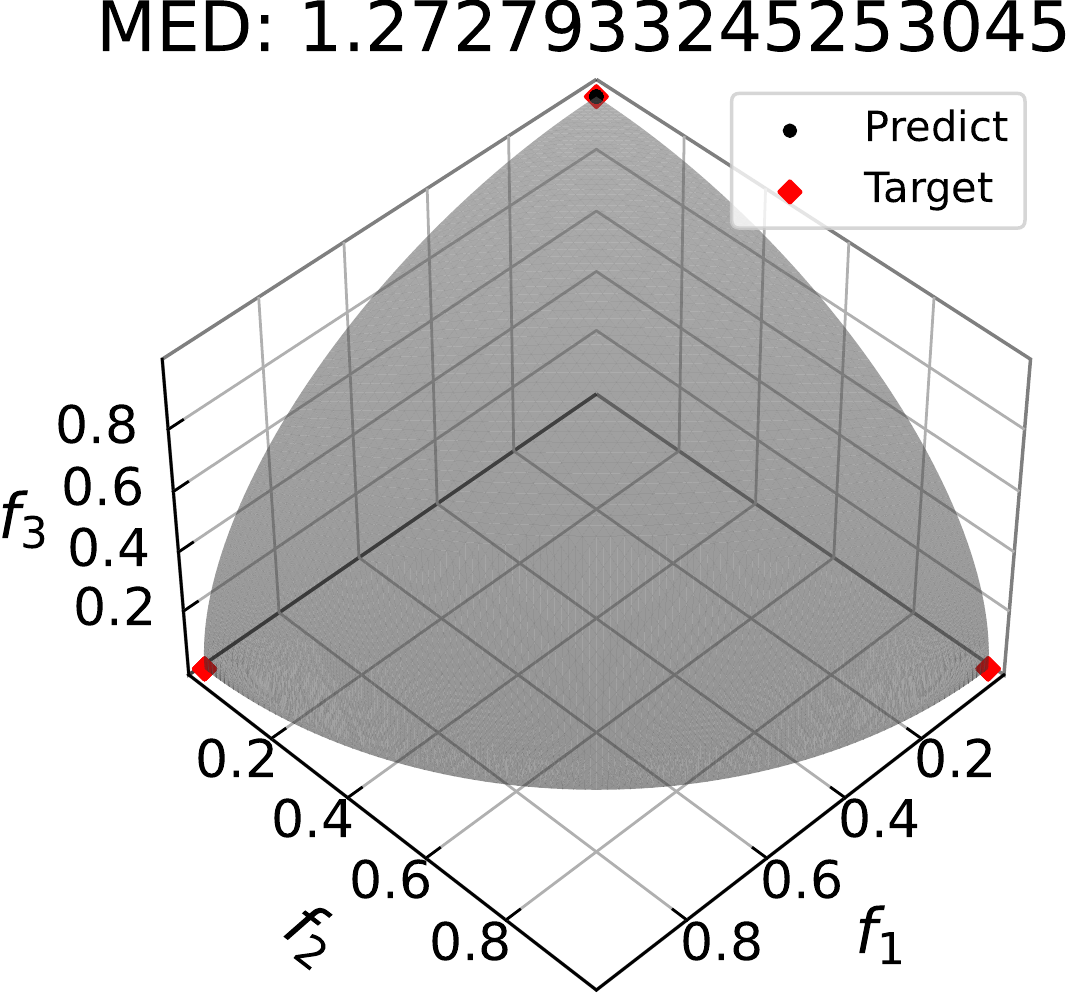}
         \caption{PHN-LS}
     \end{subfigure}
     \hfill
     \begin{subfigure}[b]{0.24\textwidth}
         \centering
         \includegraphics[width=\textwidth]{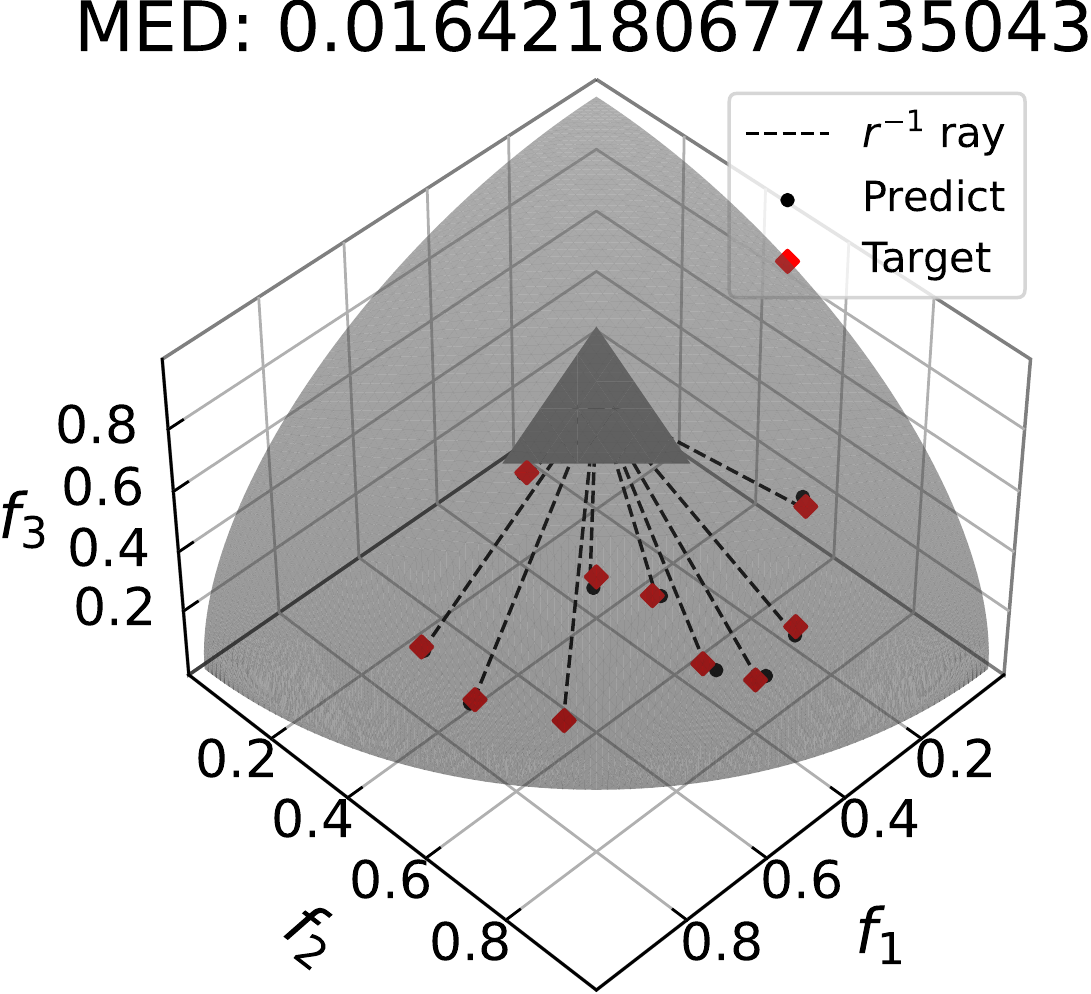}
         \caption{PHN-Cheby}
     \end{subfigure}
     \hfill
     \begin{subfigure}[b]{0.24\textwidth}
         \centering
         \includegraphics[width=\textwidth]{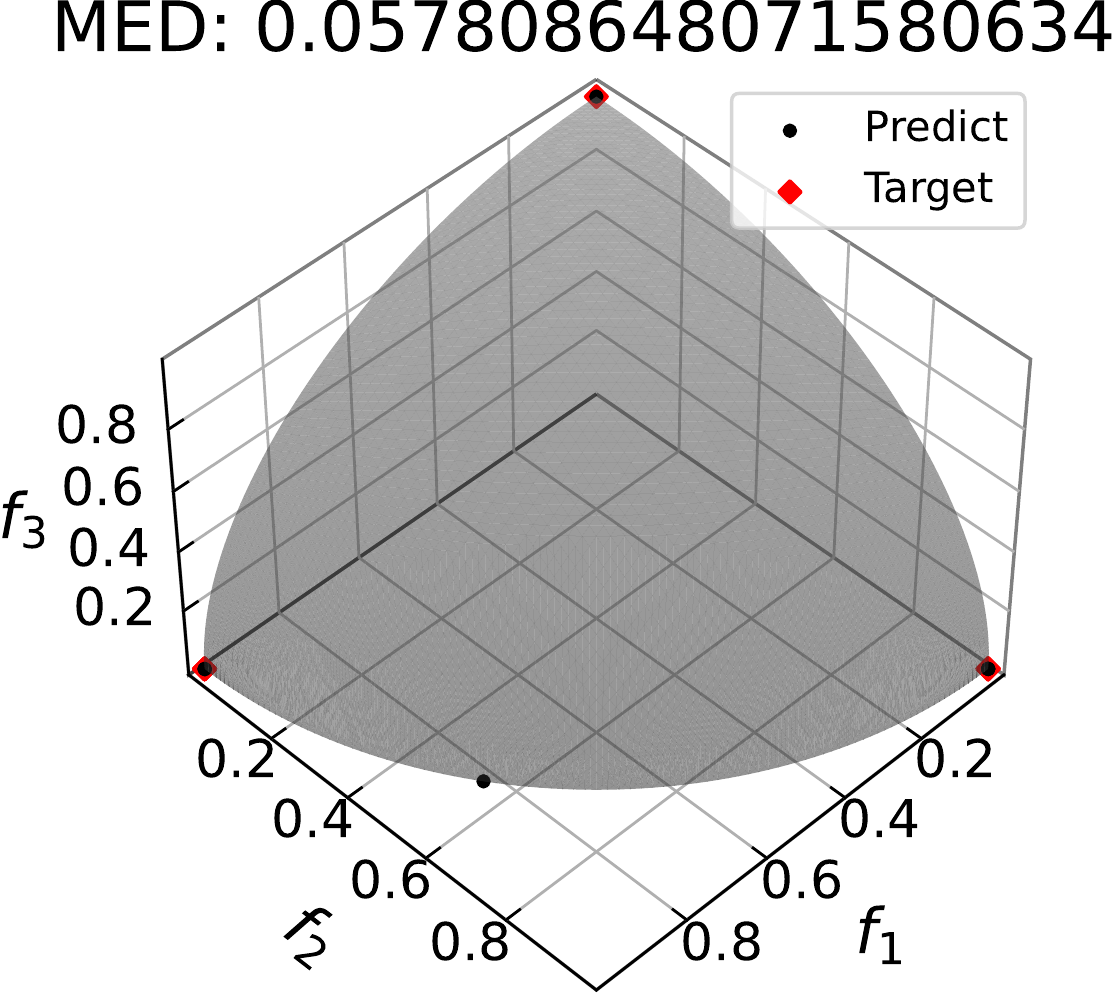}
          \caption{PHN-Utility}
     \end{subfigure}
      \caption{\revise{Output of Controllable Pareto Front Learning method in ZDT1 (top), ZDT2 (middle), and DTLZ2 (bottom). The black dashed line represents the $\mathbf{r}^{-1}$ ray in Figures (a) and Figures (c). The \textcolor{red}{red} and \textcolor{black}{black} points are the target and predicted points respectively.}}
      \label{fig}
\end{figure*}

\begin{figure*}[ht]
    \centering
     \begin{subfigure}[b]{0.3\textwidth}
     \centering
     \includegraphics[width=\textwidth]{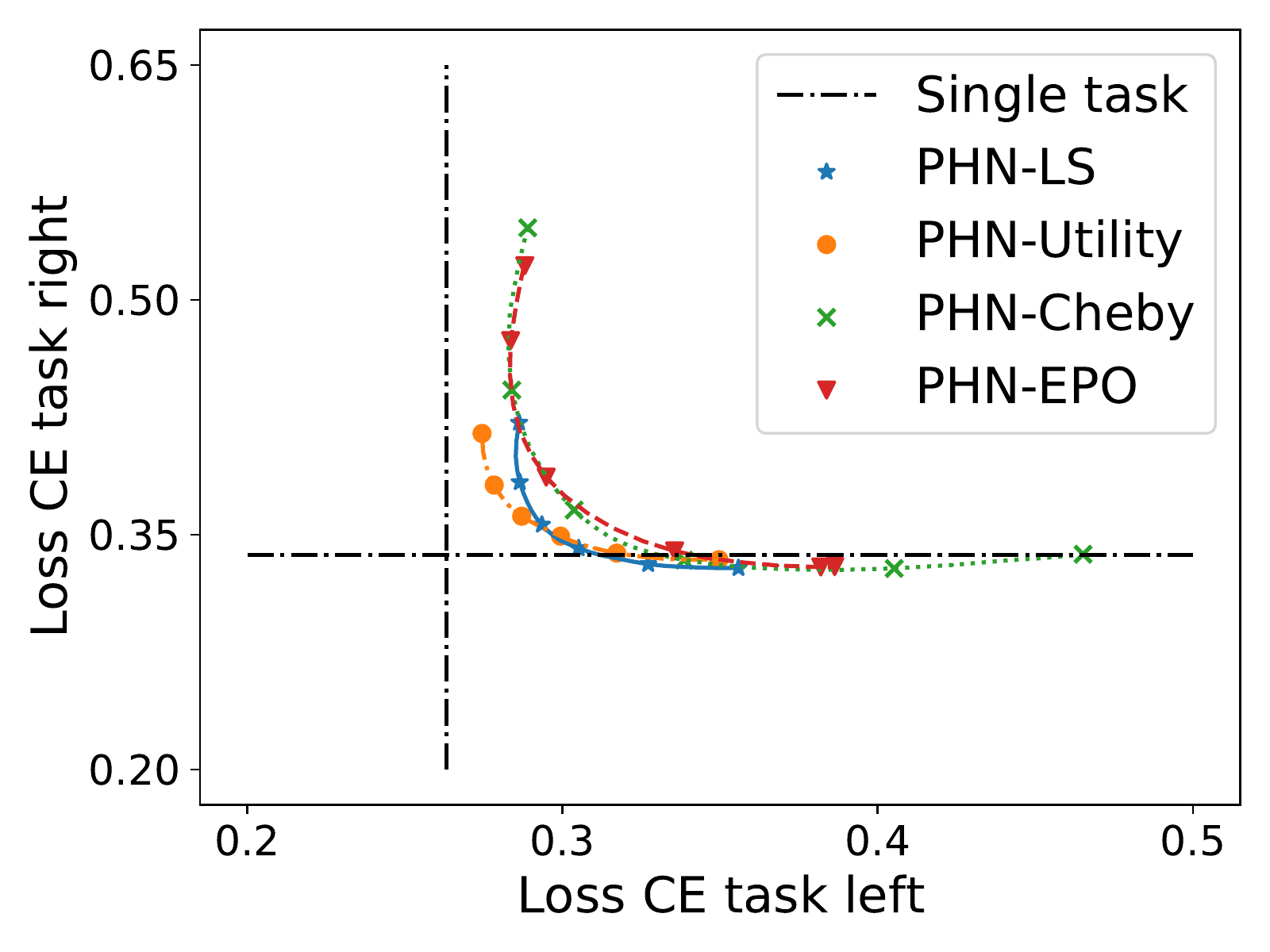}

     \end{subfigure}
         \hfill
     \begin{subfigure}[b]{0.3\textwidth}
         \centering
         \includegraphics[width=\textwidth]{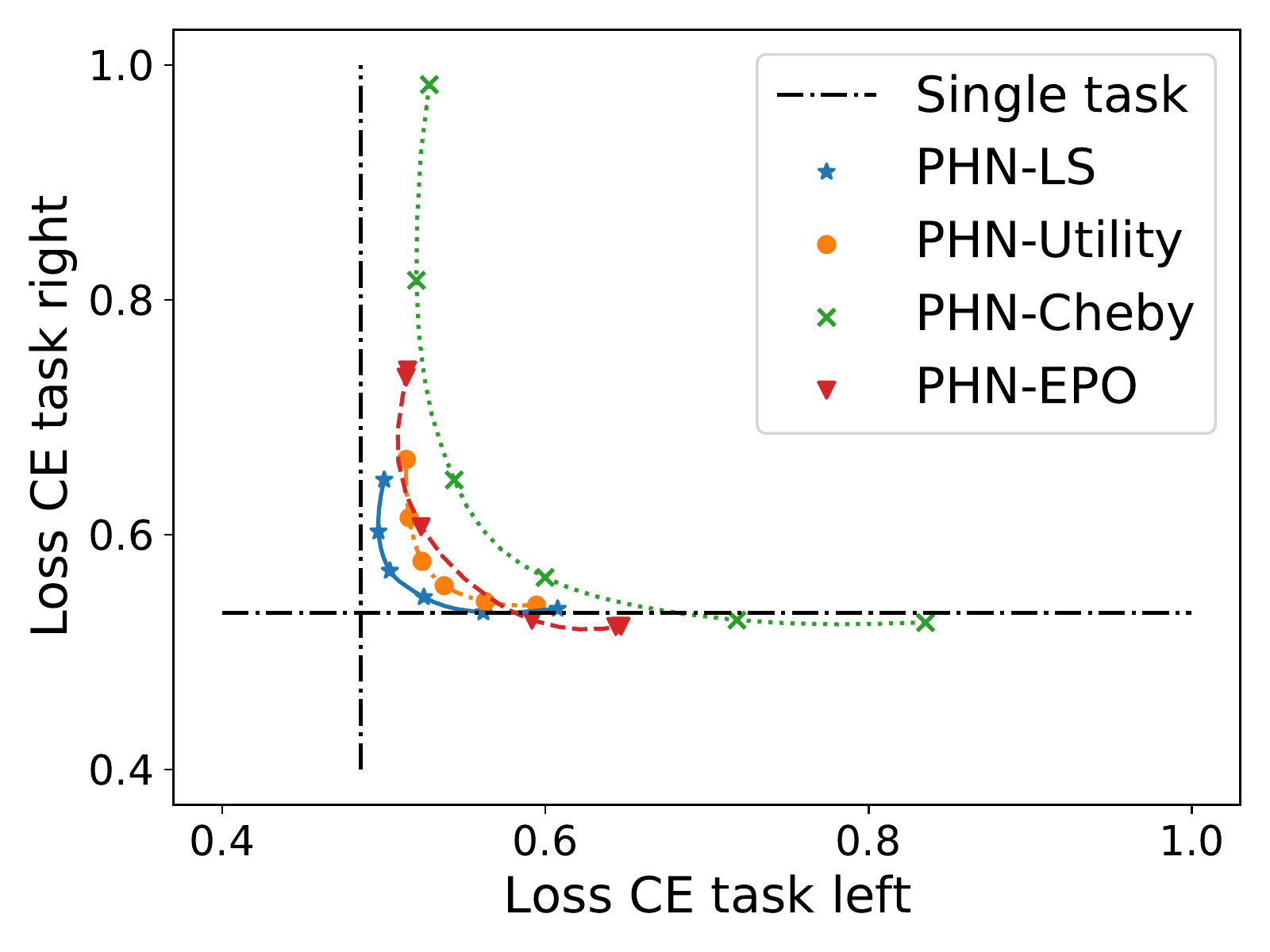}

     \end{subfigure}
              \hfill
     \begin{subfigure}[b]{0.3\textwidth}
         \centering
         \includegraphics[width=\textwidth]{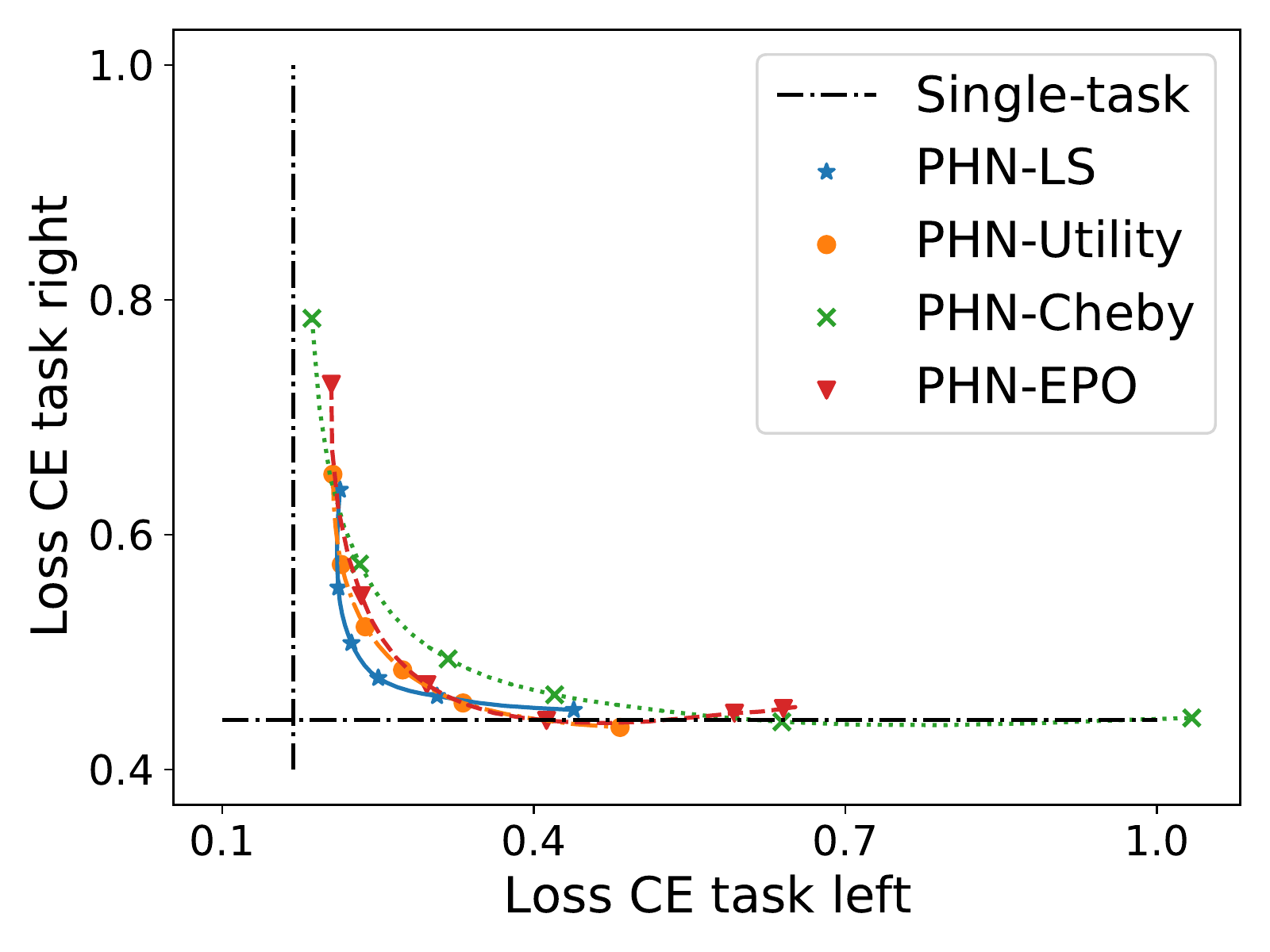}

     \end{subfigure}

     \caption{Pareto front generated by methods.}
     \label{acc}
\end{figure*}
\begin{figure*}[!ht]
    \centering
     \begin{subfigure}[b]{0.3\textwidth}
     \centering
     \includegraphics[width=\textwidth]{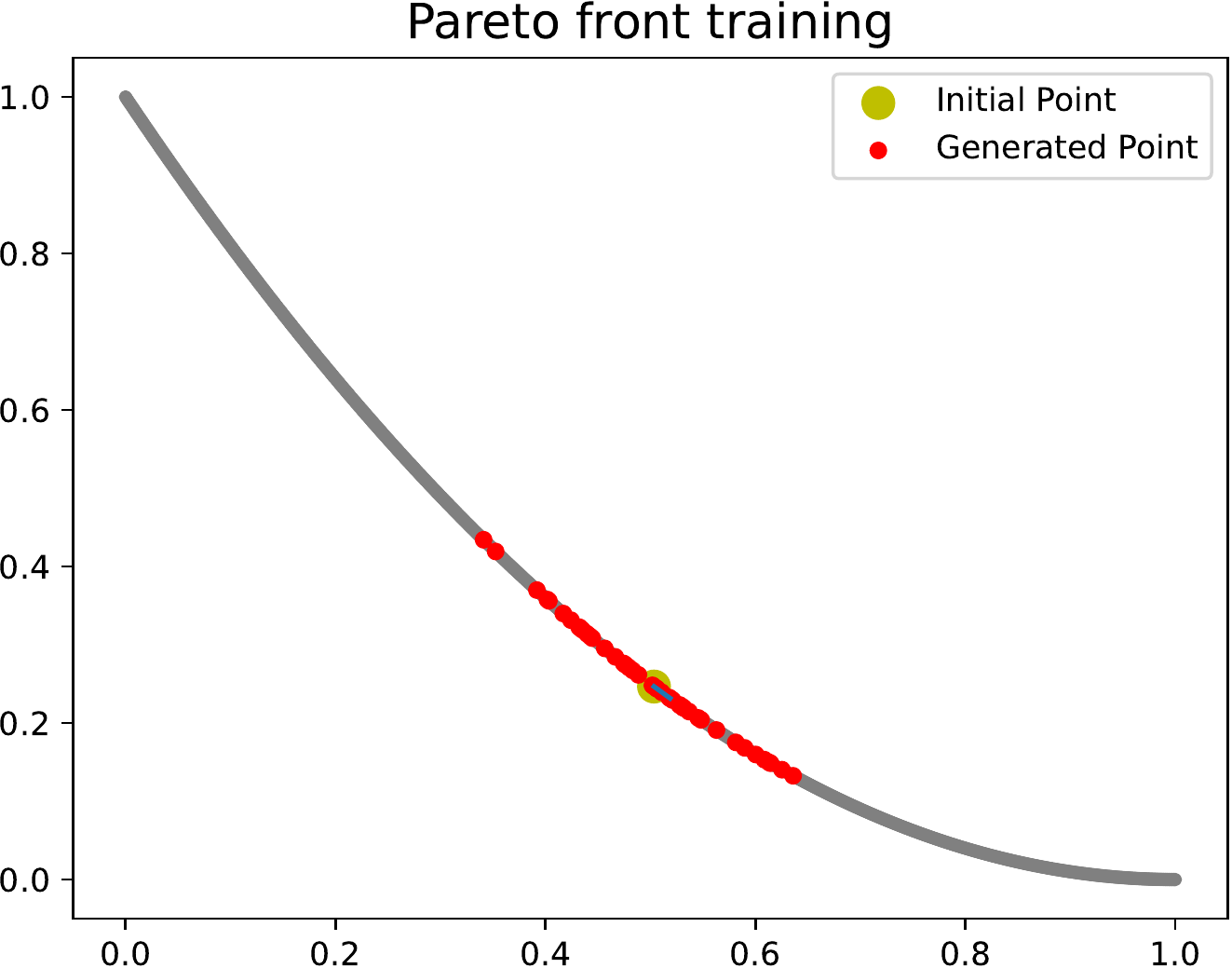}
     \caption{50 rays training.}
     \end{subfigure}
         \hfill
     \begin{subfigure}[b]{0.3\textwidth}
         \centering
         \includegraphics[width=\textwidth]{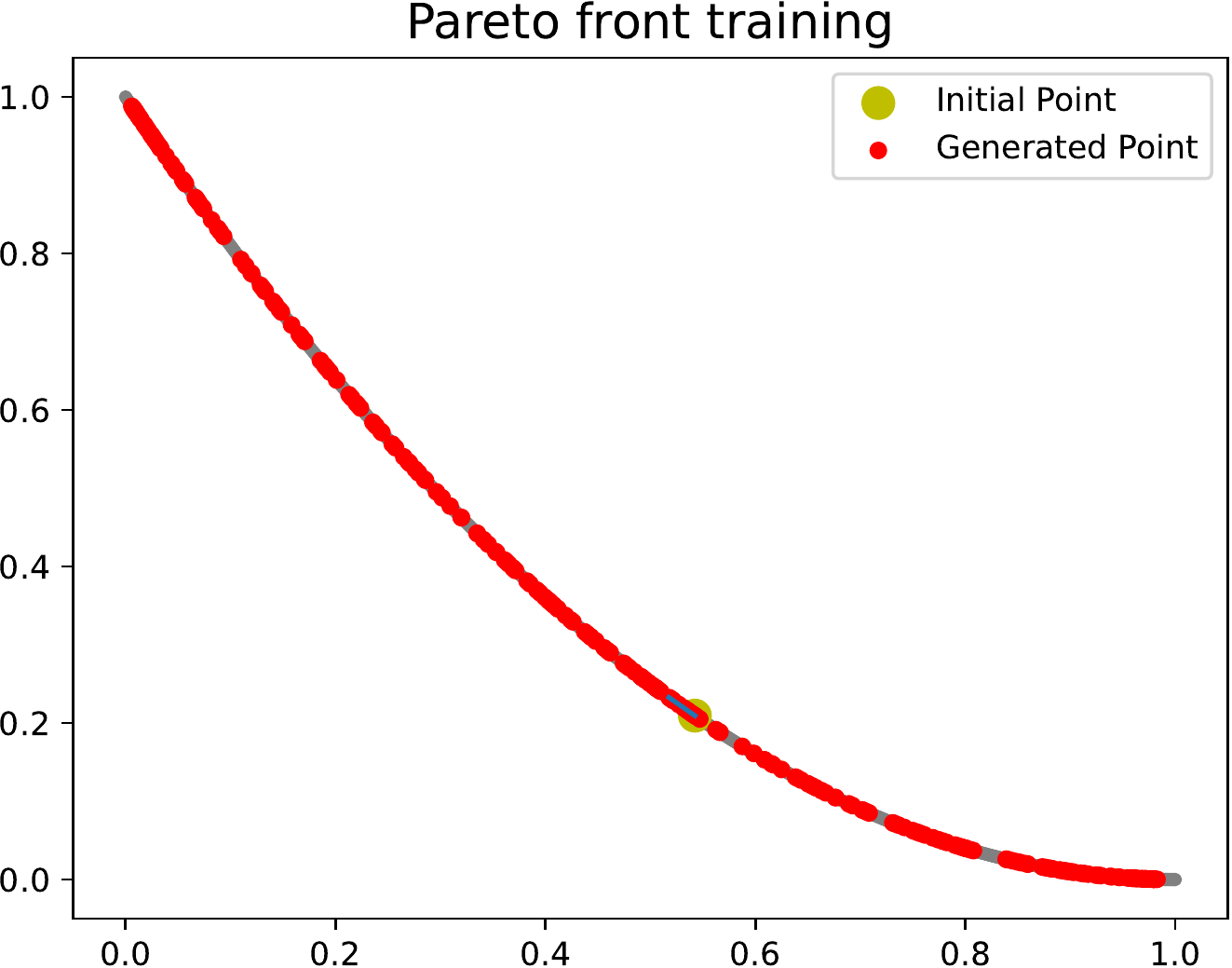}
         \caption{300 rays training.}
     \end{subfigure}
              \hfill
     \begin{subfigure}[b]{0.3\textwidth}
         \centering
         \includegraphics[width=\textwidth]{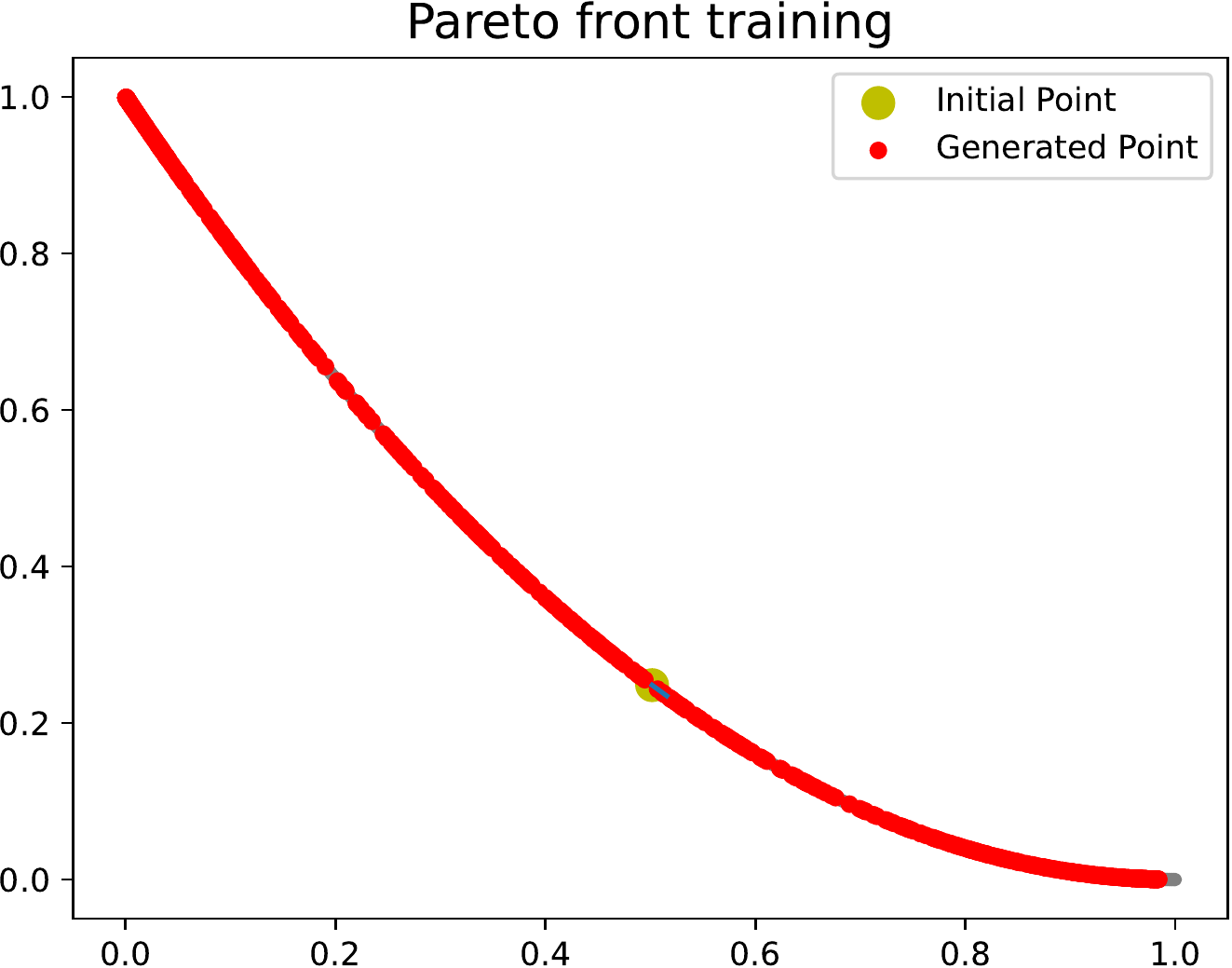}
         \caption{1000 rays training.}
     \end{subfigure}
     \caption{Utilize \eqref{LS} for learning Hypernetwork to approximate Pareto Front of Example 7.1.}
     \label{fig5}
\end{figure*}
\begin{figure*}[!htb]
    \centering
     \begin{subfigure}[b]{0.3\textwidth}
     \centering
     \includegraphics[width=\textwidth]{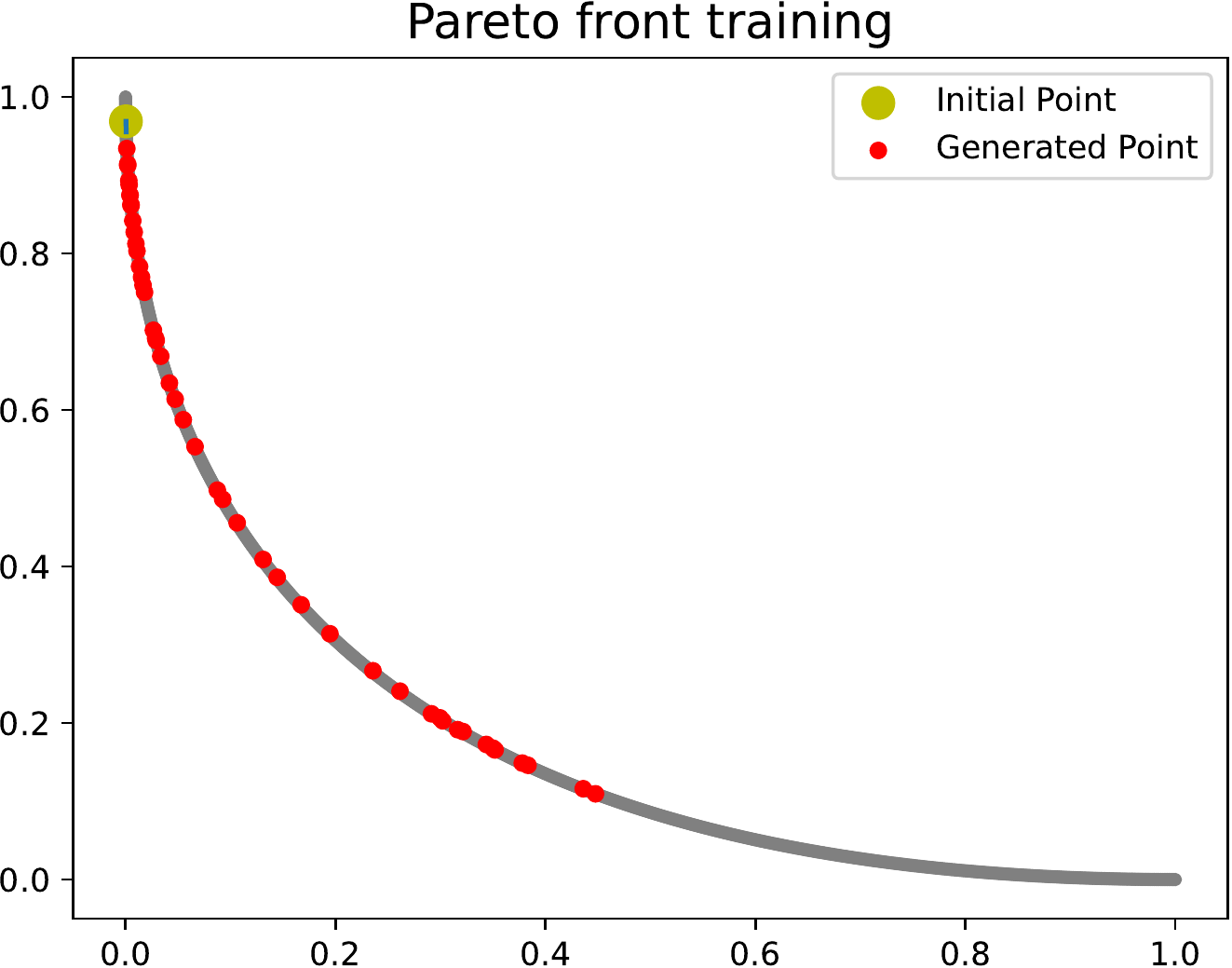}
     \caption{50 rays training.}
     \end{subfigure}
         \hfill
     \begin{subfigure}[b]{0.3\textwidth}
         \centering
         \includegraphics[width=\textwidth]{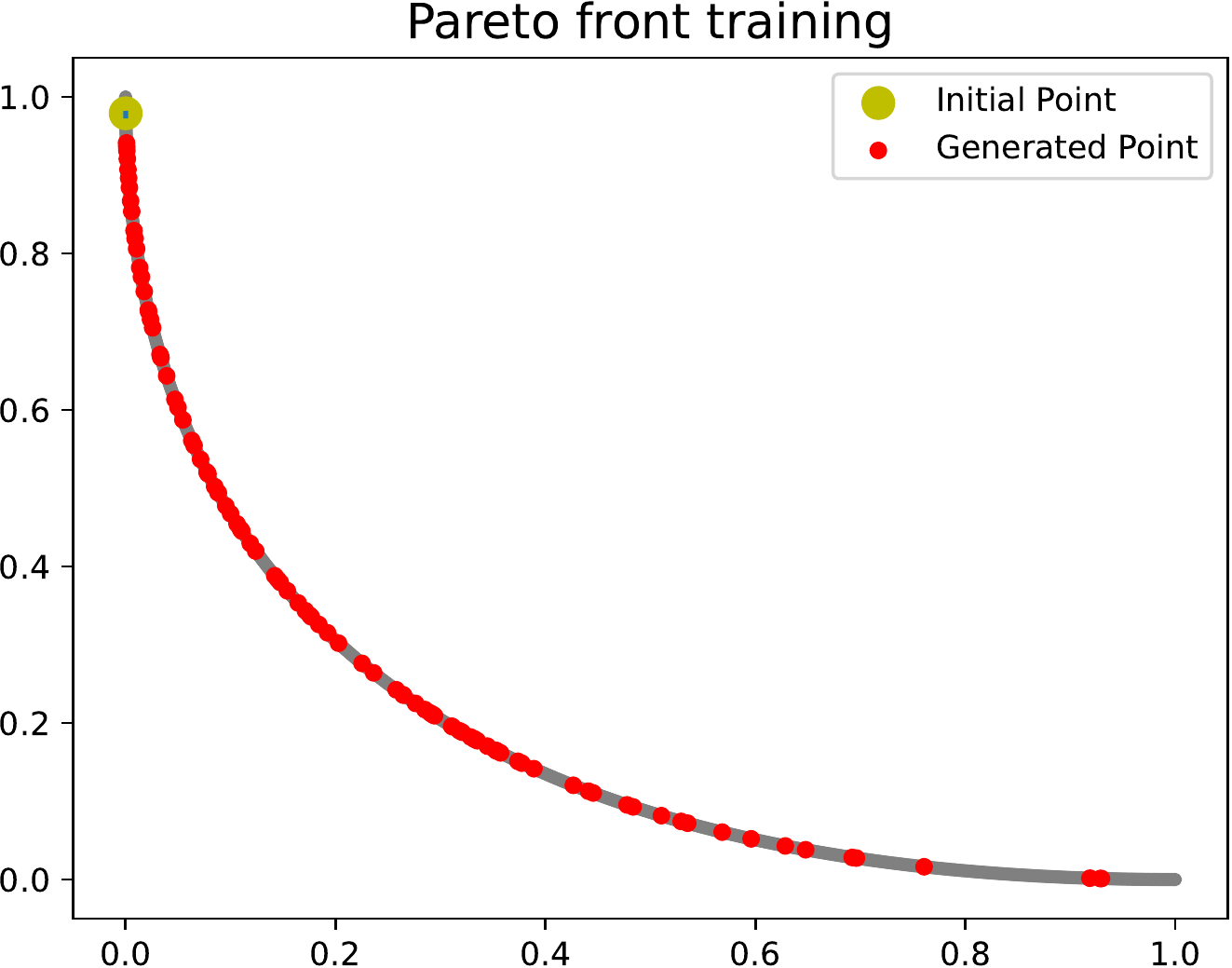}
         \caption{100 rays training.}
     \end{subfigure}
              \hfill
     \begin{subfigure}[b]{0.3\textwidth}
         \centering
         \includegraphics[width=\textwidth]{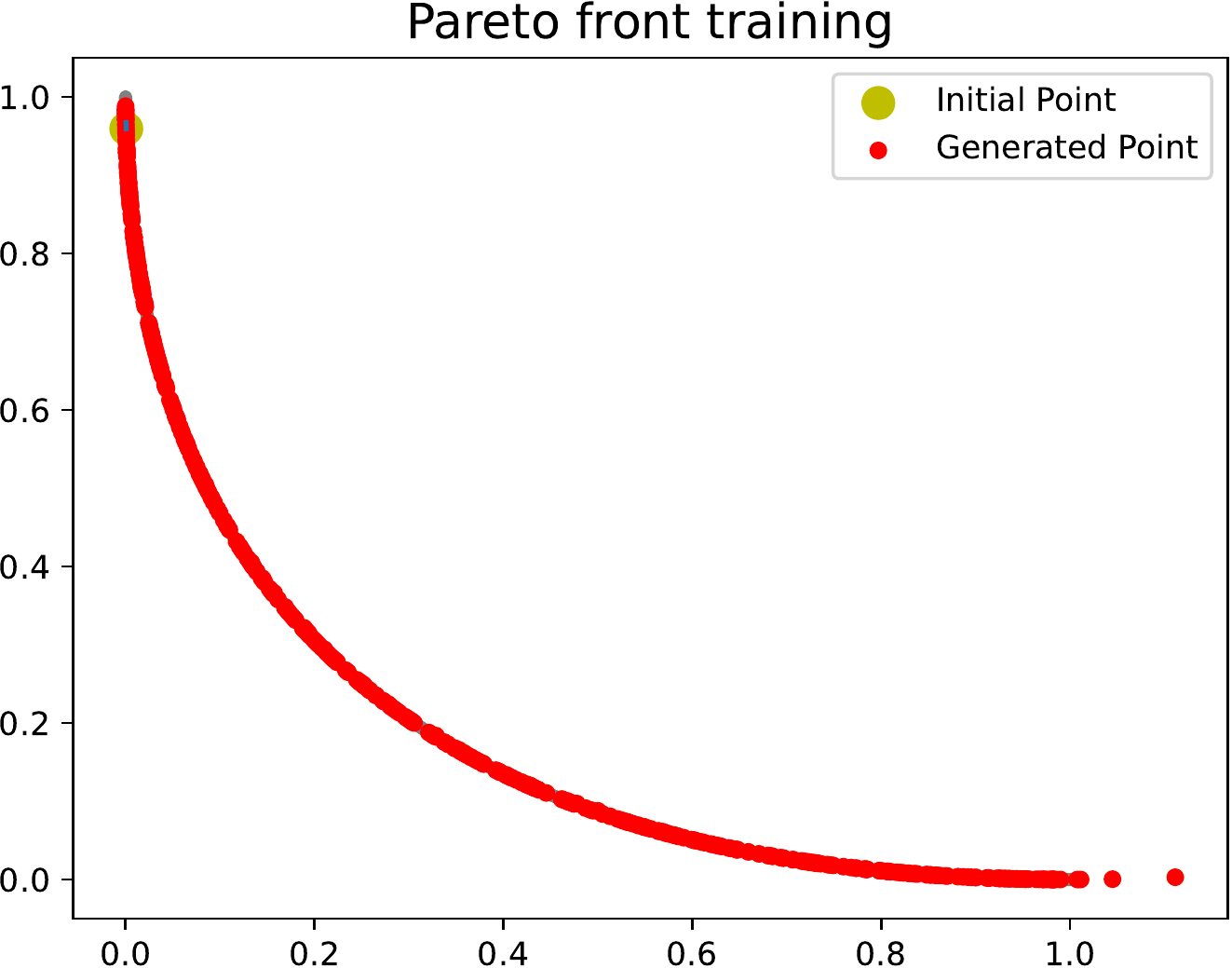}
         \caption{500 rays training.}
     \end{subfigure}
     \caption{Utilize \eqref{LS} for learning Hypernetwork to approximate Pareto Front of Example 7.2.}
     \label{fig6}
\end{figure*}
\begin{figure*}[!htb]
    \centering
     \begin{subfigure}[b]{0.3\textwidth}
     \centering
     \includegraphics[width=\textwidth]{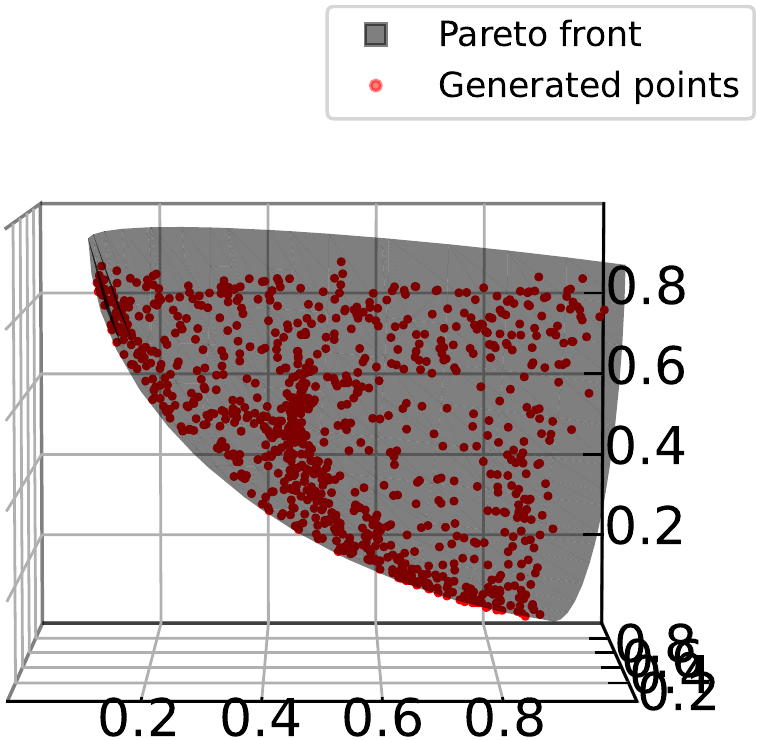}
     \caption{1000 rays training.}
     \end{subfigure}
         \hfill
     \begin{subfigure}[b]{0.3\textwidth}
         \centering
         \includegraphics[width=\textwidth]{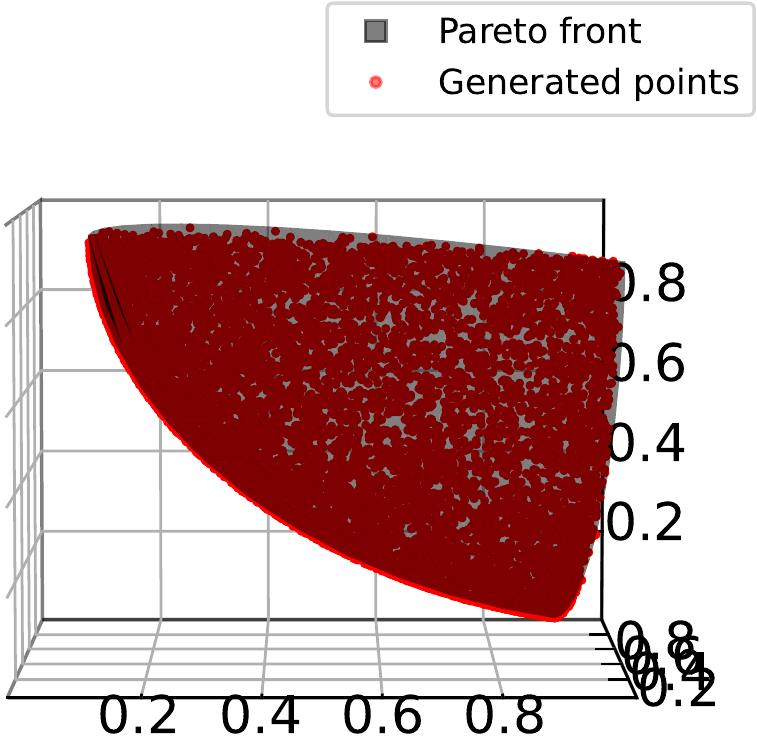}
         \caption{10000 rays training.}
     \end{subfigure}
              \hfill
     \begin{subfigure}[b]{0.3\textwidth}
         \centering
         \includegraphics[width=\textwidth]{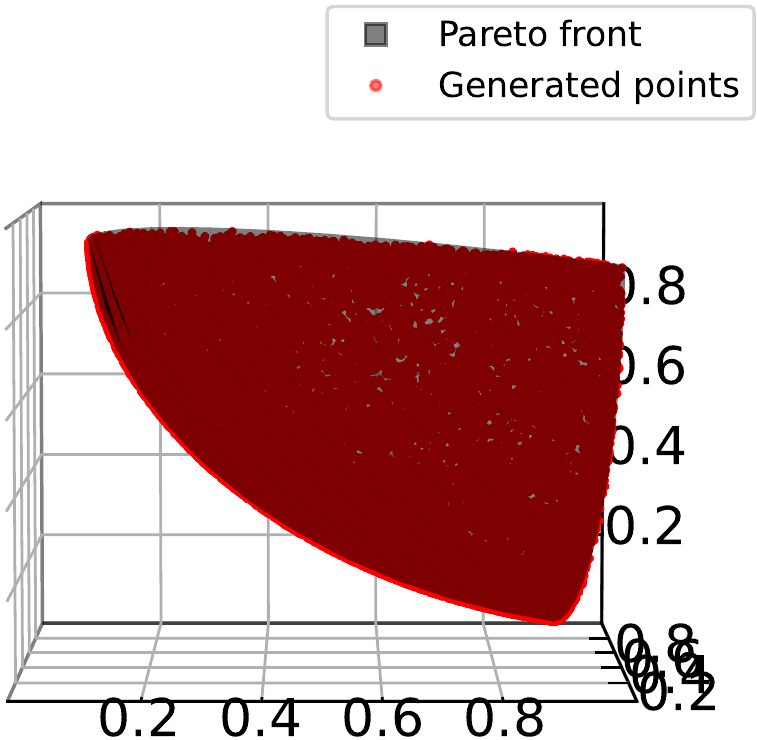}
         \caption{20000 rays training.}
     \end{subfigure}
     \caption{Utilize \eqref{Utility} for learning Hypernetwork to approximate Pareto Front of Example 7.3.}
     \label{fig7}
\end{figure*}
\begin{figure*}[!htb]
     \centering
     \includegraphics[width=0.8\textwidth]{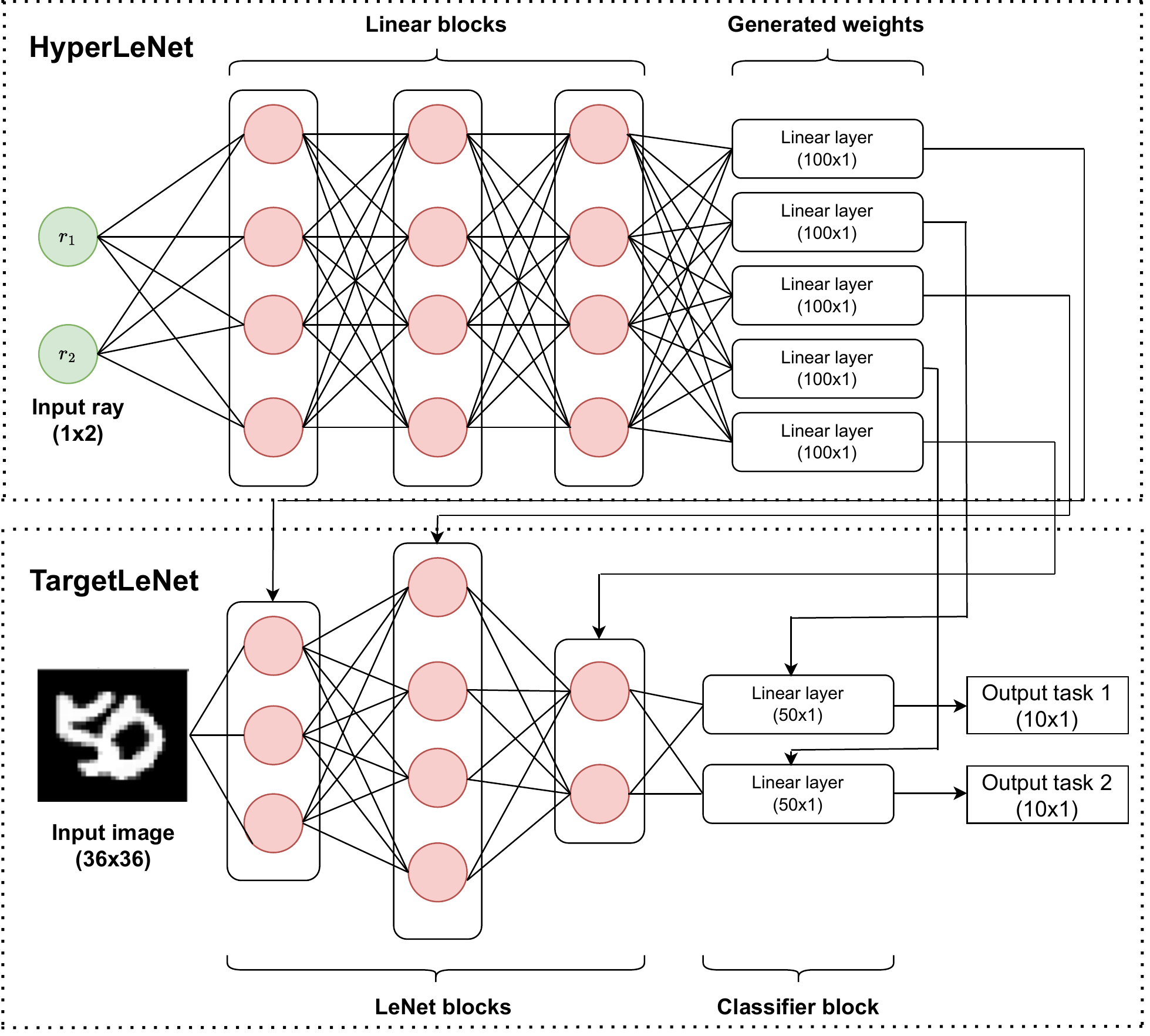}
     \caption{Multi-LeNet architecture.}
     \label{fig8}
\end{figure*}
\begin{figure*}[!htb]
     \centering
     \includegraphics[width=0.9\textwidth]{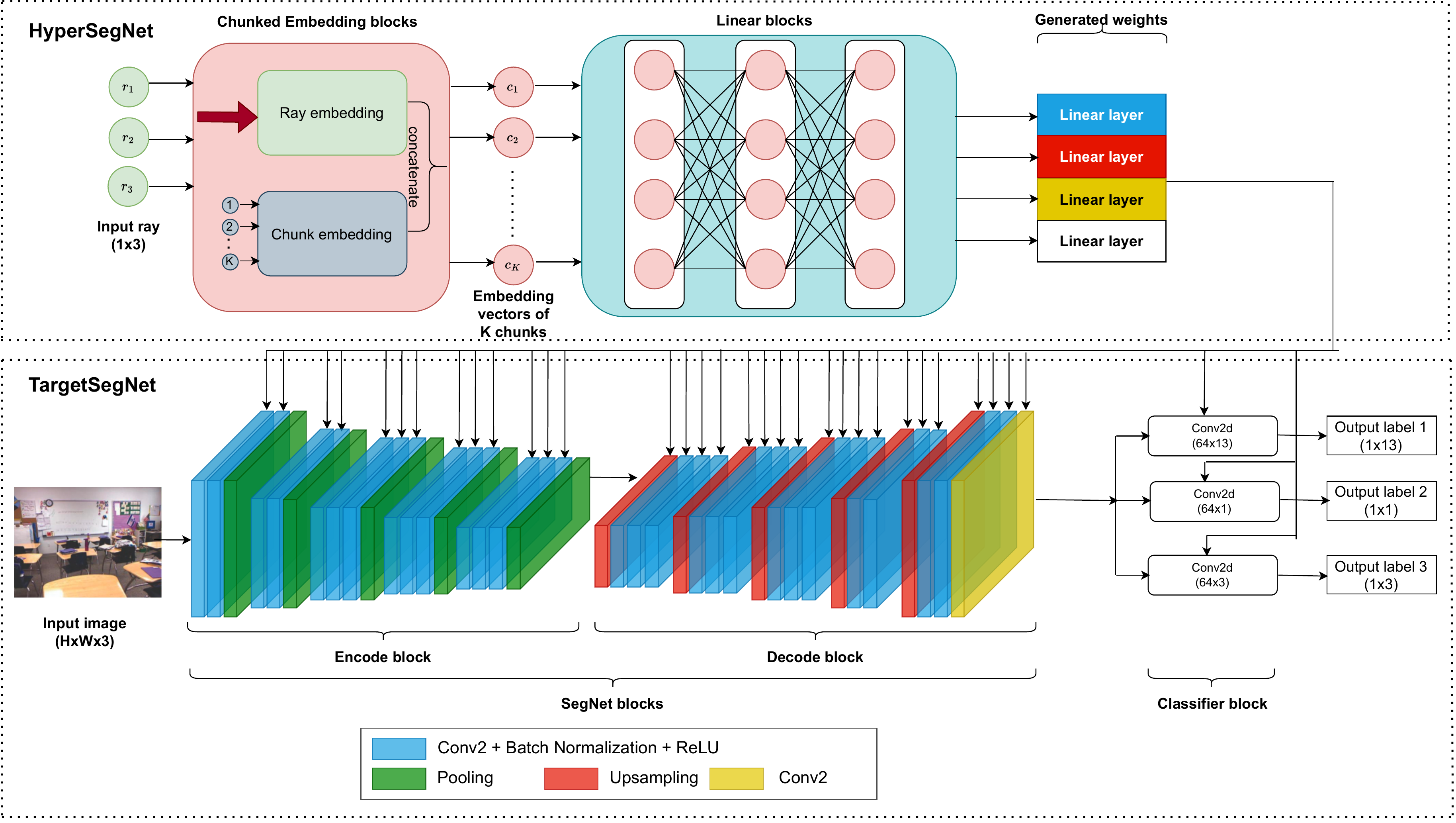}
     \caption{Multi-SegNet architecture.}
     \label{fig15}
\end{figure*}

\begin{figure*}[ht]
    \centering
    \includegraphics[width=1.3\textwidth, angle=270]{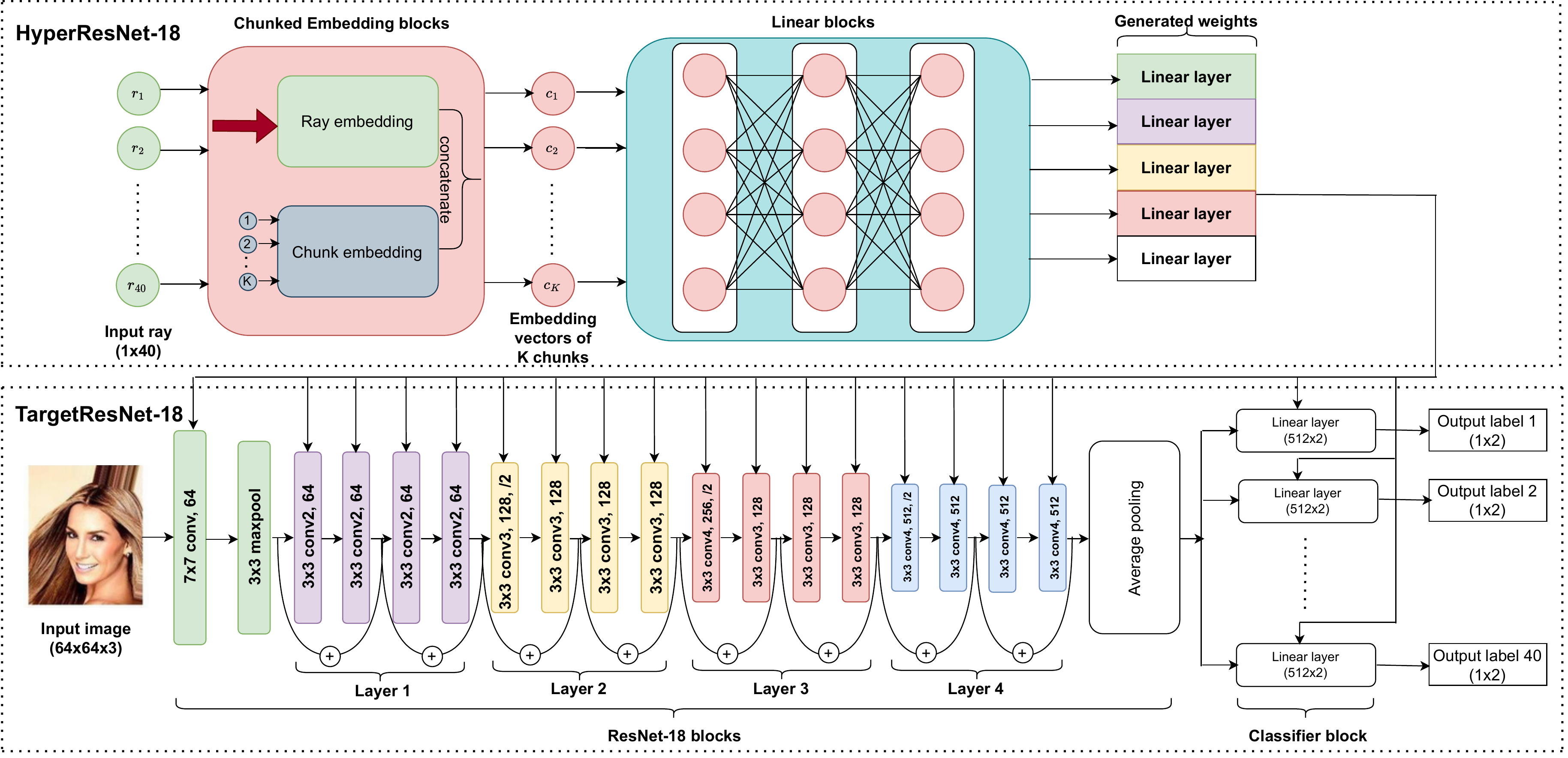}
    \caption{Multi-ResNet18 architecture.}
    \label{fig12}
\end{figure*}
\begin{figure*}[ht]
     \centering
     \includegraphics[width=0.9\textwidth]{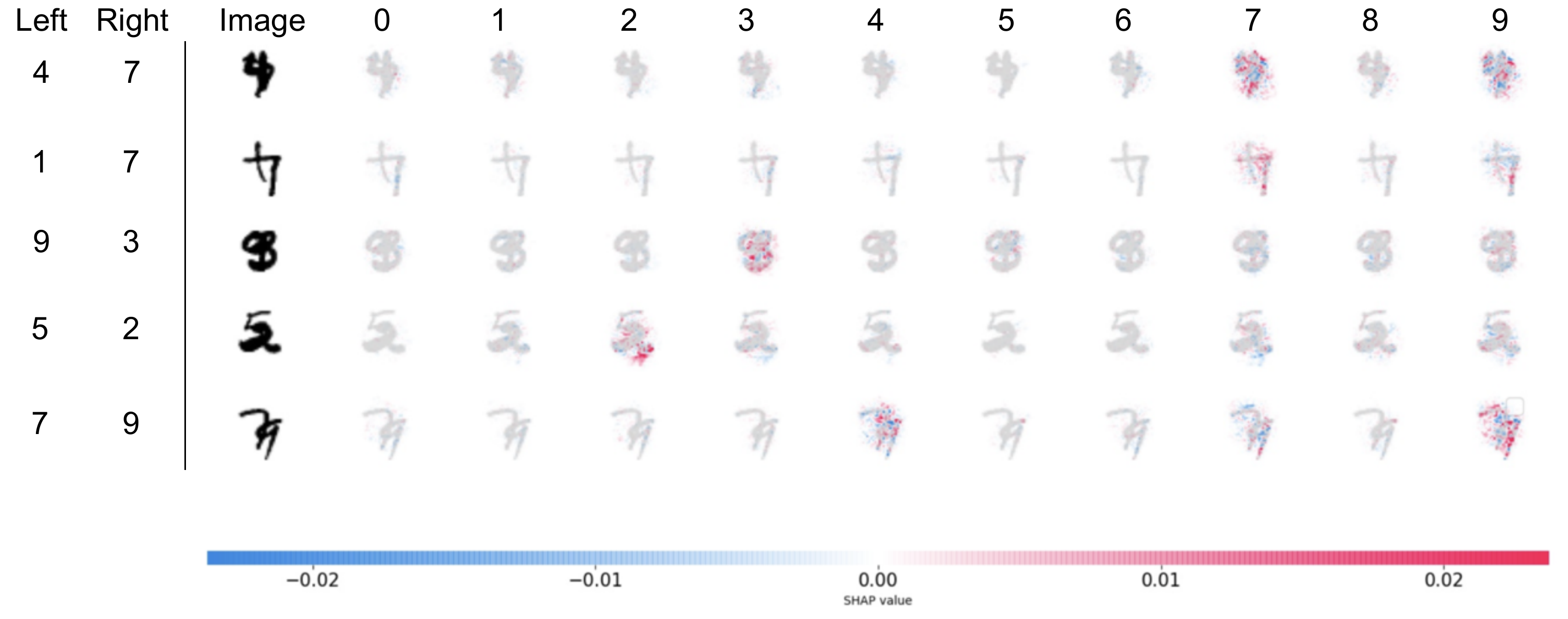}
     \caption{When the prediction of a class is higher, the pixels will be red, and when it is lower, the pixels will be blue. To demonstrate model' gradient explanation for the right digit, we choose $ray: (0.01, 0.99)$, whereby most red pixels concentrate on the right digit.}
     \label{right}
\end{figure*}
\begin{figure*}[ht]
     \centering
     \includegraphics[width=0.9\textwidth]{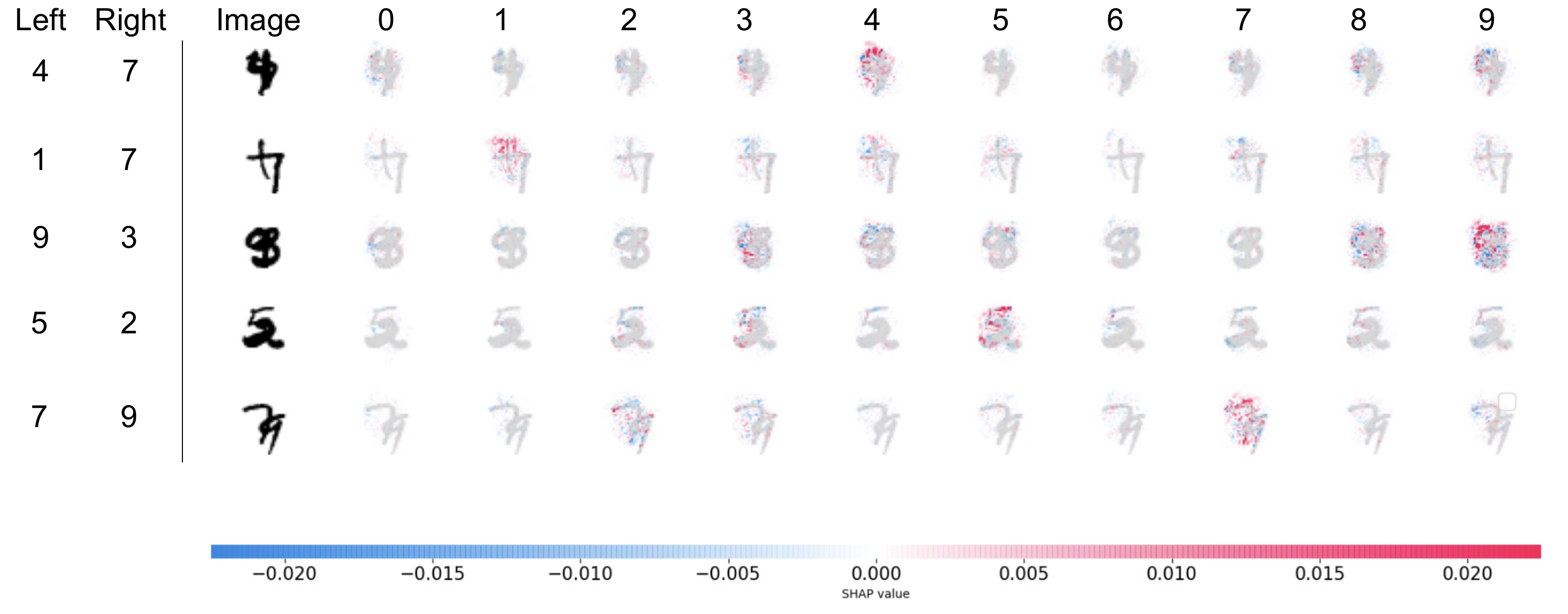}
     \caption{To the left digit, we choose $ray: (0.99, 0.01)$, then the gradient explanation of model concentrated on the left digit with the majority of the red pixels.}
     \label{left}
\end{figure*}
\begin{figure*}[ht]
     \centering
     \includegraphics[width=0.8\textwidth]{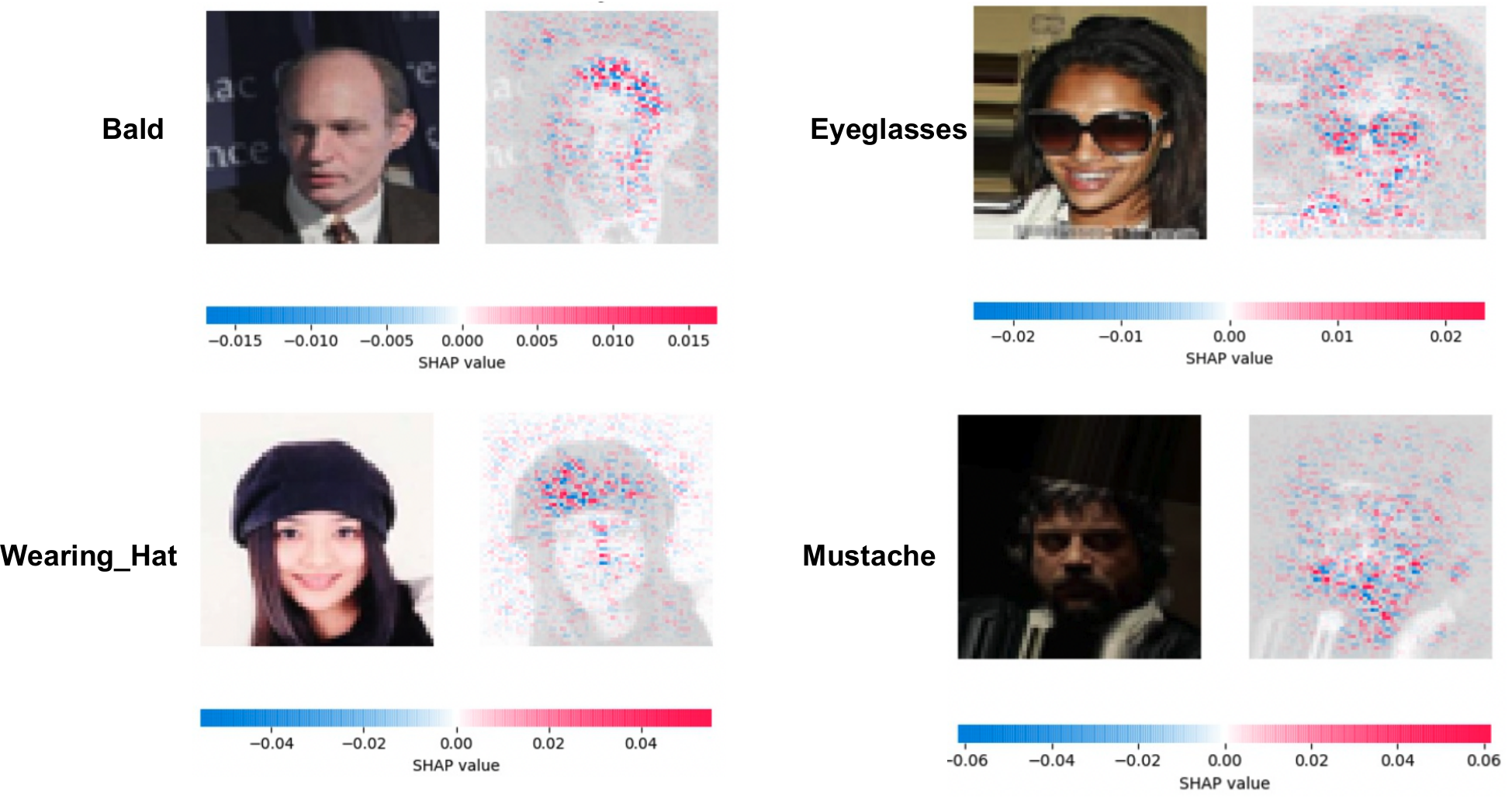}
     \caption{Higher distribution of the red points, used to estimate the model's behavior, concentrated on the attributes in 40 attributes of the CelebA dataset.}
     \label{gradient}
\end{figure*}
\begin{figure*}[ht]
    \centering
    \includegraphics[width=1.0\textwidth]{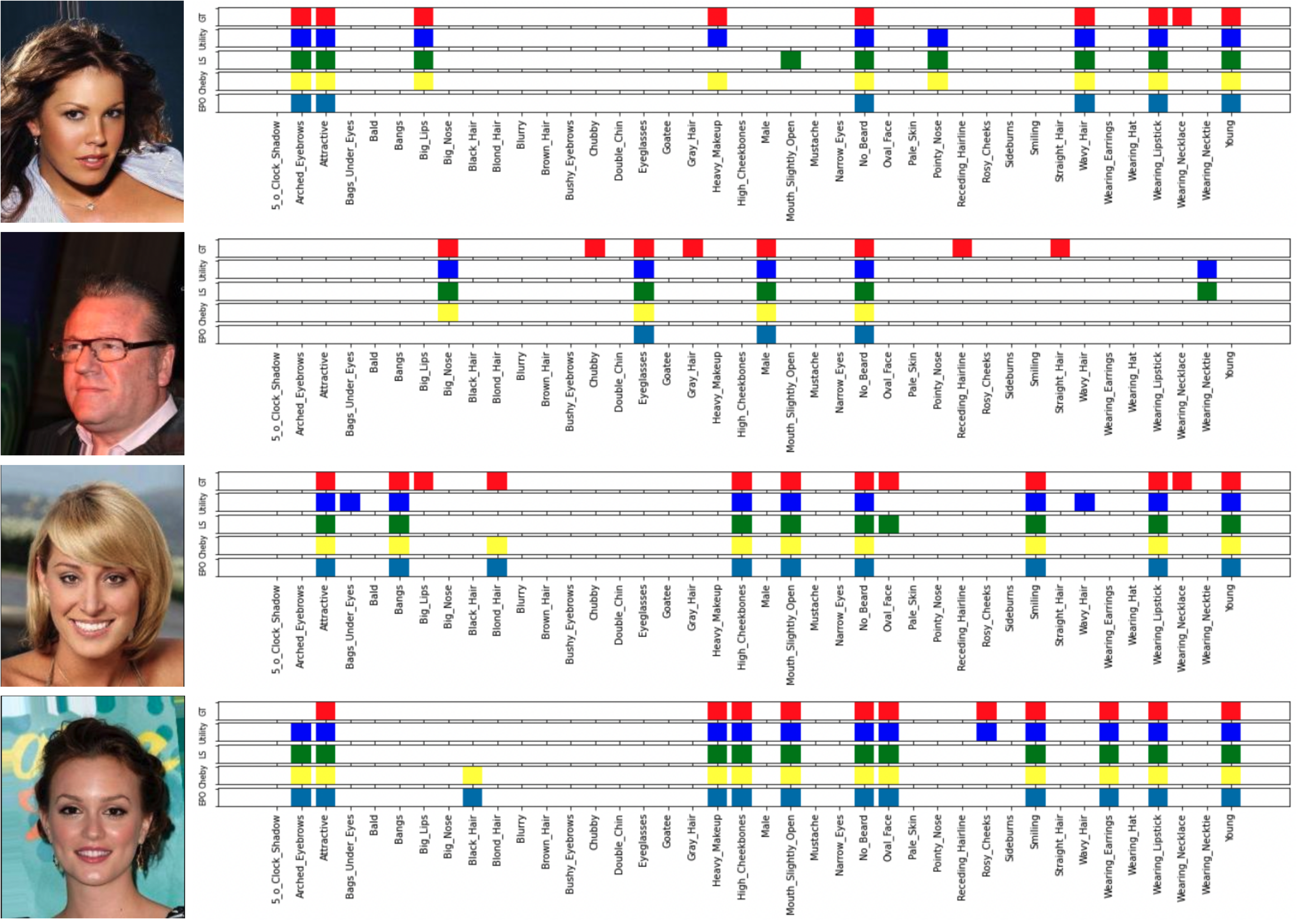}
    \caption{Prediction results of PHN-LS, PHN-Utility, PHN-Cheby, and PHN-EPO on the CelebA test dataset.}
    \label{fig14}
\end{figure*}
\newpage
\bibliographystyle{cas-model2-names}
\bibliography{cas-refs}

\end{document}